\documentclass[12pt]{article}
\usepackage[T1]{fontenc}
\usepackage[utf8]{inputenc}
\usepackage{lmodern}
\usepackage[letterpaper,margin=1.25in]{geometry}
\usepackage[nopatch=footnote]{microtype}
\usepackage{amsmath,amssymb,amsthm,mathtools,bm,mathrsfs}
\usepackage[round,authoryear]{natbib}
\usepackage{enumitem}
\usepackage{booktabs}
\usepackage{graphicx}
\usepackage{longtable}
\usepackage{pdflscape}
\usepackage{array}
\usepackage{chngcntr}
\usepackage{setspace}
\usepackage{xurl}
\usepackage{xr-hyper}
\usepackage[hidelinks]{hyperref}
\allowdisplaybreaks[2]
\newtheorem{theorem}{Theorem}[section]
\newtheorem{proposition}[theorem]{Proposition}
\newtheorem{lemma}[theorem]{Lemma}
\newtheorem{corollary}[theorem]{Corollary}
\newtheorem{assumption}[theorem]{Assumption}
\theoremstyle{definition}
\newtheorem{remark}[theorem]{Remark}
\newtheorem{example}[theorem]{Example}
\newcommand{\RR}{\mathbb R}
\newcommand{\PP}{\mathbb P}
\newcommand{\EE}{\mathbb E}
\newcommand{\1}{\mathbf 1}
\newcommand{\Cov}{\operatorname{Cov}}
\newcommand{\Var}{\operatorname{Var}}
\newcommand{\tr}{\operatorname{tr}}

\newcommand{\HS}{\mathrm{HS}}
\newcommand{\op}{\mathrm{op}}
\newcommand{\CF}{\mathrm{CF}}
\newcommand{\suppref}[1]{Supplementary Appendix~\ref{#1}}

\newcommand{\mcse}[1]{\,(#1)}

\hypersetup{pdftitle={Functional Inference for the Mean Curve of Intraday Spot Correlation},pdfauthor={}}

\newcommand{\papertitle}{Functional Inference for the Mean Curve of Intraday Spot Correlation}
\title{\papertitle}
\author{}
\date{}

\title{
  Functional Inference for the Mean Curve of Intraday Spot Correlation
}

\author{
   Aodong Guo
  \thanks{
    Faculty of Science, University of Macau.
    E-mail:
    \href{mailto:yc47928@um.edu.mo}
         {yc47928@um.edu.mo}.
  }
  \quad
  Yu Jiang
  \thanks{
    School of Economics, Renmin University of China.
    E-mail:
    \href{mailto:yujiang1101@ruc.edu.cn}
         {yujiang1101@ruc.edu.cn}.
  }
  \quad
  Zhi Liu
  \thanks{
    Faculty of Science, University of Macau.
    E-mail:
    \href{mailto:liuzhi@um.edu.mo}
         {liuzhi@um.edu.mo}.
  }
}
\date{October 22, 2026}

\begin{document}
\maketitle
\begin{abstract}
In this paper, we propose estimators for the mean of the daily spot-correlation
curves, which are robust to the presence of finite-variation jumps in the price
processes. The related feasible functional central limit theorems for the proposed
estimators have been established under a joint infill and long-span setting. Based
on the asymptotic theory, we develop a test procedure for flatness of the curves.
Monte Carlo simulations confirm the strong finite-sample performance of the
estimators. In an empirical application to 30 U.S. stocks from 2010 to 2019,
we find a pronounced end-of-day surge in the mean of correlation curves across
diverse stock groupings, alongside substantial calendar-time non-flatness in the
curves.
\end{abstract}

\bigskip

\section{Introduction}
\label{sec:introduction}

Correlation is a central input to portfolio allocation, risk measurement,
hedging, and the assessment of systematic exposure
\citep{Markowitz1952}. At the intraday horizon, the strength of market
comovement can vary over the trading session, and the sensitivity of measured
correlation to the observation interval has long been recognized
\citep{Epps1979}. A single full-day coefficient can therefore conceal when
assets move together most strongly. Moreover, an intraday covariance pattern
need not reflect variation in dependence alone, because it may be generated
by marginal volatility, correlation, or both. Studying intraday dependence
on its original scale requires recovery and inference for the latent
spot-correlation curve itself.

A major strand of the correlation literature models time-varying conditional
correlations. The dynamic conditional correlation model of
\citet{Engle2002} and its asymmetric extension in
\citet{CappielloEngleSheppard2006}, for example, describe the evolution of
conditional correlations in discrete time. Alongside this literature,
high-frequency econometrics made it
possible to study dependence at finer scales. \citet{BarndorffNielsenShephard2004}
develop distribution theory for realized covariance, regression, and
correlation; \citet{AitSahaliaFanXiu2010} treat high-frequency covariance
estimation under market microstructure noise; and
\citet{BibingerHautschMalecReiss2019} estimate spot covariation and correlation
under noisy and nonsynchronous observations. These contributions provide the
measurement foundations for analyzing the evolution of dependence within the
trading day.

A parallel literature studies recurrent time-of-day variation in volatility.
The evidence and models in \citet{AndersenBollerslev1997} and
\citet{AndersenBollerslev1998} documented pronounced intraday volatility
periodicity. \citet{ChristensenHounyoPodolskij2018} develop a nonparametric
assessment of whether a deterministic diurnal component is sufficient to
explain intraday volatility variation, while
\citet{AndersenThyrsgaardTodorov2019} allow the volatility curve itself to
vary over time. More recent work develops long-span functional inference for
latent intraday volatility curves \citep{AndersenSuTodorovZhang2024} and tests
the stability of their mean across calendar time
\citep{AndersenTanTodorovZhang2025}.

The closest correlation-specific contributions are
\citet{ChristensenHounyoLiu2026} and
\citet{AndersenTanTanTodorovZhang2026}. In their baseline model,
\citet{ChristensenHounyoLiu2026} represent spot correlation as a deterministic periodic
factor multiplied by a stationary stochastic component and test whether
the periodic factor is constant. Their Section~5.1 accommodates stochastic
diurnal correlation through a population ratio formed from the mean spot
covariance coordinates. \citet{AndersenTanTanTodorovZhang2026} retain a multiplicative
specification but normalize the correlation pattern by its $L^2$ norm,
rather than identifying the deterministic covariance component through a
signed-integral restriction. They establish functional limit theory in
$D[0,1]$ under the supremum norm and compare normalized patterns across
two prespecified periods. They point out that the covariance-to-correlation
map is not Hadamard differentiable on its natural $L^2$ domain.

This paper studies the mean of the latent daily spot-correlation curves.
Let $\rho_i$ denote the curve on trading session $i$, indexed by
intraday time $\kappa\in[0,1]$ as specified in
Section~\ref{subsec:model_semimartingale}. Our primary target is
\begin{equation}
r_u:=\EE_{\mathcal H}[\rho_i]\in\mathcal H,
\qquad \mathcal H=L^2([0,1]).
\label{eq:intro_equal_day_target}
\end{equation}
Every trading day enters with the same weight, and $r_u$ retains the
original correlation scale. We develop inference for this mean without
requiring a multiplicative decomposition of the correlation process.
When $\bar r_u:=\int_0^1r_u(s)\,ds$ is nonzero, we also study its
normalized shape $k_u^\circ=r_u/\bar r_u$ and the flatness hypothesis
$k_u^\circ=\1$ in $\mathcal H$.

We first construct the estimator of $r_{u}$.
Within each trading session, we estimate a positive-semidefinite covariance
matrix from locally truncated returns, form its correlation ratio on each
block, and assemble a bounded daily curve $\widehat\rho_i$.
The sample mean $\widetilde\rho_u=T^{-1}\sum_{i=1}^T\widehat\rho_i$
then estimates $r_u$ directly. Averaging the covariance coordinates before
forming a correlation ratio generally yields a different population
quantity; Supplementary Appendix~\ref{app:population} gives the
relationship and examples.

Under joint infill and long-span asymptotics, with finite-activity price
jumps and stationary strongly mixing daily states, we prove consistency and
the functional central limit theorem
\begin{equation}
\sqrt T\,(\widetilde\rho_u-r_u)
=\frac{1}{\sqrt T}\sum_{i=1}^T\{\rho_i-r_u\}+o_P(1)
\ \Rightarrow\ G_\rho
\qquad\text{in }\mathcal H,
\label{eq:intro_raw_fclt}
\end{equation}
where $G_\rho$ is a centered Gaussian element with the long-run covariance
of the latent daily curves. The expansion establishes that high-frequency
recovery leaves the first-order distribution of the mean unchanged. Its
proof combines finite-dimensional correlation expansions on nondegenerate
blocks with bounded-ratio control on the complement, and handles the
accumulated nonlinear remainder at the $\sqrt T$ scale. The argument allows
finite-activity price jumps without moment restrictions on their sizes.
Once this approximation is established, a Hilbert-space central
limit theorem applies to the bounded latent daily curves. A separate
functional delta method for $f\mapsto f/\int_0^1f(s)\,ds$ yields inference
for the normalized mean shape.

For feasible inference, we construct positive-semidefinite lag-window
estimators of the long-run covariance operators and prove convergence in
trace norm. This implies $\ell^1$ convergence of the ordered eigenvalues,
so whole-curve confidence sets can be calibrated using all eigenvalues of
the estimated covariance operator, without eigenvalue-separation
conditions or a prespecified principal-component truncation. The same
framework delivers inference for fixed linear summaries, normalized mean
shapes, and intraday constancy. It also accommodates fixed weighted
averages of a finite collection of dependent asset-pair curves while
retaining all cross-covariance terms among the component curves.

The simulations examine curve recovery, confidence-set calibration, and
flatness detection. As the number of trading days and the intraday sampling frequency increase jointly, errors in recovering the latent correlation curves decrease. The test also gains power as departures from intraday flatness become more pronounced. Coverage results distinguish the original
mean scale from normalized shape inference, and a paired comparison with
\citet{ChristensenHounyoLiu2026} provides a reference for finite-sample test performance.
In an application to 2,370 trading days for 30 U.S. stocks, the closing-hour
average exceeds the opening-hour average for both the estimated stock--proxy
and average-pairwise mean curves. Tests in all 38 complete quarterly windows
reject flatness for both aggregates after Holm adjustment.
The test of \citet{ChristensenHounyoLiu2026} yields the same quarterly decisions
except for average-pairwise correlation in one quarter, yet the two approaches
give different estimates of correlation levels and of the increase in average
correlation from the opening hour to the closing hour.

Section~\ref{sec:model} introduces the model and assumptions.
Section~\ref{sec:estimation} constructs the estimator and establishes
consistency and the functional central limit theorems.
Section~\ref{sec:inference} develops feasible inference.
Sections~\ref{sec:simulations} and \ref{sec:empirical} report the numerical
and empirical evidence; all proofs and additional results are collected in
the Supplementary Appendix.

\section{Setup}
\label{sec:model}

\subsection{Model}
\label{subsec:model_semimartingale}

All processes are defined on a common filtered probability space
\((\Omega,\mathcal F,(\mathcal F_t)_{t\ge0},\PP)\). Let
\(Z_t=(X_t,Y_t)^\top\) be a bivariate log-price process. We measure time in
trading-session units and write
\begin{equation}
Z_t
=
Z_0+\int_0^t a_s\,ds+M_t+J_t,
\qquad t\ge0,
\label{eq:semimartingale_jump_model}
\end{equation}
where \(a\) is predictable, \(M=(M^X,M^Y)^\top\) is a continuous local
martingale, and \(J=(J^X,J^Y)^\top\) is a finite-activity pure-jump process.
The jump count \(N_t^J:=\sum_{0<s\le t}\1\{\Delta J_s\ne0\}\) is therefore
almost surely finite on compact time intervals. This is the finite-activity
case of the high-frequency semimartingale framework
\citep{JacodProtter2012}.

Let \(c_t\) denote the predictable spot covariance matrix of \(M\), with
\(d\langle M\rangle_t=c_t\,dt\). Its diagonal entries \(c_t^X\) and
\(c_t^Y\) are the marginal spot variances, and \(\rho_t\) is the spot
correlation of the continuous martingale components:
\begin{equation}
c_t
=
\begin{pmatrix}
c_t^X & c_t^{XY}\\
c_t^{XY} & c_t^Y
\end{pmatrix},
\qquad
c_t^{XY}
=
\rho_t\sqrt{c_t^Xc_t^Y},
\qquad
\rho_t
=
\frac{c_t^{XY}}{\sqrt{c_t^Xc_t^Y}}
\quad\text{when }c_t^Xc_t^Y>0.
\label{eq:spot_correlation_definition}
\end{equation}
\citet{BarndorffNielsenShephard2004} study realized correlation;
\citet{BibingerHautschMalecReiss2019} consider spot covariation and
correlation estimation under more general observation schemes.
The index \(\wp\) ranges over \(\{X,Y,XY\}\) for covariance coordinates
and over \(\{X,Y\}\) for price, volatility, jump, and threshold quantities.

Let $g_{\mathrm{nt}}\ge0$ be the fixed length of the non-trading interval
between consecutive sessions, and set
\[
\underline t_i=(i-1)(1+g_{\mathrm{nt}}),\qquad i\ge1.
\]
The $i$th observed trading session is
$[\underline t_i,\underline t_i+1]$; observations from the intervening
non-trading intervals are not used. For $\kappa\in[0,1]$, define
\begin{equation}
\rho_i(\kappa):=\rho_{\underline t_i+\kappa},
\qquad
c_i^\wp(\kappa):=c_{\underline t_i+\kappa}^\wp,
\quad
\wp\in\{X,Y,XY\}.
\label{eq:daily_rho_curve}
\end{equation}
The case $g_{\mathrm{nt}}=0$ gives contiguous sessions and is used in the
simulation design below.

We view the daily curves as random elements of
\(\mathcal H=L^2([0,1])\), with inner product
\(\langle f,g\rangle_{\mathcal H}=\int_0^1f(\kappa)g(\kappa)\,d\kappa\)
and norm \(\|\cdot\|_{\mathcal H}\). Our target is the mean of the daily
correlation curves, with every day given the same weight. Under
Assumption~\ref{ass:C5_daily_mixing}, define the mean curve and the
centered daily curves by
\begin{equation}
\begin{aligned}
r_u&:=\EE_{\mathcal H}[\rho_i],
&
A_i^\rho&:=\rho_i-r_u.
\end{aligned}
\label{eq:ru_target}
\end{equation}
Since \(|\rho_i|\le1\), the mean is well defined in \(\mathcal H\), and
\(r_u(\kappa)=\EE[\rho_i(\kappa)]\) for almost every \(\kappa\in[0,1]\).
Equalities between curves are understood in \(\mathcal H\).
When $g_{\mathrm{nt}}>0$, the mean correlations at the open and close need
not be equal.

We also consider the integrated mean and, when it is nonzero, the
normalized shape:
\begin{equation}
\bar r_u:=\int_0^1r_u(\kappa)\,d\kappa,
\qquad
k_u^\circ:=\frac{r_u}{\bar r_u}.
\label{eq:normalized_target}
\end{equation}
Normalization is applied to the mean curve, not to the daily curves before
averaging. The mean curve itself remains well defined when it changes sign
or has zero integral.

\subsection{Assumptions}
\label{subsec:theory_assumptions}

\begin{assumption}
\label{ass:J_price_jumps}
For each \(\wp\in\{X,Y\}\),
\[
J_t^\wp=\int_0^t\int_{\RR}x\,\mu^\wp(ds,dx),
\qquad
\nu^\wp(dt,dx)=\chi_t^\wp dt\otimes F^\wp(dx),
\]
where \(\mu^\wp\) is an integer-valued random measure with compensator
\(\nu^\wp\), and \(\chi^\wp\) is nonnegative, predictable, and c\`adl\`ag.
\end{assumption}

No moment condition is imposed on price-jump sizes.  Finite activity follows
from \(F^\wp(\RR)<\infty\) and the intensity regularity in
Assumption~\ref{ass:C4_moments}.

\begin{assumption}
\label{ass:V_volatility}
For each \(\wp\in\{X,Y\}\), the marginal volatility
\(\sigma_t^\wp=\sqrt{c_t^\wp}\) is an adapted c\`adl\`ag It\^o
semimartingale,
\[
\sigma_t^\wp
=
\sigma_0^\wp
+\int_0^t\widetilde a_s^\wp\,ds
+\int_0^t\widetilde\sigma_s^\wp\,dW_s
+\int_0^t\widetilde v_s^\wp\,d\widetilde W_s
+\int_0^t\int_{\RR}x\,\widetilde\mu^\wp(ds,dx).
\]
Here \(W\) drives a representation of the continuous price martingale,
\(\widetilde W\) is independent of \(W\), and
\(\widetilde\mu^\wp\) is the integer-valued jump measure of
\(\sigma^\wp\), with compensator
\[
\widetilde\nu^\wp(dt,dx)
=
\widetilde\chi_t^\wp dt\otimes\widetilde F^\wp(dx).
\]
The intensity \(\widetilde\chi^\wp\) is nonnegative and predictable.
We use the convention
\(\widetilde\mu^\wp(\{t\}\times\RR)\le1\), with mark
\(x=\Delta\sigma_t^\wp\) at a volatility-jump time.  Brownian motions and
coefficient loadings may be vector-valued with conformable dimensions.
\end{assumption}

\begin{assumption}
\label{ass:C4_moments}
For each \(\wp\in\{X,Y\}\),
\[
\sup_{t\ge0}\EE\!\left[
e^{|a_t^\wp|}+e^{|\sigma_t^\wp|}+e^{|\chi_t^\wp|}
\right]<\infty,
\qquad
F^\wp(\RR)<\infty.
\]
Moreover,
\[
\widetilde F^\wp(\RR)<\infty,
\qquad
\int_{\RR}|x|^8\widetilde F^\wp(dx)<\infty,
\]
and, for some common \(p_\star>8\),
\[
\sup_{t\ge0}\EE\!\left[
|\widetilde a_t^\wp|^{p_\star/2}
+|\widetilde\sigma_t^\wp|^{p_\star}
+|\widetilde v_t^\wp|^{p_\star}
+|\widetilde\chi_t^\wp|^{p_\star/2}
\right]<\infty.
\]
\end{assumption}

The exponential bounds imply moment bounds of every finite order for
$a^\wp,\sigma^\wp,\chi^\wp$, uniformly in time. The higher moment order for
$\widetilde\sigma^\wp$ and $\widetilde v^\wp$ accounts for the squared
terms in It\^o's formula for the covariance coordinates. The condition
$p_\star>8$ permits H\"older bounds for products of states and coefficients
without an independence assumption.

\begin{assumption}
\label{ass:primitive_correlation_regularity}
The spot correlation is adapted and c\`adl\`ag, satisfies
\(|\rho_t|\le1\) almost surely, and admits the finite-activity It\^o
semimartingale representation
\begin{equation}
\rho_t
=
\rho_0
+\int_0^t b_s^\rho\,ds
+\int_0^t\Sigma_s^\rho\,dB_s^\rho
+\int_0^t\int_E\delta^\rho(s,x)\,\mu^\rho(ds,dx).
\label{eq:primitive_correlation_raw}
\end{equation}
Here \(B^\rho\) is a Brownian motion, \((E,\mathcal E)\) is a measurable mark
space, and \(\mu^\rho\) is an integer-valued random measure on
\(\RR_+\times E\) with compensator
\[
\nu^\rho(ds,dx)=\chi_s^\rho ds\otimes F^\rho(dx),
\qquad
F^\rho(E)<\infty.
\]
The intensity \(\chi^\rho\) is nonnegative and predictable.
The jump coefficient \(\delta^\rho\) is
\(\mathcal P\otimes\mathcal E\)-measurable, where \(\mathcal P\) is the
predictable \(\sigma\)-field of the underlying filtration.
We further require
\(\rho_{s-}+\delta^\rho(s,x)\in[-1,1]\) for
\(\nu^\rho\)-almost every \((s,x)\).  For some
\(p_\star^\rho>8\),
\[
\sup_{s\ge0}\EE\!\left[
|b_s^\rho|^{p_\star^\rho/2}
+\|\Sigma_s^\rho\|^{p_\star^\rho}
+|\chi_s^\rho|^{p_\star^\rho/2}
\right]<\infty.
\]
\end{assumption}

When Assumptions~\ref{ass:C4_moments} and
\ref{ass:primitive_correlation_regularity} are used together, we write
\(p_\star\) for the smaller of the two moment exponents, which is still
greater than \(8\).
All jump integrals use the original measures, not their compensated
versions. For each coordinate, there is at most one mark at a given time,
but different coordinates may jump simultaneously. The local bounds for
the high-frequency approximation are given in Supplementary
Appendix~\ref{app:regularity}.

\begin{assumption}
\label{ass:C5_daily_mixing}
For $i\ge1$ and $\kappa\in[0,1]$, the vector
\[
\bigl(c_{\underline t_i+\kappa}^X,
      c_{\underline t_i+\kappa}^Y,
      \rho_{\underline t_i+\kappa}\bigr)
\]
is a measurable function of the multivariate state
$\mathsf Z_{\underline t_i+\kappa}$. The mapping may depend on
$\kappa$, but not on the trading-session index $i$.
The state process is stationary and
\(\alpha\)-mixing, with
\[
\begin{aligned}
\alpha_{\mathsf Z}(s)
&:=
\sup_{t\ge0}\sup\{
|\PP(A\cap B)-\PP(A)\PP(B)|:
A\in\mathcal G_t,\ B\in\mathcal G^{t+s}\},
\\
\mathcal G_t&:=\sigma(\mathsf Z_u:u\le t),
\qquad
\mathcal G^{t+s}:=\sigma(\mathsf Z_u:u\ge t+s),
\end{aligned}
\]
and \(\alpha_{\mathsf Z}(s)=O(s^{-q-\iota})\) for some \(q>0\) and
\(\iota>0\).
\end{assumption}

These conditions make the daily curves stationary and ergodic across
sessions, while allowing their mean to vary with $\kappa$. Together with
high-frequency recovery, they yield consistency. The functional central
limit theorem and feasible inference use \(q\ge3\). For related results
on dependent functional time series, see
\citet{HormannKokoszka2010,Merlevede2003HilbertMixingCLT}.

\begin{assumption}
\label{ass:positivity_normalization}
There is a constant \(\underline c>0\) such that, almost surely,
\[
c_t^X\ge\underline c,
\qquad
c_t^Y\ge\underline c,
\qquad t\ge0.
\]
For normalized objects only, we additionally require
\(\bar r_u\ne0\).
\end{assumption}

The variance lower bound keeps the population correlation denominator
away from zero. The condition \(\bar r_u\ne0\) is needed only for
normalization; the consistency and functional central limit theorem for
the unnormalized mean do not require it.

\section{Estimation and functional limit theory}
\label{sec:estimation}

We estimate $r_u$ by forming local correlation ratios within each day and
averaging the resulting daily curves. The results below allow dependence
within and across trading days.

\subsection{Estimation of daily spot-correlation curves}
\label{subsec:daily_curve_estimator}
\label{subsec:truncated_covariance}
\label{subsec:direct_estimator}

During each trading session $i=1,\ldots,T$, the two log prices are observed
without noise at the same $n+1$ equally spaced times in
$[\underline t_i,\underline t_i+1]$. With $\Delta_n=1/n$, write
\[
\begin{gathered}
t_{i,\ell}=\underline t_i+\ell\Delta_n,
\qquad \ell=0,\ldots,n,\\
\Delta_{i,\ell}Z=Z_{t_{i,\ell}}-Z_{t_{i,\ell-1}},
\qquad \ell=1,\ldots,n.
\end{gathered}
\]
The same increment notation is used for \(X\) and \(Y\). Partition each
session into blocks of \(k_n\) increments. For simplicity, suppose that
\(k_n\) divides \(n\), and set
\[
\begin{gathered}
m_n:=\frac{n}{k_n},
\qquad
L_n:=\frac{k_n}{n}=\frac1{m_n},\\
\mathcal I_j^n
:=
\{(j-1)k_n+1,\ldots,jk_n\},
\qquad
\kappa_j:=\frac{j-1}{m_n},
\qquad
j=1,\ldots,m_n.
\end{gathered}
\]
Thus block \(j\) covers \([\kappa_j,\kappa_j+L_n)\).

To accommodate price jumps, we truncate each return series using thresholds
based on its blockwise bipower variation
\citep{BarndorffNielsenShephard2004Bipower,Mancini2009,ChristensenHounyoLiu2026}.
For \(\wp\in\{X,Y\}\), fix
\(\tau_\wp>0\) and \(0<\varpi<1/2\). The blockwise bipower estimator and the
corresponding threshold are
\begin{subequations}
\begin{align}
\widehat b_{i,j,n}^{\wp}
&=
\frac{\pi}{2(k_n-1)}
\sum_{\substack{\ell\in\mathcal I_j^n\\
\ell>(j-1)k_n+1}}
\bigl|\sqrt n\,\Delta_{i,\ell-1}\wp\bigr|
\bigl|\sqrt n\,\Delta_{i,\ell}\wp\bigr|,
\qquad k_n\ge2,
\label{eq:localized_bipower_pilot}\\
\widehat u_{i,j,n}^{\wp}
&=
\tau_\wp\sqrt{\widehat b_{i,j,n}^{\wp}}\,n^{-\varpi}.
\label{eq:adaptive_localized_threshold}
\end{align}
\end{subequations}
The threshold scales with the estimated local volatility, allowing it to
vary across assets, days, and intraday blocks.

For \(\ell\in\mathcal I_j^n\), let
\[
\begin{gathered}
I_{i,\ell}^{X,n}
=
\1\{|\Delta_{i,\ell}X|\le \widehat u_{i,j,n}^X\},
\qquad
I_{i,\ell}^{Y,n}
=
\1\{|\Delta_{i,\ell}Y|\le \widehat u_{i,j,n}^Y\},\\
\Delta_{i,\ell}Z^{\mathrm{tr}}
:=
\bigl(
\Delta_{i,\ell}X I_{i,\ell}^{X,n},
\Delta_{i,\ell}Y I_{i,\ell}^{Y,n}
\bigr)^\top.
\end{gathered}
\]
The truncated local covariance matrix and its ordered coordinate vector are
\begin{equation}
\begin{aligned}
\widehat c_{i,\kappa_j}
&:=
\frac{n}{k_n}
\sum_{\ell\in\mathcal I_j^n}
\Delta_{i,\ell}Z^{\mathrm{tr}}
(\Delta_{i,\ell}Z^{\mathrm{tr}})^\top\\
&=
\begin{pmatrix}
\widehat c_{i,\kappa_j}^{X}
&
\widehat c_{i,\kappa_j}^{XY}
\\
\widehat c_{i,\kappa_j}^{XY}
&
\widehat c_{i,\kappa_j}^{Y}
\end{pmatrix},
\qquad
\widehat C_{i,\kappa_j}
:=
\bigl(
\widehat c_{i,\kappa_j}^{X},
\widehat c_{i,\kappa_j}^{Y},
\widehat c_{i,\kappa_j}^{XY}
\bigr)^\top.
\end{aligned}
\label{eq:truncated_cov_matrix}
\end{equation}
Define the local correlation estimate by
\begin{equation}
\widehat\rho_{i,\kappa_j}
=
\begin{cases}
\displaystyle
\frac{
\widehat c_{i,\kappa_j}^{XY}
}{
\sqrt{
\widehat c_{i,\kappa_j}^{X}
\widehat c_{i,\kappa_j}^{Y}
}
},
&
\widehat c_{i,\kappa_j}^{X}>0,\ 
\widehat c_{i,\kappa_j}^{Y}>0, \\[1.1em]
0,
&
\text{otherwise}.
\end{cases}
\label{eq:direct_local_ratio}
\end{equation}
The outer-product form of \(\widehat c_{i,\kappa_j}\) ensures positive
semidefiniteness and hence \(\lvert\widehat\rho_{i,\kappa_j}\rvert\le1\) for
every sample. Keeping each ratio constant within its block gives the daily
estimate
\[
\widehat\rho_i(\kappa)
=
\sum_{j=1}^{m_n}
\widehat\rho_{i,\kappa_j}
\1_{[\kappa_j,\kappa_j+L_n)}(\kappa),
\qquad
\widehat\rho_i\in\mathcal H,
\qquad
\|\widehat\rho_i\|_\infty\le1.
\]
\subsection{Estimation of the mean curve and consistency}
\label{subsec:equal_day_target}
\label{subsec:theory_consistency}

The estimator of the mean curve \(r_u\) is
\begin{equation}
\widetilde\rho_u(\kappa)
:=
\frac1T\sum_{i=1}^T
\widehat\rho_i(\kappa).
\label{eq:direct_average_curve}
\end{equation}

To estimate the normalized shape, use a deterministic cutoff
\(\varepsilon_T\downarrow0\) and set
\begin{equation}
\widehat k_u^\circ(\kappa)
:=
\begin{cases}
\displaystyle
\dfrac{\widetilde\rho_u(\kappa)}
{\int_0^1\widetilde\rho_u(s)\,ds},
&
\left|\int_0^1\widetilde\rho_u(s)\,ds\right|>\varepsilon_T,
\\[1.1em]
0,
&
\left|\int_0^1\widetilde\rho_u(s)\,ds\right|\leq\varepsilon_T.
\end{cases}
\label{eq:direct_normalized_estimator}
\end{equation}
Because the blocks have equal length,
\(\int_0^1\widetilde\rho_u(s)\,ds
=m_n^{-1}\sum_{j=1}^{m_n}\widetilde\rho_u(\kappa_j)\) exactly.

\begin{assumption}
\label{ass:rates_finite_jump}
\leavevmode
\begin{enumerate}
\item[(i)] The truncation thresholds are those in
\eqref{eq:adaptive_localized_threshold}, where \(\tau_\wp>0\) for
\(\wp\in\{X,Y\}\) and \(0<\varpi<1/2\). The number of returns per block and
the number of trading days satisfy
\begin{equation}
k_n\to\infty,
\qquad
\frac{k_n}{n}\to0,
\qquad
T\to\infty.
\label{eq:block_rate_basic}
\end{equation}
\item[(ii)] In addition, for the functional limits, let
\[
T\asymp n^c,
\qquad
k_n\asymp n^d,
\qquad
0<c<1,
\qquad
\frac c2<d<
1-\frac c2.
\]
\end{enumerate}
\end{assumption}

Part~(i) lets each block contain more returns while covering a shorter
interval. Together with the process assumptions in Section~\ref{sec:model},
it is sufficient for consistency.

For the functional limits, the high-frequency estimation error in the mean
must also vanish after multiplication by $\sqrt T$. The lower bound on $d$ gives
$\sqrt T/k_n\to0$ and controls finite-block bias in the local correlation
estimates. The upper bound gives $\sqrt T L_n\to0$ and controls the terms
from within-block approximation. The condition $c<1$ also gives $T/n\to0$.
Supplementary Appendix~\ref{app:mean_approximation} gives the full error
bound. These rates are asymptotic restrictions, not a finite-sample rule
for choosing $k_n$.

\begin{proposition}
\label{prop:daily_recovery}
Suppose Assumptions~\ref{ass:J_price_jumps}--\ref{ass:primitive_correlation_regularity},
the variance lower bound in Assumption~\ref{ass:positivity_normalization},
and Assumption~\ref{ass:rates_finite_jump}(i) hold. Then,
for a constant $C$ independent of $i,T,n$,
\begin{equation}
\sup_{i\ge1}\EE\|\widehat\rho_i-\rho_i\|_{\mathcal H}^2
\le C\left(k_n^{-1}+L_n\right).
\label{eq:daily_recovery}
\end{equation}
Consequently,
\[
\frac1T\sum_{i=1}^T\|\widehat\rho_i-\rho_i\|_{\mathcal H}^2
=O_P(k_n^{-1}+L_n).
\]
Under Assumption~\ref{ass:rates_finite_jump}(ii),
the last quantity is $o_P(T^{-1/2})$.
\end{proposition}

The proposition concerns recovery of individual daily curves and uses
neither stationarity nor independence across days. The next theorem combines
this recovery with the ergodic law for the latent daily curves.

\begin{theorem}
\label{thm:finite_jump_consistency}
Suppose Assumptions~\ref{ass:J_price_jumps}--\ref{ass:C5_daily_mixing},
the variance lower bound in Assumption~\ref{ass:positivity_normalization},
and Assumption~\ref{ass:rates_finite_jump}(i) hold. Then
\[
\|\widetilde\rho_u-r_u\|_{\mathcal H}\xrightarrow{P}0.
\]
If, in addition, $\bar r_u\ne0$, then
\[
\|\widehat k_u^\circ-k_u^\circ\|_{\mathcal H}\xrightarrow{P}0.
\]
\end{theorem}

The proof is given in Supplementary Appendix~\ref{app:main_mean_proofs}.

\subsection{Functional central limit theorem}
\label{subsec:mean_clt}

Let $\bar\rho_T=T^{-1}\sum_{i=1}^T\rho_i$ be the sample mean of the unobserved
daily curves. The next theorem establishes the approximation
\begin{equation}
\sqrt T(\widetilde\rho_u-r_u)
=\frac1{\sqrt T}\sum_{i=1}^TA_i^\rho
+\sqrt T(\widetilde\rho_u-\bar\rho_T),
\qquad
\|\sqrt T(\widetilde\rho_u-\bar\rho_T)\|_{\mathcal H}=o_P(1).
\label{eq:mean_expansion}
\end{equation}
Daily mean-square recovery alone does not give the $o_P(1)$ bound.
In the proof, each daily estimation error is centered by its conditional
mean at the session start. The centered errors are orthogonal across
sessions, so Proposition~\ref{prop:daily_recovery} controls their sum.
The remaining conditional means are bounded using a finite-dimensional
correlation expansion and comparison with an auxiliary estimate; see
Supplementary Appendix~\ref{app:mean_approximation}.

For $f,g\in\mathcal H$, write
$(f\otimes g)w=f\langle g,w\rangle_{\mathcal H}$. On a stationary extension
of $(A_i^\rho)$ indexed by $i\in\mathbb Z$, define
\begin{equation}
\Gamma_{\rho,h}:=\EE(A_0^\rho\otimes A_h^\rho),
\qquad
\mathcal C_\rho:=\sum_{h\in\mathbb Z}\Gamma_{\rho,h}.
\label{eq:rho_long_run_operator}
\end{equation}
The conditions in the next theorem imply absolute convergence of this
series in trace norm. Its kernel representation is
\begin{equation}
C_\rho(\kappa,\kappa')
:=\sum_{h\in\mathbb Z}
\Cov\{A_0^\rho(\kappa),A_h^\rho(\kappa')\},
\label{eq:rho_long_run_kernel}
\end{equation}
for Lebesgue-almost every $(\kappa,\kappa')\in[0,1]^2$.
Since $\Gamma_{\rho,-h}=\Gamma_{\rho,h}^*$, $\mathcal C_\rho$ is
self-adjoint.

\begin{theorem}
\label{thm:finite_jump_fclt_unnormalized}
Suppose Assumptions~\ref{ass:J_price_jumps}--\ref{ass:C5_daily_mixing}
hold, with $q\ge3$ in Assumption~\ref{ass:C5_daily_mixing}. Suppose also
that the variance lower bound in
Assumption~\ref{ass:positivity_normalization} and parts~(i)--(ii) of
Assumption~\ref{ass:rates_finite_jump} hold. Then
\eqref{eq:mean_expansion} holds and
\begin{equation}
\sqrt T(\widetilde\rho_u-r_u)\Rightarrow G_\rho
\quad\text{in }\mathcal H,
\label{eq:raw_fclt}
\end{equation}
where $G_\rho$ is a centered Gaussian element with covariance operator
$\mathcal C_\rho$. The operator is nonnegative, self-adjoint, and trace
class, and has the kernel \eqref{eq:rho_long_run_kernel}.
\end{theorem}

Under these joint sampling rates, the estimated and latent sample means
have the same limit after centering at $r_u$ and scaling by $\sqrt T$.
The limiting covariance $\mathcal C_\rho$ includes the dependence of the
daily curves across trading days; high-frequency estimation contributes no
additional first-order term.
Supplementary Appendices~\ref{app:latent_clt} and
\ref{app:main_mean_proofs} give the latent functional limit and the proof
of the theorem.

\subsection{Normalization of the mean curve}
\label{subsec:theory_normalization_delta}

For $f\in\mathcal H$, let
$\ell(f)=\int_0^1f(s)\,ds$ and
$\mathcal D_\Phi=\{f\in\mathcal H:\ell(f)\ne0\}$. Define
\begin{equation}
\Phi:\mathcal D_\Phi\longrightarrow\mathcal H,
\qquad \Phi(f)=\frac{f}{\ell(f)}.
\label{eq:normalization_map}
\end{equation}
This map is Fr\'echet differentiable, with bounded linear derivative
\begin{equation}
D\Phi_f[g]=\frac{g}{\ell(f)}
-\frac{f}{\ell(f)^2}\ell(g).
\label{eq:normalization_derivative_section2}
\end{equation}
The second term in the derivative accounts for estimating the integral in
the denominator. When $\bar r_u\ne0$, let
$\xi_i=D\Phi_{r_u}[A_i^\rho]$ and define
\begin{equation}
\mathcal C_k:=\sum_{h\in\mathbb Z}\EE(\xi_0\otimes\xi_h),
\qquad
C_k(\kappa,\kappa'):=\sum_{h\in\mathbb Z}
\Cov\{\xi_0(\kappa),\xi_h(\kappa')\}.
\label{eq:long_run_covariance_normalized}
\end{equation}
As before, the kernel identity is understood almost everywhere.

\begin{theorem}
\label{thm:finite_jump_fclt_normalized}
Under the assumptions of Theorem~\ref{thm:finite_jump_fclt_unnormalized}
and $\bar r_u\ne0$,
\[
\sqrt T(\widehat k_u^\circ-k_u^\circ)\Rightarrow G_k
\quad\text{in }\mathcal H,
\qquad G_k=D\Phi_{r_u}[G_\rho].
\]
The limit is centered Gaussian, with
$\mathcal C_k=D\Phi_{r_u}\mathcal C_\rho D\Phi_{r_u}^*$ and kernel
\begin{align}
C_k(\kappa,\kappa')
={}&\frac{C_\rho(\kappa,\kappa')}{\bar r_u^2}
-\frac{r_u(\kappa')}{\bar r_u^3}
  \int_0^1C_\rho(\kappa,s)\,ds \notag\\
&-\frac{r_u(\kappa)}{\bar r_u^3}
  \int_0^1C_\rho(s,\kappa')\,ds
+\frac{r_u(\kappa)r_u(\kappa')}{\bar r_u^4}
  \int_0^1\!\int_0^1 C_\rho(s,t)\,ds\,dt.
\label{eq:Ck_finite_jump_theory}
\end{align}
No condition on the speed of $\varepsilon_T\downarrow0$ is required.
\end{theorem}

For a fixed mean curve with $\bar r_u\ne0$, the estimated integral stays
away from zero with probability tending to one. The cutoff case in
\eqref{eq:direct_normalized_estimator} therefore has vanishing probability
for any deterministic $\varepsilon_T\downarrow0$.
Supplementary Appendix~\ref{app:normalization} gives the quadratic
remainder, covariance calculation, and proof.

\subsection{Finite weighted aggregation}

The same results apply to fixed weighted averages of finitely many dependent
daily curves. An average of pairwise correlations need not equal the
correlation of aggregated returns.

\begin{corollary}
\label{cor:fixed_aggregation}
Let $P<\infty$ and let $w_1,\ldots,w_P$ be fixed nonnegative weights with
$\sum_aw_a=1$. For each asset pair $a=1,\ldots,P$, let $\rho_i^{(a)}$ and
$\widehat\rho_i^{(a)}$ denote its daily correlation curve and estimator,
respectively. Suppose that these curves and estimators jointly satisfy
the conditions of Theorem~\ref{thm:finite_jump_fclt_unnormalized}, with the
high-frequency conditions holding uniformly over $a$ and a common
stationary mixing state. Define
\[
\rho_i^{(w)}=\sum_{a=1}^Pw_a\rho_i^{(a)},\qquad
\widehat\rho_i^{(w)}=\sum_{a=1}^Pw_a\widehat\rho_i^{(a)},\qquad
r_u^{(w)}=\sum_{a=1}^Pw_ar_u^{(a)}.
\]
Consistency, the approximation \eqref{eq:mean_expansion}, and the
functional central limit theorem hold for the aggregate. With
$A_i^{(a)}=\rho_i^{(a)}-r_u^{(a)}$, its long-run covariance is
\begin{equation}
\mathcal C_w=\sum_{a,b=1}^Pw_aw_b
\sum_{h\in\mathbb Z}\EE(A_0^{(a)}\otimes A_h^{(b)}).
\label{eq:aggregate_covariance}
\end{equation}
If $\int_0^1r_u^{(w)}(s)\,ds\ne0$, the normalized conclusion also holds.
The feasible procedures in Section~\ref{sec:inference} apply when their
additional nondegeneracy conditions hold for the aggregate.
\end{corollary}

For inference on a fixed aggregate, the covariance
\eqref{eq:aggregate_covariance} includes all cross-covariance terms among
the component curves. The proof is given in Supplementary
Appendix~\ref{app:normalization}.

\section{Feasible inference}
\label{sec:inference}

We estimate the long-run covariance operators in the functional limits and
use them to construct confidence intervals for fixed linear summaries,
confidence sets for the mean curve, and a test of intraday constancy.
The conditions of the corresponding functional limit theorems remain in
force.

\subsection{Long-run covariance estimation}

Center the estimated daily curves at their sample mean:
\[
\widehat A_i^\rho=\widehat\rho_i-\widetilde\rho_u,
\qquad i=1,\ldots,T.
\]
Let $\omega$ be bounded and even, supported on $[-1,1]$, continuous at
zero, and such that $\omega(0)=1$. Assume also that
$[\omega((i-j)/\ell)]_{i,j=1}^T$ is positive semidefinite for every
positive integer bandwidth $\ell$ and every $T$. The Bartlett and Parzen
kernels are standard choices
\citep{NeweyWest1987,Andrews1991HAC,Politis2011PSD}.
For an integer bandwidth $\ell_T^\rho$, set
\begin{equation}
\widehat{\mathcal C}_{\rho}
=\frac1T\sum_{i,j=1}^T
\omega\!\left(\frac{i-j}{\ell_T^\rho}\right)
\widehat A_i^\rho\otimes\widehat A_j^\rho.
\label{eq:raw_hac}
\end{equation}
For $0\le h<T$, the empirical covariance at lag $h$ is
$T^{-1}\sum_{i=1}^{T-h}\widehat A_i^\rho\otimes
\widehat A_{i+h}^\rho$. The weights in \eqref{eq:raw_hac} combine these
lag covariances to account for dependence across trading days.
For every sample, $\widehat{\mathcal C}_\rho$ is finite rank,
self-adjoint, and nonnegative.

For the normalized mean, use estimated versions of
$\xi_i=D\Phi_{r_u}[A_i^\rho]$. Let
$\widehat{\bar r}_u=\ell(\widetilde\rho_u)$ and set
\begin{equation}
\widehat\xi_i=
\begin{cases}
D\Phi_{\widetilde\rho_u}[\widehat A_i^\rho],
 &|\widehat{\bar r}_u|>\varepsilon_T,\\
0,&|\widehat{\bar r}_u|\le\varepsilon_T.
\end{cases}
\label{eq:xi_hat}
\end{equation}
With a possibly different bandwidth $\ell_T^k$, define
\begin{equation}
\widehat{\mathcal C}_{k}
=\frac1T\sum_{i,j=1}^T
\omega\!\left(\frac{i-j}{\ell_T^k}\right)
\widehat\xi_i\otimes\widehat\xi_j.
\label{eq:normalized_hac}
\end{equation}
Write $\widehat{\mathcal C}_\rho[\ell]$ for the unnormalized estimator
\eqref{eq:raw_hac} evaluated at bandwidth $\ell$, and let
$\widehat{\mathsf D}_u=D\Phi_{\widetilde\rho_u}$ on
$\{|\widehat{\bar r}_u|>\varepsilon_T\}$ and zero otherwise.
Since $\widehat\xi_i=\widehat{\mathsf D}_u\widehat A_i^\rho$,
the covariance estimates satisfy, for every sample,
\begin{equation}
\widehat{\mathcal C}_k
=\widehat{\mathsf D}_u\widehat{\mathcal C}_\rho[\ell_T^k]
\widehat{\mathsf D}_u^*.
\label{eq:normalized_hac_conjugation}
\end{equation}
The bandwidths $\ell_T^\rho$ and $\ell_T^k$ determine the range of lags
across trading days, whereas $L_n$ is the intraday block length. Both
estimators use the step curves without interpolation.

Write $\|\cdot\|_1$ for the trace norm.
\begin{theorem}
\label{thm:hac}
Suppose the conditions of
Theorem~\ref{thm:finite_jump_fclt_unnormalized} hold. If
$\ell_T^\rho\to\infty$ and $\ell_T^\rho=o(T^{1/2})$, then
\begin{equation}
\|\widehat{\mathcal C}_{\rho}-\mathcal C_\rho\|_1
\xrightarrow{P}0.
\label{eq:raw_hac_consistency}
\end{equation}
If also $\bar r_u\ne0$, $\varepsilon_T\downarrow0$,
$\ell_T^k\to\infty$, and $\ell_T^k=o(T^{1/2})$, then
\begin{equation}
\|\widehat{\mathcal C}_{k}-\mathcal C_k\|_1
\xrightarrow{P}0.
\label{eq:normalized_hac_consistency}
\end{equation}
These conclusions hold whether or not the population mean is constant.
\end{theorem}

The trace-norm conclusion controls the sum of the absolute errors in the
ordered eigenvalues. This is used below to justify critical values based
on the estimated eigenvalues. The normalized result follows from
\eqref{eq:normalized_hac_conjugation} and continuity of the derivative.
Supplementary Appendix~\ref{app:hac} gives the proofs of the covariance
and linear-functional results.

\subsection{Linear functionals and confidence sets}

For a fixed $g\in\mathcal H$, write
$\theta_g=\langle r_u,g\rangle_{\mathcal H}$,
$\widehat\theta_g=\langle\widetilde\rho_u,g\rangle_{\mathcal H}$, and
$\widehat v_g=\langle\widehat{\mathcal C}_{\rho}g,g\rangle_{\mathcal H}$.
Examples include $g=\1$ for the full-day mean and
$g=|I|^{-1}\1_I$ for an average over a fixed intraday interval $I$ of
positive length. To obtain the difference between two such interval
averages, take $g$ to be the difference of their weight functions. The
intervals and the contrast must be specified independently of the sample
used for inference.

\begin{corollary}
\label{cor:linear_inference}
Under the conditions for \eqref{eq:raw_hac_consistency}, if
$v_g=\langle\mathcal C_\rho g,g\rangle_{\mathcal H}>0$, then
\[
\widehat v_g\xrightarrow{P}v_g,
\qquad
\frac{\sqrt T(\widehat\theta_g-\theta_g)}{\sqrt{\widehat v_g}}
\Rightarrow N(0,1).
\]
The quotient may be defined as zero when $\widehat v_g=0$, an event with
probability tending to zero. Thus
$\widehat\theta_g\pm z_{1-\alpha/2}\sqrt{\widehat v_g/T}$ has asymptotic
coverage $1-\alpha$. For a fixed linear functional of $k_u^\circ$, define
$v_{k,g}=\langle\mathcal C_k g,g\rangle_{\mathcal H}$ and
$\widehat v_{k,g}=\langle\widehat{\mathcal C}_k g,g\rangle_{\mathcal H}$.
Under the conditions for \eqref{eq:normalized_hac_consistency} and
$v_{k,g}>0$, the same studentized limit and confidence interval apply with
$\langle\widehat k_u^\circ,g\rangle_{\mathcal H}$ and
$\widehat v_{k,g}$.
\end{corollary}

For whole-curve inference, let
$(\widehat\lambda_{\rho,j})_{j\ge1}$ denote the eigenvalues of
$\widehat{\mathcal C}_{\rho}$ in nonincreasing order, repeated with
multiplicity and completed by zeros. Let $(\eta_j)_{j\ge1}$ be independent
standard normal variables, independent of the observations, and write
$\PP_\eta$ and $\EE_\eta$ for probability and expectation with respect
to these variables. Define
\begin{equation}
\begin{aligned}
\widehat{\mathcal Q}_\rho
&=\sum_{j=1}^{m_n}\widehat\lambda_{\rho,j}\eta_j^2,\\
\widehat c_{\rho,1-\alpha}
&=\inf\{x:\PP_\eta(\widehat{\mathcal Q}_\rho\le x\mid\text{data})
\ge1-\alpha\}.
\end{aligned}
\label{eq:gaussian_calibration}
\end{equation}
The residual curves lie in the $m_n$-dimensional step space, so eigenvalues
beyond $m_n$ are zero. The sum therefore uses all nonzero eigenvalues of
the estimated operator, without a separate truncation to a fixed number
of principal components. No eigenvalue-separation condition is required.
The quantities $\widehat{\mathcal Q}_k$ and
$\widehat c_{k,1-\alpha}$ are defined analogously from the eigenvalues of
$\widehat{\mathcal C}_{k}$.

\begin{corollary}
\label{cor:confidence_sets}
Under the conditions for \eqref{eq:raw_hac_consistency}, suppose
$\tr(\mathcal C_\rho)>0$. For each fixed $\alpha\in(0,1)$, the set
\begin{equation}
\mathfrak C_{\rho,T}(1-\alpha)
=\{f\in\mathcal H:
T\|\widetilde\rho_u-f\|_{\mathcal H}^2
\le\widehat c_{\rho,1-\alpha}\}
\label{eq:confidence_ball}
\end{equation}
satisfies $\PP\{r_u\in\mathfrak C_{\rho,T}(1-\alpha)\}\to1-\alpha$.
Under the conditions for \eqref{eq:normalized_hac_consistency} and
$\tr(\mathcal C_k)>0$, the analogous set centered at
$\widehat k_u^\circ$ with critical value $\widehat c_{k,1-\alpha}$ has
asymptotic coverage $1-\alpha$ for $k_u^\circ$.
\end{corollary}

The set in \eqref{eq:confidence_ball} is an $L^2$ ball of radius
$(\widehat c_{\rho,1-\alpha}/T)^{1/2}$ around $\widetilde\rho_u$.
This set quantifies uncertainty in integrated squared distance; it does
not provide pointwise confidence intervals or simultaneous confidence
bands. Positive trace ensures a continuous
limiting squared-norm distribution and a unique positive $(1-\alpha)$
quantile. Supplementary Appendix~\ref{app:calibration} proves these facts
and the confidence-set and testing results.

\subsection{Intraday constancy of the mean}

When $\bar r_u\ne0$, constancy of the mean curve is equivalent to
$k_u^\circ=\1$. We test
\[
H_{0,\mathrm{flat}}:k_u^\circ=\1,
\qquad
H_{1,\mathrm{flat}}:\|k_u^\circ-\1\|_{\mathcal H}>0.
\]
This is a restriction on the mean curve, not a requirement that every
daily correlation path be constant. Define
\begin{equation}
Q_T^{\mathrm{flat}}
=T\|\widehat k_u^\circ-\1\|_{\mathcal H}^2
=T\sum_{j=1}^{m_n}L_n
\{\widehat k_u^\circ(\kappa_j)-1\}^2,
\label{eq:flatness_statistic}
\end{equation}
where the second equality is exact for the step curves. Reject when
$Q_T^{\mathrm{flat}}>\widehat c_{k,1-\alpha}$ and
$|\widehat{\bar r}_u|>\varepsilon_T$. If the integral fails this cutoff,
do not reject.

\begin{corollary}
\label{cor:flatness}
Suppose the conditions for \eqref{eq:normalized_hac_consistency} hold.
Under $H_{0,\mathrm{flat}}$,
\[
Q_T^{\mathrm{flat}}\Rightarrow\|G_k\|_{\mathcal H}^2
\ \stackrel{d}{=}\ \sum_{j\ge1}\lambda_{k,j}\eta_j^2,
\]
where $(\lambda_{k,j})$ are the eigenvalues of $\mathcal C_k$.
If $\tr(\mathcal C_k)>0$, the rejection probability converges to $\alpha$.
Under every fixed alternative in $H_{1,\mathrm{flat}}$,
\[
\frac{Q_T^{\mathrm{flat}}}{T}\xrightarrow{P}
\|k_u^\circ-\1\|_{\mathcal H}^2,
\]
and the rejection probability converges to one.
\end{corollary}

\begin{remark}
\label{rem:projection}
Constancy of the unnormalized mean can be written as $P_\perp r_u=0$, where
$P_\perp f=f-\ell(f)\1$ is the orthogonal projection onto the functions with
zero integral. Since $P_\perp$ is bounded,
$\sqrt T P_\perp(\widetilde\rho_u-r_u)\Rightarrow P_\perp G_\rho$, and its
covariance is consistently estimated in trace norm by
$P_\perp\widehat{\mathcal C}_{\rho}P_\perp$. Thus constancy can be tested with
$T\|P_\perp\widetilde\rho_u\|_{\mathcal H}^2$ and the corresponding
critical value, without assuming $\bar r_u\ne0$. Positive trace of
$P_\perp\mathcal C_\rho P_\perp$ gives exact asymptotic size. Under a nonzero
constant mean, $D\Phi_{r_u}=\bar r_u^{-1}P_\perp$, so
$G_k=\bar r_u^{-1}P_\perp G_\rho$. The corresponding squared-norm limits
therefore differ by the factor $\bar r_u^{-2}$.
\end{remark}

\subsection{Implementation on the intraday grid}

For computation, let $\mathbf R$ be the $T\times m_n$ matrix with entries
$\widehat A_i^\rho(\kappa_j)$ for the unnormalized mean and
$\widehat\xi_i(\kappa_j)$ for the normalized mean. Let $\mathbf W$ have
entries $\omega((i-i')/\ell_T)$, using $\ell_T^\rho$ or $\ell_T^k$,
respectively. In the orthonormal step basis
$\psi_j=L_n^{-1/2}\1_{[\kappa_j,\kappa_j+L_n)}$, the estimated operator is
represented by
\begin{equation}
\frac{L_n}{T}\mathbf R^\top\mathbf W\mathbf R.
\label{eq:operator_matrix}
\end{equation}
The factor $L_n$ comes from the $L^2$ inner product. The nonzero eigenvalues
of the matrix in \eqref{eq:operator_matrix} coincide with those of the
estimated covariance operator, counting multiplicities, and give the
weights in \eqref{eq:gaussian_calibration}.

The results above use the exact conditional quantiles in
\eqref{eq:gaussian_calibration}. Numerically, these can be approximated by
repeated independent draws of $(\eta_1,\ldots,\eta_{m_n})$.
The kernel condition makes \eqref{eq:operator_matrix} positive
semidefinite in exact arithmetic. Only negative eigenvalues within a
prespecified floating-point tolerance may be set to zero.

\section{Monte Carlo evidence}
\label{sec:simulations}

We examine recovery of the latent daily and mean curves, coverage of the
$L^2$ confidence sets in Corollary~\ref{cor:confidence_sets}, and rejection
frequencies for the flatness test in Corollary~\ref{cor:flatness}. A paired
comparison with \citet[Section~5.1]{ChristensenHounyoLiu2026} uses the same
simulated paths and common numerical settings.

\subsection{Design}
\label{subsec:simulation_design}

The simulations use contiguous trading sessions, $g_{\mathrm{nt}}=0$, so
that $\underline t_i=i-1$. Let $V^X,V^Y,V^\rho$ be independent stationary
Ornstein--Uhlenbeck processes,
\[
 dV_t^a=-1.5V_t^a\,dt+\sqrt{3}\,dB_t^a,
 \qquad a\in\{X,Y,\rho\},
 \qquad V_0^a\sim N(0,1),
\]
with independent state and price drivers.  Put
$v(\kappa)=1+0.4\cos(2\pi\kappa)$ and
$s(\kappa)=-\sin(2\pi\kappa)$.  The spot variances and correlation on day
$i$ are
\begin{align*}
 c_i^a(\kappa)
 &=v(\kappa)\{1+0.25\tanh(V^a_{i-1+\kappa})\},
 \qquad a\in\{X,Y\},\\
 \rho_i(\kappa)
 &=r_0\{1+\delta s(\kappa)\}
   +0.2\tanh(V^\rho_{i-1+\kappa}).
\end{align*}
The main design uses $r_0\in\{0.2,0.6\}$ and
$\delta\in\{0,0.05,0.10,0.20\}$.  Symmetry of the stationary state gives
the exact population targets
\[
 r_u(\kappa)=r_0\{1+\delta s(\kappa)\},
 \qquad \bar r_u=r_0,
 \qquad k_u^\circ(\kappa)=1+\delta s(\kappa),
\]
and hence $\|k_u^\circ-\1\|_{\mathcal H}^2=\delta^2/2$. The normalized
departure from flatness is therefore the same at both correlation levels,
although their finite-sample rejection probabilities may differ. Over the
main parameter grid, correlations remain within $[-0.04,0.44]$ for
$r_0=0.2$ and $[0.28,0.92]$ for $r_0=0.6$; no clipping is used.

Continuous returns are generated from $c_i^X,c_i^Y$, and $\rho_i$ on a
one-second grid with $23{,}400$ cells per day. The OU states are advanced
by their exact transition and continued across days; lower-frequency
returns are sums of the one-second returns. The prices also receive
independent compound-Poisson jumps with intensity $0.2$ per asset-day and
centered Gaussian marks with variance $5/9$. Since expected continuous
quadratic variation is one, jumps account for $10\%$ of expected total
quadratic variation. The return-generation formulas are given in
\suppref{suppsec:simulation_implementation}.

We use two sampling designs.  In the joint-growth design,
\[
 (T,n,k_n,m_n)\in
 \{(66,390,30,13),(132,1560,60,26),(228,4680,104,45)\}.
\]
Both $T$ and $n$ increase across these configurations. The fixed-frequency
design instead holds $(n,k_n,m_n)=(390,30,13)$ and uses
$T\in\{22,66,250\}$. It examines additional trading days at the one-minute
frequency used in the application, rather than a joint high-frequency
sequence. The common $(66,390,30,13)$ configuration is simulated once and
reported in both designs.

All feasible curves use the estimator of Section~\ref{sec:estimation},
with blockwise bipower thresholds computed separately for the two return
series, $\tau_X=\tau_Y=6$, and $\varpi=0.49$. The bipower estimates are not
capped. The primary HAC estimates use the Parzen kernel with
$\ell_T^\rho=\ell_T^k=\lfloor T^{1/3}\rfloor$; the normalization cutoff is
$\varepsilon_T=1/T$. Critical values are computed from $4{,}999$ conditional
Gaussian draws.

For recovery of the daily and mean curves, we use $2{,}000$ replications
for each configuration. For coverage, we use $5{,}000$ replications for
each of the two main configurations with $T=66$ and $2{,}000$ for each
remaining main configuration. For flatness tests, we use $5{,}000$
replications under the null and $2{,}000$ under each fixed alternative.
Monte Carlo standard errors accompany coverage
and rejection frequencies. \suppref{suppsec:simulation_implementation}
gives the remaining calibration and evaluation details, and
\suppref{suppsec:simulations} reports the supplementary results.

\subsection{Recovery of daily and mean curves}
\label{subsec:simulation_recovery}

Daily RMSE is obtained by averaging
$\|\widehat\rho_i-\rho_i\|_{\mathcal H}^2$ over replications and trading
days and then taking the square root. For the mean curve, let
$\bar\rho_T=T^{-1}\sum_{i=1}^T\rho_i$ denote the latent sample mean evaluated
on the one-second grid. We report Monte Carlo means of the sampling error
$\|\bar\rho_T-r_u\|_{\mathcal H}$, the high-frequency estimation error
$\|\widetilde\rho_u-\bar\rho_T\|_{\mathcal H}$, and the total error
$\|\widetilde\rho_u-r_u\|_{\mathcal H}$. These norms are not additive. We also
multiply the high-frequency estimation error by $\sqrt T$ to assess its
magnitude on the scale of the functional central limit theorem. The
estimation and block-projection components of daily squared error are
reported in \suppref{suppsec:ring1_full}.

Table~\ref{tab:simulation_recovery} reports the jump design used in the
subsequent experiments. Daily RMSE declines across the three joint-growth
configurations. The sampling error is approximately proportional to
$T^{-1/2}$: multiplying the displayed values by $\sqrt T$ gives
$0.1356$--$0.1366$. The scaled high-frequency estimation error falls from
$0.1767$ to $0.0944$ for $r_0=0.2$ and from $0.1474$ to $0.0746$ for
$r_0=0.6$. This decrease is consistent with the approximation in
\eqref{eq:mean_expansion}. Results without price jumps show a similar
pattern; see \suppref{suppsec:ring1_full}.

\begingroup
\renewcommand{\mcse}[1]{ (#1)}
\begin{table}[t]
\centering
\caption{Curve recovery in the joint-growth design}
\label{tab:simulation_recovery}
\small
\begin{tabular}{rrrrrrrr}
\toprule
$T$ & $n$ & $r_0$ & Daily RMSE & Sampling & HF estimation & Total & $\sqrt{T}\,$ HF error \\
\midrule
66  & 390  & 0.2 & 0.1759 & 0.01681 & 0.02175 & 0.02735 & 0.1767 \\
132 & 1560 & 0.2 & 0.1239 & 0.01186 & 0.01081 & 0.01616 & 0.1242 \\
228 & 4680 & 0.2 & 0.0939 & 0.00898 & 0.00625 & 0.01110 & 0.0944 \\
66  & 390  & 0.6 & 0.1255 & 0.01681 & 0.01814 & 0.02423 & 0.1474 \\
132 & 1560 & 0.6 & 0.0865 & 0.01186 & 0.00866 & 0.01464 & 0.0995 \\
228 & 4680 & 0.6 & 0.0652 & 0.00898 & 0.00494 & 0.01030 & 0.0746 \\
\bottomrule
\end{tabular}
\begin{minipage}{0.96\linewidth}
\footnotesize
\emph{Notes:} Results use $2{,}000$ replications with jumps and
$\delta=0.10$. Daily RMSE is obtained by averaging the integrated squared
error over replications and trading days and then taking the square root.
Sampling, HF estimation, and Total are
Monte Carlo means of $\|\bar\rho_T-r_u\|_{\mathcal H}$,
$\|\widetilde\rho_u-\bar\rho_T\|_{\mathcal H}$, and
$\|\widetilde\rho_u-r_u\|_{\mathcal H}$, respectively. The final column
multiplies the HF estimation error by $\sqrt T$.
\end{minipage}
\end{table}
\endgroup

\subsection{Functional confidence sets}
\label{subsec:simulation_coverage}

Table~\ref{tab:simulation_coverage} reports coverage and mean radii of the
nominal $95\%$ whole-curve confidence sets with jumps and $\delta=0.10$.
Coverage of the unnormalized confidence sets is below the nominal level,
ranging from $0.8684$ to $0.9185$.
Normalized coverage is close to nominal for $r_0=0.6$
($0.9450$--$0.9552$) and conservative for $r_0=0.2$
($0.9660$--$0.9690$).

Mean radii decrease across the joint-growth configurations for both
objects. The larger normalized radii at $r_0=0.2$ are consistent with the
sensitivity of normalization to the mean level. Coverage does not improve
uniformly as the sets become narrower. The experiment does not separately
identify the contributions of long-run covariance estimation,
high-frequency estimation error, and finite-sample Gaussian approximation
to the coverage errors.

\begingroup
\renewcommand{\mcse}[1]{ (#1)}
\begin{table}[t]
\centering
\caption{Coverage of nominal $95\%$ whole-curve $L^2$ confidence sets}
\label{tab:simulation_coverage}
\small
\resizebox{\textwidth}{!}{%
\begin{tabular}{rrrrrrrr}
\toprule
$T$ & $n$ & $r_0$ & $N$ & Unnormalized coverage & Unnormalized radius & Normalized coverage & Normalized radius \\
\midrule
66  & 390  & 0.2 & 5000 & 0.9110\mcse{0.0040} & 0.03953 & 0.9660\mcse{0.0026} & 0.15478 \\
66  & 390  & 0.6 & 5000 & 0.8684\mcse{0.0048} & 0.03521 & 0.9552\mcse{0.0029} & 0.03858 \\
132 & 1560 & 0.2 & 2000 & 0.9105\mcse{0.0064} & 0.02523 & 0.9675\mcse{0.0040} & 0.07774 \\
132 & 1560 & 0.6 & 2000 & 0.8990\mcse{0.0067} & 0.02392 & 0.9510\mcse{0.0048} & 0.02058 \\
228 & 4680 & 0.2 & 2000 & 0.9185\mcse{0.0061} & 0.01867 & 0.9690\mcse{0.0039} & 0.04814 \\
228 & 4680 & 0.6 & 2000 & 0.9145\mcse{0.0063} & 0.01814 & 0.9450\mcse{0.0051} & 0.01359 \\
\bottomrule
\end{tabular}%
}
\begin{minipage}{0.96\linewidth}
\footnotesize
\emph{Notes:} Monte Carlo standard errors are in parentheses.  Radius is
the replication mean of the $L^2$-ball radius.  Jumps are included and
$\delta=0.10$ throughout.
\end{minipage}
\end{table}
\endgroup

\subsection{Flatness tests}
\label{subsec:simulation_flatness}

The flatness experiment uses $Q_T^{\mathrm{flat}}$ in
\eqref{eq:flatness_statistic}, with $\delta=0$ under the null and
$\delta\in\{0.05,0.10,0.20\}$ as fixed alternatives.
Table~\ref{tab:simulation_flatness} reports rejection frequencies at the
nominal $5\%$ level. Null rejection rates range from $0.0202$ to $0.0408$
and remain below nominal across the design, without a monotone pattern
along the joint-growth configurations.

Power increases with $\delta$ and as $T$ and $n$ increase in the
joint-growth design. For $r_0=0.2$ and $\delta=0.10$, rejection frequencies
are $0.1800$, $0.7935$, and $0.9980$ across the three configurations. The
weaker alternative $\delta=0.05$ gives $0.0615$, $0.2295$, and $0.5975$.
Rejection frequencies are generally higher at $r_0=0.6$.

At fixed one-minute sampling, additional days also increase power. For
$r_0=0.2$ and $\delta=0.10$, rejection frequencies are $0.0415$, $0.1800$,
and $0.7600$ at $T=22,66,250$. These results describe the effect of a
longer sample at fixed $n$ and $k_n$, rather than the joint limit in
Section~\ref{sec:estimation}.

\begingroup
\renewcommand{\mcse}[1]{ (#1)}
\begin{table}[t]
\centering
\caption{Flatness rejection frequencies}
\label{tab:simulation_flatness}
\small
\resizebox{\textwidth}{!}{%
\begin{tabular}{llrrrrrr}
\toprule
Design & $r_0$ & $T$ & $n$ & $\delta=0$ & $\delta=0.05$ & $\delta=0.10$ & $\delta=0.20$ \\
\midrule
Joint growth & 0.2 & 66  & 390  & 0.0348\mcse{0.0026} & 0.0615\mcse{0.0054} & 0.1800\mcse{0.0086} & 0.7370\mcse{0.0098} \\
Joint growth & 0.2 & 132 & 1560 & 0.0282\mcse{0.0023} & 0.2295\mcse{0.0094} & 0.7935\mcse{0.0091} & 1.0000\mcse{0.0000} \\
Joint growth & 0.2 & 228 & 4680 & 0.0328\mcse{0.0025} & 0.5975\mcse{0.0110} & 0.9980\mcse{0.0010} & 1.0000\mcse{0.0000} \\
Joint growth & 0.6 & 66  & 390  & 0.0304\mcse{0.0024} & 0.7200\mcse{0.0100} & 1.0000\mcse{0.0000} & 1.0000\mcse{0.0000} \\
Joint growth & 0.6 & 132 & 1560 & 0.0288\mcse{0.0024} & 1.0000\mcse{0.0000} & 1.0000\mcse{0.0000} & 1.0000\mcse{0.0000} \\
Joint growth & 0.6 & 228 & 4680 & 0.0364\mcse{0.0026} & 1.0000\mcse{0.0000} & 1.0000\mcse{0.0000} & 1.0000\mcse{0.0000} \\
\midrule
Fixed frequency & 0.2 & 22  & 390 & 0.0202\mcse{0.0020} & 0.0200\mcse{0.0031} & 0.0415\mcse{0.0045} & 0.1890\mcse{0.0088} \\
Fixed frequency & 0.2 & 66  & 390 & 0.0348\mcse{0.0026} & 0.0615\mcse{0.0054} & 0.1800\mcse{0.0086} & 0.7370\mcse{0.0098} \\
Fixed frequency & 0.2 & 250 & 390 & 0.0396\mcse{0.0028} & 0.2025\mcse{0.0090} & 0.7600\mcse{0.0095} & 1.0000\mcse{0.0000} \\
Fixed frequency & 0.6 & 22  & 390 & 0.0208\mcse{0.0020} & 0.1955\mcse{0.0089} & 0.8155\mcse{0.0087} & 1.0000\mcse{0.0000} \\
Fixed frequency & 0.6 & 66  & 390 & 0.0304\mcse{0.0024} & 0.7200\mcse{0.0100} & 1.0000\mcse{0.0000} & 1.0000\mcse{0.0000} \\
Fixed frequency & 0.6 & 250 & 390 & 0.0408\mcse{0.0028} & 1.0000\mcse{0.0000} & 1.0000\mcse{0.0000} & 1.0000\mcse{0.0000} \\
\bottomrule
\end{tabular}%
}
\begin{minipage}{0.96\linewidth}
\footnotesize
\emph{Notes:} Monte Carlo standard errors are in parentheses.  Null columns
use $N=5{,}000$; fixed-alternative columns use $N=2{,}000$.  All cells include
jumps.  The $(66,390)$ cells are common to the two designs.  An entry of
$1.0000$ means that all $2{,}000$ simulated samples rejected.
\end{minipage}
\end{table}
\endgroup

Figure~\ref{fig:simulation_size_power} plots the rejection frequencies
for the two sampling designs.

\begin{figure}[!htbp]
\centering
\includegraphics[width=\textwidth]{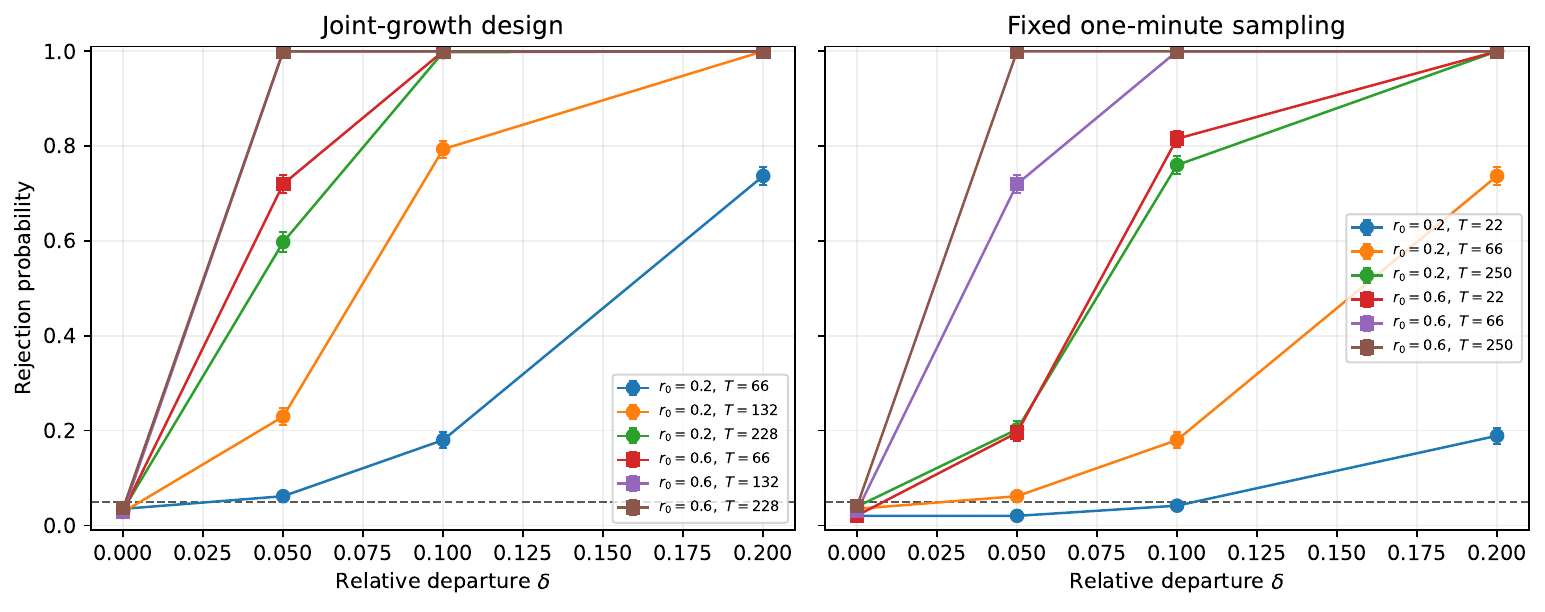}
\caption{Flatness rejection frequencies under joint growth and at fixed
one-minute sampling. The horizontal dashed line is the 5\% nominal level;
bars are 95\% Monte Carlo intervals. Each line holds $(r_0,T)$ fixed and
varies the departure from flatness $\delta$.}
\label{fig:simulation_size_power}
\end{figure}

We compare the test with the procedure in Section~5.1 of
\citet{ChristensenHounyoLiu2026}, denoted CHL below, using the same simulated
paths. The CHL procedure forms a covariance-first correlation curve and uses the
long-run covariance and Gaussian calibration associated with that
estimator. The two flatness null hypotheses coincide under this design.

Rejection frequencies are broadly comparable under common numerical
settings. CHL null rejection rates are $0.0234$--$0.0456$, compared with
$0.0202$--$0.0408$ for the proposed test. Alternative rejection rates are
also close, with differences in either direction across configurations;
they are not adjusted to a common empirical size.
\suppref{suppsec:chl_simulation} reports the paired results and examines
the sensitivity of CHL rejection frequencies to the numerical settings.
Changing these settings has a material effect in some low-correlation
configurations.

The comparisons in \suppref{suppsec:ring3_diagnostics} examine latent
curves, alternative HAC bandwidths, and observations with and without
price jumps. Null rejection remains conservative when latent curves are
used, while high-frequency estimation is associated with a substantial
power loss in the coarse, low-correlation configuration. The specified
price jumps have little effect on rejection frequencies in these
comparisons.

\section{Empirical application}
\label{sec:empirical}

We estimate the level and intraday shape of stock--proxy and
average-pairwise correlation among 30 U.S. stocks over July 2010--December
2019. Pooled curves describe the full sample, while tests of intraday
constancy use prespecified quarterly windows. We compare the estimates
and test decisions with the covariance-first procedure of
\citet[Section~5.1]{ChristensenHounyoLiu2026}.

\subsection{Data and implementation}
\label{subsec:empirical_design}

The balanced panel contains 30 stocks and 2,370 trading days from July~1,
2010 through December~31, 2019, with 390 one-minute regular-session log
returns per stock and day. We divide the 09:30--16:00 session into thirteen
nonoverlapping thirty-minute blocks.

For each stock, the leave-one-out proxy is the equally weighted one-minute
return on the other 29 stocks. We form the proxy before computing local
scales and truncating returns, estimate the 30 stock--proxy daily
correlation curves, and average them with weight $1/30$. The second
aggregate equally averages the daily correlation curves for all
$\binom{30}{2}=435$ unordered stock pairs. The weights are fixed in
advance, and Corollary~\ref{cor:fixed_aggregation} applies to both
aggregates for the proposed estimator.

The proposed daily curves use the estimator in
Section~\ref{sec:estimation}, with blockwise bipower-variation scales,
threshold constant 6, exponent $0.49$, and no upper bound on the scale
estimates. For each stock--proxy or stock pair, the CHL procedure instead
averages covariance coordinates across the days of a window before
forming the correlation ratio; it then averages these curves with the
same fixed weights. For inference, each method estimates the long-run
covariance from its aggregate daily residual or influence curve, retaining
cross-covariances among the component curves.
Supplementary Appendix~\ref{supsubsec:empirical_implementation} gives the
aggregation details.

Formal inference uses all 38 complete calendar quarters from the third
quarter of 2010 (2010Q3) through the fourth quarter of 2019 (2019Q4),
each containing 60--64 trading days. We assume stationarity and
weak dependence within each quarter, while allowing the population mean
curve to differ across quarters. For a quarter of length $T_q$, both
procedures use the Parzen kernel, HAC bandwidth
$\ell_{T_q}=\lfloor T_q^{1/3}\rfloor$, and $200{,}000$ draws from their
respective estimated Gaussian limits. Each procedure uses the covariance
estimate appropriate to its estimator.

\subsection{Correlation levels and intraday shape}
\label{subsec:empirical_curves}

Figure~\ref{fig:empirical_curves} displays the pooled curves on the original
correlation scale and after division by their respective full-day means.
Table~\ref{tab:empirical_levels} reports the full-day, early, and close
averages. For a step curve $f$, let $\bar f_E$ and $\bar f_C$ denote its
averages over 09:30--10:30 and 15:00--16:00, respectively. These intervals
comprise the first two and last two thirty-minute blocks, and the
close-minus-early contrast is $\bar f_C-\bar f_E$.

\begin{figure}[!htbp]
  \centering
  \includegraphics[width=\textwidth]{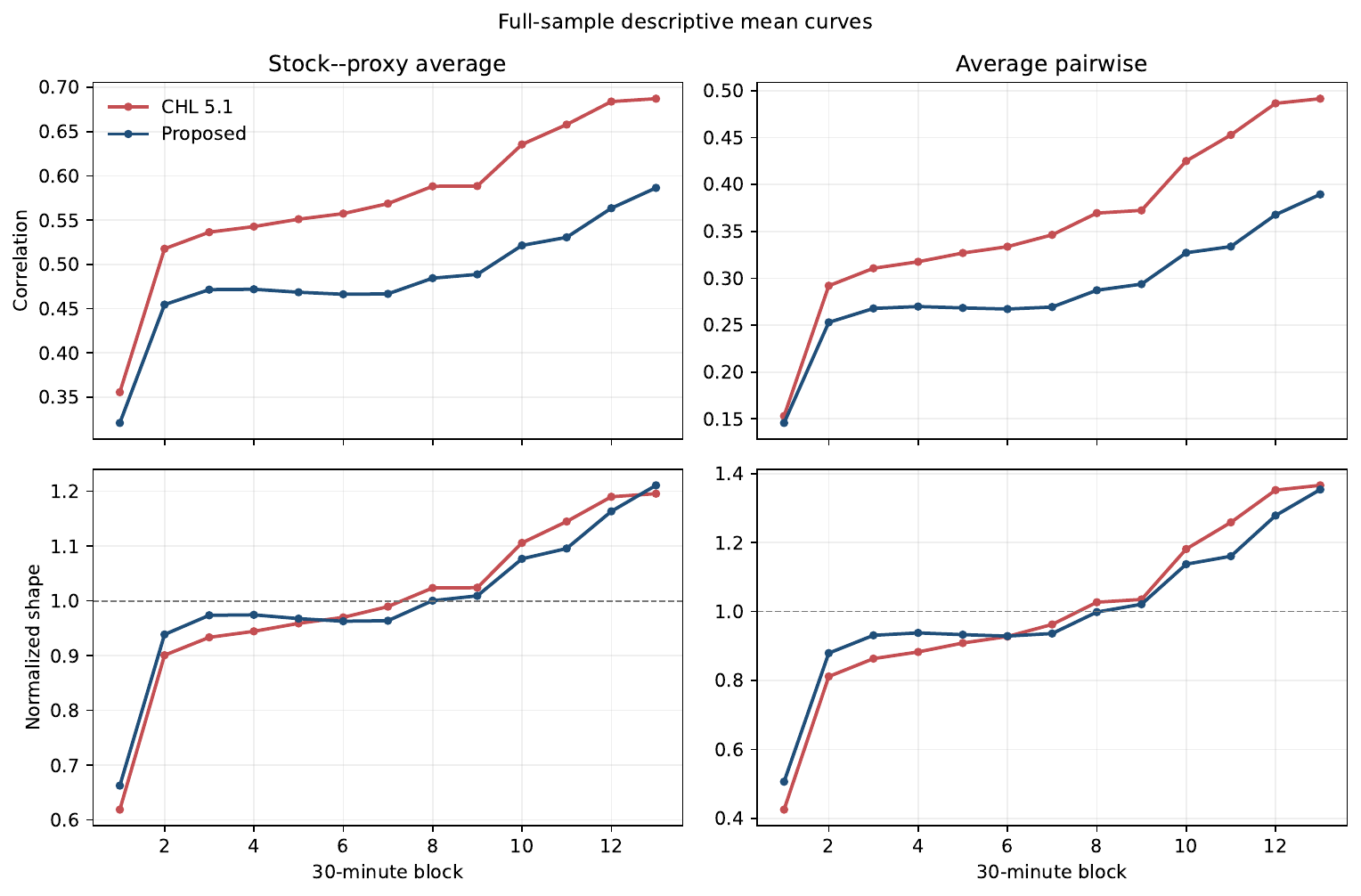}
  \caption{Pooled intraday correlation curves, July 2010--December 2019.
  Columns show the stock--proxy and average-pairwise aggregates. The top
  row uses the original correlation scale; the bottom row divides each
  aggregate curve by its own full-day mean. Proposed denotes the mean of
  the estimated daily correlation curves; CHL~5.1 denotes the
  covariance-first reference. These curves are descriptive; formal
  inference uses quarterly windows.}
  \label{fig:empirical_curves}
\end{figure}

\begin{table}[!htbp]
\centering
\caption{Pooled correlation levels and close-minus-early contrasts}
\label{tab:empirical_levels}
\small
\begin{tabular}{llrrrr}
\toprule
Object & Method & Full-day average & Early & Close & Close $-$ early\\
\midrule
\multicolumn{6}{l}{\textit{Panel A: original correlation scale}}\\
Stock--proxy & Proposed & 0.484250 & 0.387723 & 0.574913 & 0.187190\\
              & CHL 5.1   & 0.574644 & 0.436639 & 0.685501 & 0.248862\\
Average pairwise & Proposed & 0.287664 & 0.199293 & 0.378530 & 0.179237\\
                 & CHL 5.1   & 0.359788 & 0.222514 & 0.488923 & 0.266408\\
\addlinespace
\multicolumn{6}{l}{\textit{Panel B: normalized shape}}\\
Stock--proxy & Proposed & 1 & 0.800668 & 1.187224 & 0.386556\\
              & CHL 5.1   & 1 & 0.759842 & 1.192914 & 0.433072\\
Average pairwise & Proposed & 1 & 0.692797 & 1.315873 & 0.623075\\
                 & CHL 5.1   & 1 & 0.618460 & 1.358920 & 0.740460\\
\bottomrule
\end{tabular}
\begin{minipage}{0.96\linewidth}\footnotesize
\emph{Notes:} All rows use the same 2,370 days and 30-stock sample.
Early and Close report averages over the first two and last two
thirty-minute blocks, respectively. Panel B divides each method's curve
by its own full-day average. The entries are
descriptive estimates, not a test that the two population curves coincide.
\end{minipage}
\end{table}

The proposed full-day estimates are $0.484250$ for stock--proxy
correlation and $0.287664$ for average-pairwise correlation. The
closing-hour averages exceed the opening-hour averages by approximately
$0.187$ and $0.179$ in correlation-coefficient units, respectively. Thus the
two aggregates have similar absolute increases in average correlation
from the opening hour to the closing hour, despite their different
full-day levels.

After normalization, the close-minus-early contrast is $0.623075$ for
average-pairwise correlation and $0.386556$ for stock--proxy correlation.
Pairwise correlation therefore has the larger increase relative to its
full-day mean, although its increase on the original scale is slightly
smaller.

The covariance-first estimates show the same ordering: the closing-hour
average exceeds the opening-hour average for both aggregates, and the
relative increase is larger for average-pairwise correlation. Their full-day estimates,
$0.574644$ and $0.359788$, exceed the corresponding proposed estimates,
as do their normalized contrasts, $0.433072$ and $0.740460$. The two
procedures thus agree on the sign of the close-minus-early contrast but
differ in its magnitude and in the full-day level.

\subsection{Quarterly evidence against flatness}
\label{subsec:empirical_quarterly}

For each quarter and aggregate, we apply the proposed statistic
$Q_T^{\mathrm{flat}}$ in \eqref{eq:flatness_statistic}. CHL~5.1 subtracts
the intraday mean from its estimated curve and calibrates its statistic
using the long-run covariance of its influence curve.
Table~\ref{tab:empirical_quarterly} reports decisions at the $5\%$ level,
before and after Holm adjustment. The adjustment is applied separately
to the 38 quarterly tests for each method--object pair.

\begin{table}[!htbp]
\centering
\caption{Quarterly flatness inference and method agreement}
\label{tab:empirical_quarterly}
\small
\begin{tabular}{llrr}
\toprule
Object & Method & \shortstack{Unadjusted\\rejections} & \shortstack{Holm-adjusted\\rejections}\\
\midrule
Stock--proxy average & Proposed & 38/38 & 38/38\\
                     & CHL 5.1   & 38/38 & 38/38\\
Average pairwise     & Proposed & 38/38 & 38/38\\
                     & CHL 5.1   & 37/38 & 37/38\\
\midrule
\multicolumn{2}{l}{Decision agreement: stock--proxy} &
\multicolumn{2}{c}{38/38}\\
\multicolumn{2}{l}{Decision agreement: pairwise} &
\multicolumn{2}{c}{37/38}\\
\bottomrule
\end{tabular}
\begin{minipage}{0.96\linewidth}\footnotesize
\emph{Notes:} Decisions use the $5\%$ level in all 38 complete quarters
from 2010Q3 to 2019Q4, with 60--64 trading days per quarter. Holm adjustment
uses four families of 38 tests, one for each method--object pair.
Each test concerns intraday constancy within a quarter under maintained
stationarity, not equality of mean curves across quarters.
\end{minipage}
\end{table}

The proposed test rejects flatness for both aggregates in all 38
prespecified quarters, including after Holm adjustment.

The two procedures give the same decision in 75 of the 76 quarter--object
comparisons. The only disagreement is for average-pairwise correlation in
the fourth calendar quarter of 2017 (2017Q4; 62 trading days): the
proposed test rejects flatness
(Holm-adjusted $p=0.000190$), whereas CHL~5.1 does not
(Holm-adjusted $p=0.123469$). Supplementary
Appendix~\ref{supsubsec:empirical_quarterly} reports all quarterly results
and the statistics for this comparison.

\section{Conclusion}
\label{sec:conclusion}

This paper provides direct functional inference for the mean curve of
latent intraday spot correlation. Forming correlation ratios within each
day before averaging aligns the estimator with an interpretable
population mean on the original correlation scale, without requiring a
multiplicatively separable correlation pattern.  Together with positive-semidefinite,
trace-norm consistent long-run covariance estimation, this yields
inference for linear summaries, normalized shapes, and intraday constancy,
as well as whole-curve confidence sets calibrated using all eigenvalues of
the estimated covariance operator. 

In the simulations, errors in estimating the daily and mean curves
decrease as $T$ and $n$ increase jointly. Power increases with the
magnitude of departures from flatness but remains limited in some
low-correlation configurations with weak alternatives.
In the empirical application, the closing-hour average exceeds the
opening-hour average for both aggregates. The proposed test rejects
intraday constancy for both aggregates in all 38 quarters, including
after Holm adjustment.
The methods of this paper and \citet{ChristensenHounyoLiu2026} yield the same flatness-test conclusions in all stock–proxy quarters and in 37 of the 38 average-pairwise quarters, despite differences in their estimates of correlation levels and of the increase in average
correlation from the opening hour to the closing hour. The mean curves quantify these differences, revealing the time-of-day variation of dependence beyond what a single full-day coefficient or a flatness-test decision can convey.

\clearpage
\thispagestyle{empty}
\begin{center}
\vspace*{0.28\textheight}
{\LARGE Supplementary Appendix to\par}
\vspace{1.2em}
{\Large \papertitle\par}
\end{center}
\clearpage

\clearpage
\setcounter{section}{0}
\setcounter{equation}{0}
\setcounter{figure}{0}
\setcounter{table}{0}
\setcounter{theorem}{0}
\renewcommand{\thesection}{S\arabic{section}}
\numberwithin{equation}{section}
\counterwithin{figure}{section}
\counterwithin{table}{section}
\renewcommand{\thepage}{S\arabic{page}}
\makeatletter
\renewcommand*{\theHsection}{supp.\arabic{section}}
\renewcommand*{\theHsubsection}{supp.\arabic{section}.\arabic{subsection}}
\renewcommand*{\theHsubsubsection}{supp.\arabic{section}.\arabic{subsection}.\arabic{subsubsection}}
\renewcommand*{\theHequation}{supp.\arabic{section}.\arabic{equation}}
\renewcommand*{\theHfigure}{supp.\arabic{section}.\arabic{figure}}
\renewcommand*{\theHtable}{supp.\arabic{section}.\arabic{table}}
\providecommand*{\theHtheorem}{}
\renewcommand*{\theHtheorem}{supp.\arabic{section}.\arabic{theorem}}
\providecommand*{\theHpage}{}
\renewcommand*{\theHpage}{supp.\arabic{page}}
\makeatother
\setcounter{page}{1}
\section*{Guide to the Supplementary Appendix}
\addcontentsline{toc}{section}{Guide to the Supplementary Appendix}
This supplement contains all proofs, together with the population
calculations and additional numerical and empirical results. The table
below gives their locations.

\begin{center}
\small
\begin{tabular}{p{0.39\textwidth}p{0.53\textwidth}}
\toprule
Main-paper result or evidence & Proof and supporting material\\
\midrule
Daily recovery & Supplementary Appendix~\ref{app:daily_moments},
using the bounds in Supplementary Appendix~\ref{app:regularity}\\
Consistency and raw FCLT & Supplementary Appendix~\ref{app:consistency},
with the theorem proofs in \ref{app:main_mean_proofs}\\
Normalized FCLT and fixed aggregation &
Supplementary Appendix~\ref{app:normalization}\\
HAC consistency and feasible calibration &
Supplementary Appendices~\ref{app:hac}--\ref{app:calibration}\\
Population identities and separating examples &
Supplementary Appendix~\ref{app:population}\\
Monte Carlo evidence & Supplementary Appendix~\ref{suppsec:simulations}\\
Empirical implementation and quarterly output &
Supplementary Appendix~\ref{suppsec:empirical}\\
\bottomrule
\end{tabular}
\end{center}

\section{High-frequency and semimartingale bounds}
\label{app:regularity}

This section collects the local estimates used to recover the latent daily
curves. Throughout, $C$ denotes a finite constant that may vary from line
to line and, unless stated otherwise, is uniform in the trading-session
and block indices and in $T,n,k_n$.

\subsection{Block notation and continuous characteristics}

Put $\kappa_{m_n+1}=1$ and
$s_{i,j}=\underline t_i+\kappa_j$. For
$\kappa\in[\kappa_j,\kappa_{j+1})$, define
\[
C_i^{[n]}(\kappa)
=(c^X_{s_{i,j}},c^Y_{s_{i,j}},c^{XY}_{s_{i,j}})^\top,
\qquad
\rho_i^{[n]}(\kappa)=\rho_{s_{i,j}}.
\]
They are left-endpoint discretizations; their values at $\kappa=1$ are
immaterial in $L^2$. For coordinate vectors, $|\cdot|$ denotes the Euclidean
norm; for matrices, it denotes the Frobenius norm. For a random scalar,
vector, or matrix $U$, write $\|U\|_{L^2(\Omega)}=(\EE|U|^2)^{1/2}$.
Let $Z^c=Z-J$ be the original continuous-price process. For each price
$\wp\in\{X,Y\}$, write
\[
A_{i,\ell}^\wp=\int_{t_{i,\ell-1}}^{t_{i,\ell}}a_s^\wp\,ds,
\qquad
M_{i,\ell}^\wp=M^\wp_{t_{i,\ell}}-M^\wp_{t_{i,\ell-1}},
\qquad \Delta_{i,\ell}\wp^c=A_{i,\ell}^\wp+M_{i,\ell}^\wp.
\]
Here $A_{i,\ell}^\wp$ denotes a price-drift increment, while $A_i^\rho$
denotes the centered daily correlation curve.

\begin{lemma}
\label{lem:continuous_characteristics}
Under Assumptions~\ref{ass:V_volatility}--\ref{ass:primitive_correlation_regularity},
for $\wp\in\{X,Y,XY\}$,
\begin{equation}
c_t^\wp=c_0^\wp+\int_0^tb_u^{c,\wp}\,du+N_t^{c,\wp}
+\sum_{0<u\le t}\Delta c_u^\wp,
\label{eq:continuous_covariance_decomposition}
\end{equation}
where $N^{c,\wp}$ is a continuous local martingale and
$v_u^{c,\wp}=d\langle N^{c,\wp}\rangle_u/du$ exists. The drift and bracket
density satisfy
\begin{equation}
\sup_u\EE\left[|b_u^{c,\wp}|^4+(v_u^{c,\wp})^2\right]<\infty.
\label{eq:continuous_characteristics}
\end{equation}
The corresponding bounds for $\rho$ hold with drift $b^\rho$ and
bracket density $\|\Sigma^\rho\|^2$.
\end{lemma}

\begin{proof}
For $\wp\in\{X,Y\}$, It\^o's formula applied to
$c^\wp=(\sigma^\wp)^2$ gives
\[
b_s^{c,\wp}=2\sigma_{s-}^\wp\widetilde a_s^\wp
+\|\widetilde\sigma_s^\wp\|^2+\|\widetilde v_s^\wp\|^2,
\]
\[
N_t^{c,\wp}=2\int_0^t\sigma_{s-}^\wp\widetilde\sigma_s^\wp\,dW_s
+2\int_0^t\sigma_{s-}^\wp\widetilde v_s^\wp\,d\widetilde W_s.
\]
For the cross coordinate, let $N^\rho,N^X,N^Y$ denote the continuous
local-martingale parts of $\rho,\sigma^X,\sigma^Y$, and write
$\gamma_s^{UV}=d\langle N^U,N^V\rangle_s/ds$.  The continuous product
rule for $c^{XY}=\rho\sigma^X\sigma^Y$ yields, with $\rho$, $\sigma^X$,
and $\sigma^Y$ evaluated at $s-$,
\[
\begin{aligned}
b_s^{c,XY}
={}&\sigma^X\sigma^Yb_s^\rho
+\rho\sigma^Y\widetilde a_s^X
+\rho\sigma^X\widetilde a_s^Y
+\rho\gamma_s^{XY}
+\sigma^Y\gamma_s^{\rho X}
+\sigma^X\gamma_s^{\rho Y},\\
dN_s^{c,XY}
={}&\sigma^X\sigma^Y\,dN_s^\rho
+\rho\sigma^Y\,dN_s^X
+\rho\sigma^X\,dN_s^Y .
\end{aligned}
\]
The densities $\gamma^{UV}$ exist by the Kunita--Watanabe inequality
and the absolute continuity of the diagonal brackets. The remaining
terms are the finite sum of the actual covariance jumps, giving
\eqref{eq:continuous_covariance_decomposition}.

To bound the continuous characteristics, put
\begin{align*}
S_s&=1+|\sigma_{s-}^X|+|\sigma_{s-}^Y|,\\
B_s&=1+|\widetilde a_s^X|+|\widetilde a_s^Y|+|b_s^\rho|,\\
H_s&=1+\sum_{\wp=X,Y}(\|\widetilde\sigma_s^\wp\|
+\|\widetilde v_s^\wp\|)+\|\Sigma_s^\rho\|.
\end{align*}
Under the assumptions, $S_s$ has moments of all finite orders, and
$B_s$ and $H_s$ have moments of orders $p_\star/2$ and $p_\star$,
respectively, all uniformly bounded in $s$.
The product formulas and Kunita--Watanabe imply, for every covariance coordinate,
\[
|b_s^{c,\wp}|\le CS_s^2(B_s+H_s^2),
\qquad v_s^{c,\wp}\le CS_s^4H_s^2.
\]
Hence
\[
|b_s^{c,\wp}|^4+(v_s^{c,\wp})^2
\le CS_s^8(B_s^4+H_s^8).
\]
H\"older's inequality with conjugate exponents
$\theta=p_\star/(p_\star-8)$ and $\theta'=p_\star/8$ gives
\[
\sup_s\EE\{S_s^8(B_s^4+H_s^8)\}<\infty,
\]
since $4\theta'=p_\star/2$ and $8\theta'=p_\star$. This proves
\eqref{eq:continuous_characteristics}. The assertion for $\rho$ follows
directly from Assumption~\ref{ass:primitive_correlation_regularity}.
\end{proof}

\subsection{Auxiliary continuous-price processes}

On each local block, we use an auxiliary continuous-price process whose
covariance is frozen at the left limit of the first state-jump mark.
The process agrees with the original continuous prices on blocks without
state marks. The fourth-moment bounds obtained from this construction are used below
to control the quadratic remainder in the local correlation expansion. Let $N^{\rm st}$ count all marks of
$\widetilde\mu^X,\widetilde\mu^Y,\mu^\rho$. Its predictable intensity is
\[
\lambda_u^{\rm st}
=\widetilde F^X(\RR)\widetilde\chi_u^X
+\widetilde F^Y(\RR)\widetilde\chi_u^Y+F^\rho(E)\chi_u^\rho.
\]
For a deterministic block start $s$, put
$\tau_s=\inf\{t>s:N_t^{\rm st}>N_s^{\rm st}\}$, with $\inf\varnothing=\infty$.
For $t\ge s$, define the locally frozen states by
\begin{equation}
(c_t^{*,s},\rho_t^{*,s})
=
\begin{cases}
(c_t,\rho_t),&s\le t<\tau_s,\\
(c_{\tau_s-},\rho_{\tau_s-}),&t\ge\tau_s.
\end{cases}
\label{eq:auxiliary_states}
\end{equation}
We write $c^{*,s,\wp}$ for a covariance coordinate of $c^{*,s}$.

Extend the space by a two-dimensional Brownian motion $B^*$ independent
of the original $\sigma$-field $\mathcal F$. On the product filtration, which we continue
to denote by $(\mathcal F_t)$, the original martingales and compensators
are unchanged. Conditional expectations of original random variables
are also unchanged. Set
\begin{equation}
\begin{aligned}
M_t^{*,s}
&=\int_s^t\1_{\{u\le\tau_s\}}\,dM_u
+\int_s^t\1_{\{u>\tau_s\}}c_{\tau_s-}^{1/2}\,dB_u^*,\\
Z_t^{*,s}&=Z_s^c+\int_s^ta_u\,du+M_t^{*,s},\qquad t\ge s.
\end{aligned}
\label{eq:auxiliary_price}
\end{equation}
The second integrand is defined as zero before $\tau_s$ and when
$\tau_s=\infty$; the square root is the nonnegative symmetric matrix
square root.

\begin{lemma}
\label{lem:auxiliary_state}
Under the conditions of Lemma~\ref{lem:continuous_characteristics} and
the variance lower bound, $M^{*,s}$ is a continuous local martingale
with covariance density $c^{*,s}$. The latter is continuous and positive
semidefinite, with both marginal variances at least $\underline c$.
For deterministic $s\ge0$ and $0<h\le1$, let
$J_{s,h}^{\rm st}=\{\tau_s\le s+h\}$. Then, uniformly in $s,h$,
\begin{equation}
\begin{aligned}
\PP(J_{s,h}^{\rm st}\mid\mathcal F_s)&\le h\Lambda_{s,h}^{\rm st},
&\EE(\Lambda_{s,h}^{\rm st})^2&\le C,\\
\Lambda_{s,h}^{\rm st}
&=\frac1h\EE\left(\int_s^{s+h}\lambda_u^{\rm st}\,du
\,\middle|\,\mathcal F_s\right).
\end{aligned}
\label{eq:state_jump_probability}
\end{equation}
The auxiliary covariance satisfies
\begin{equation}
\begin{aligned}
\EE|c_{s+h}^{*,s}-c_s|^2&\le Ch,&
\EE|c_{s+h}^{*,s}-c_s|^4&\le Ch^2,\\
\EE\sup_{s\le t\le s+1}|c_t^{*,s}|^4&\le C,&
\left\|\EE(c_{s+h}^{*,s}-c_s\mid\mathcal F_s)\right\|_{L^2(\Omega)}
&\le Ch.
\end{aligned}
\label{eq:auxiliary_state_moments}
\end{equation}
For the actual correlation process,
\begin{equation}
\EE|\rho_{s+h}-\rho_s|^2\le Ch,
\qquad
\left\|\EE(\rho_{s+h}-\rho_s\mid\mathcal F_s)\right\|_{L^2(\Omega)}
\le Ch.
\label{eq:rho_increment_bounds}
\end{equation}
These correlation increment bounds require only
Assumption~\ref{ass:primitive_correlation_regularity}.
On $(J_{s,h}^{\rm st})^c$, $Z_t^{*,s}=Z_t^c$ throughout $[s,s+h]$.
\end{lemma}
\begin{proof}
\emph{Step 1: construction and the probability of a state mark.}
Both indicators in \eqref{eq:auxiliary_price} are predictable. In
particular, $\1_{\{u>\tau_s\}}c_{\tau_s-}^{1/2}$ is adapted and
left-continuous. The integral of its squared norm is finite on compact
intervals after localization. Independence of $B^*$ makes it orthogonal to $M$,
so the bracket density in \eqref{eq:auxiliary_price} is
$c_u$ before $\tau_s$ and $c_{\tau_s-}$ thereafter. This is $c^{*,s}$.
Before the first state mark, $c$ is continuous; freezing at its left
limit preserves continuity, positive semidefiniteness, and the variance
lower bound. The same construction preserves $|\rho^{*,s}|\le1$.

The indicator of $J_{s,h}^{\rm st}$ is bounded by
$N_{s+h}^{\rm st}-N_s^{\rm st}$. Compensation gives the first inequality
in \eqref{eq:state_jump_probability}. Conditional Jensen and the uniform
second moment of $\lambda^{\rm st}$ give
\[
\EE(\Lambda_{s,h}^{\rm st})^2
\le \frac1h\int_s^{s+h}\EE(\lambda_u^{\rm st})^2du\le C.
\]
On the event without a state mark, the first integral in
\eqref{eq:auxiliary_price} equals $M_t-M_s$ and the second is zero.
The original drift is retained, proving the asserted equality of prices.

\emph{Step 2: moments of the stopped continuous dynamics.}
For every covariance coordinate, the absence of state marks before
$\tau_s$ gives
\begin{equation}
c_t^{*,s,\wp}
=c_s^\wp+\int_s^t\1_{\{u\le\tau_s\}}b_u^{c,\wp}\,du
+\int_s^t\1_{\{u\le\tau_s\}}\,dN_u^{c,\wp}.
\label{eq:auxiliary_state_representation}
\end{equation}
The continuous integrals reach the left limit at $\tau_s$, so this
identity also holds for $t\ge\tau_s$. Jensen and continuous-martingale
isometry give second-moment orders $h^2$ and $h$ for its two increments.
Jensen and BDG similarly give
\[
\begin{aligned}
\EE\left|\int_s^{s+h}\1_{\{u\le\tau_s\}}b_u^{c,\wp}du\right|^4
&\le h^3\int_s^{s+h}\EE|b_u^{c,\wp}|^4du\le Ch^4,\\
\EE\sup_{s\le t\le s+h}
\left|\int_s^t\1_{\{u\le\tau_s\}}dN_u^{c,\wp}\right|^4
&\le C\EE\left(\int_s^{s+h}v_u^{c,\wp}du\right)^2\le Ch^2.
\end{aligned}
\]
Taking $h=1$ and using $\sup_s\EE|c_s|^4<C$ proves the supremum bound
in \eqref{eq:auxiliary_state_moments}. Thus the required moments at the
random freezing time follow from the stopped representation. The
martingale term is square integrable and has conditional mean zero.
Conditional contraction applied to the drift integral proves the
remaining conditional-mean bound.

The same stopped representation holds for $\rho^{*,s}$ with drift
$b^\rho$ and continuous martingale $\int\Sigma^\rho dB^\rho$. It gives
\[
\EE|\rho_{s+h}^{*,s}-\rho_s|^2\le Ch,
\qquad
\left\|\EE(\rho_{s+h}^{*,s}-\rho_s\mid\mathcal F_s)\right\|_{L^2(\Omega)}
\le Ch.
\]
The difference between the actual and auxiliary correlations satisfies
$|\rho_{s+h}-\rho_{s+h}^{*,s}|\le2\1_{J_{s,h}^{\rm st}}$. Hence
\eqref{eq:state_jump_probability} gives
$\EE|\rho_{s+h}-\rho_{s+h}^{*,s}|^2\le4\PP(J_{s,h}^{\rm st})\le Ch$.
For its conditional mean, the conditional probability bound gives
\[
\begin{aligned}
&\left\|\EE(\rho_{s+h}-\rho_{s+h}^{*,s}\mid\mathcal F_s)
\right\|_{L^2(\Omega)}\\
&\qquad\le
2\left\|\PP(J_{s,h}^{\rm st}\mid\mathcal F_s)\right\|_{L^2(\Omega)}
\le 2h\|\Lambda_{s,h}^{\rm st}\|_{L^2(\Omega)}\le Ch.
\end{aligned}
\]
Combining these bounds with those for $\rho^{*,s}$ proves
\eqref{eq:rho_increment_bounds}. To obtain the same conclusion under
Assumption~\ref{ass:primitive_correlation_regularity} alone, stop at the
first mark of $\mu^\rho$ after $s$. The argument then uses only the drift
$b^\rho$, bracket density $\|\Sigma^\rho\|^2$, and mark intensity
$F^\rho(E)\chi^\rho$.
\end{proof}

\subsection{Continuous returns and the blockwise bipower estimator}

\begin{lemma}
\label{lem:price_moments}
Under Assumptions~\ref{ass:V_volatility} and \ref{ass:C4_moments}, for
every finite $r\ge2$, uniformly in $i,\ell$ and
$\wp\in\{X,Y\}$,
\[
\EE|A_{i,\ell}^\wp|^r\le C_rn^{-r},\qquad
\EE|M_{i,\ell}^\wp|^r\le C_rn^{-r/2},\qquad
\EE|\Delta_{i,\ell}\wp^c|^r\le C_rn^{-r/2}.
\]
For fixed $a>0$, $p\ge0$, $q>0$, and $0<\varpi<1/2$,
\begin{equation}
\sup_{i,\ell,\wp}\EE\left[
|\Delta_{i,\ell}\wp^c|^p
\1\{|\Delta_{i,\ell}\wp^c|>an^{-\varpi}\}\right]
\le C_{a,p,q}n^{-q}.
\label{eq:large_continuous_return}
\end{equation}
\end{lemma}
\begin{proof}
The increment bounds follow from Jensen's inequality, the BDG inequality,
and the uniform moment bounds in Assumption~\ref{ass:C4_moments}.
For $m>\max\{p,2\}$, Markov's inequality gives
\[
\begin{aligned}
&\EE\left[|\Delta_{i,\ell}\wp^c|^p
\1\{|\Delta_{i,\ell}\wp^c|>an^{-\varpi}\}\right]\\
&\qquad\le a^{p-m}n^{\varpi(m-p)}
\EE|\Delta_{i,\ell}\wp^c|^m
\le C_{a,p,m}n^{\varpi(m-p)-m/2}.
\end{aligned}
\]
Choosing $m$ such that $m(1/2-\varpi)+p\varpi\ge q$ proves
\eqref{eq:large_continuous_return}.
\end{proof}

\begin{lemma}
\label{lem:pilot}
Under Assumptions~\ref{ass:V_volatility}--\ref{ass:C4_moments} and the
variance lower bound in Assumption~\ref{ass:positivity_normalization},
let $\widehat b_{i,j,n}^{\wp,c}$ be the blockwise bipower estimator based
on continuous-price returns, given by
\eqref{eq:localized_bipower_pilot} with $\Delta\wp$ replaced by
$\Delta\wp^c$. There exist $b_0,c_0>0$ and a fixed
$p_0\in(2,4]$ such that
\begin{equation}
\sup_{i,j,\wp}\PP\{\widehat b_{i,j,n}^{\wp,c}<b_0\}
\le C r_n,
\qquad r_n=e^{-c_0k_n}+n^{-1}+n^{-p_0/2}.
\label{eq:pilot_tail}
\end{equation}
Let $\mathcal E_{i,j,n}^{\rm fd}$ be the event that a block contains no
price jump, both
continuous-price bipower estimates are at least $b_0$, and a continuous
return exceeds its threshold in \eqref{eq:adaptive_localized_threshold}. Then,
for every $q>0$,
\begin{equation}
\sup_{i,j}\PP(\mathcal E_{i,j,n}^{\rm fd})\le C_qn^{-q}.
\label{eq:false_deletion}
\end{equation}
\end{lemma}
\begin{proof}
Fix a marginal component and select $h_n=\lfloor k_n/2\rfloor$
disjoint adjacent pairs of increments in a block. A marginal Brownian
representation is $dM_t^\wp=\sigma_t^\wp\,dB_t^\wp$. For an increment
starting at $s$, set
$G_s=\sqrt n\,\sigma_s^\wp(B_{s+1/n}^\wp-B_s^\wp)$.
Conditionally on $\mathcal F_s$, $G_s$ is Gaussian with variance
$c_s^\wp\ge\underline c$. Hence some $a_0,\pi_1>0$ satisfy
$\PP(|G_s|>2a_0\mid\mathcal F_s)\ge\pi_1$ uniformly. Let
$I_h^{\rm pair}$ indicate
that both frozen increments in pair $h$ exceed $2a_0$ in absolute value,
and let $\mathcal K_h$ be the sigma-field at its start. Condition first
on the information available at the start of the second increment.
The tower property then gives
$\EE(I_h^{\rm pair}\mid\mathcal K_h)\ge\pi_0:=\pi_1^2$. Thus
$I_h^{\rm pair}-\EE(I_h^{\rm pair}\mid\mathcal K_h)$ are bounded
martingale differences in
pair order, and Azuma's inequality gives
\begin{equation}
\PP\left\{\sum_{h=1}^{h_n}I_h^{\rm pair}<\pi_0h_n/2\right\}
\le Ce^{-c_0k_n}.
\label{eq:pilot_pairs}
\end{equation}
The normalized continuous return satisfies
$\sqrt n\Delta\wp^c=G_s+\varepsilon_s$, with freezing error
\nopagebreak[4]
\[
\varepsilon_s=\sqrt n\int_s^{s+1/n}a_u^\wp\,du
+\sqrt n\int_s^{s+1/n}(\sigma_u^\wp-\sigma_s^\wp)\,dB_u^\wp.
\]
Write the continuous volatility increment as
\[
\sigma_{s,u}^{\wp,c}
:=\int_s^u\widetilde a_v^\wp\,dv
+\int_s^u\widetilde\sigma_v^\wp\,dW_v
+\int_s^u\widetilde v_v^\wp\,d\widetilde W_v.
\]
Fix $p_0\in(2,4]$. Jensen and BDG give
$\EE|\sigma_{s,u}^{\wp,c}|^{p_0}\le C(u-s)^{p_0/2}$ for $u-s\le1$.
Define
\[
\varepsilon_s^c
=\sqrt n\int_s^{s+1/n}a_u^\wp\,du
+\sqrt n\int_s^{s+1/n}\sigma_{s,u}^{\wp,c}\,dB_u^\wp.
\]
Applying Jensen and BDG to these integrals, without conditioning on the
jump event, gives $\EE|\varepsilon_s^c|^{p_0}\le Cn^{-p_0/2}$.
The volatility integrands agree before the first volatility jump after
$s$. Stopping the two price integrals at that time shows that
$\varepsilon_s=\varepsilon_s^c$ on the event of no volatility jump in
$(s,s+1/n]$. Consequently,
\[
\EE\bigl[|\varepsilon_s|^{p_0}
\1\{\text{no volatility jump on }(s,s+1/n]\}\bigr]
\le \EE|\varepsilon_s^c|^{p_0}\le Cn^{-p_0/2}.
\]

Among the $2h_n$ selected increments, the expected number with
$|\varepsilon_s|>a_0$ and no volatility jump is at most
$Ck_nn^{-p_0/2}$. The expected number of intervals containing a volatility
jump is at most $Ck_n/n$, by compensation. Since $h_n$ is of order $k_n$,
Markov's inequality bounds the probability that at least $\pi_0h_n/4$
selected increments fall in either class by $C(n^{-p_0/2}+n^{-1})$.
Thus we control the proportion of affected increments, rather than
requiring the entire block to be free of volatility jumps.

Each affected increment can remove at most one of the disjoint pairs.
Outside this exceptional event and the event in \eqref{eq:pilot_pairs},
at least $\pi_0h_n/4$ pairs therefore have both actual normalized returns
at least $a_0$ in absolute value. Their contribution to the nonnegative
bipower sum gives
\[
\widehat b_{i,j,n}^{\wp,c}
\ge \frac{\pi a_0^2}{2(k_n-1)}\,\frac{\pi_0h_n}{4}
\ge \frac{\pi\pi_0a_0^2}{16},
\]
where $h_n/(k_n-1)\ge1/2$ for $k_n\ge2$. Taking
$b_0=\pi\pi_0a_0^2/16$ proves \eqref{eq:pilot_tail}.

On a price-jump-free block, observed and continuous returns, and hence
their bipower estimates, coincide. If the bipower estimates are at least
$b_0$, deleting a return requires
$|\Delta\wp^c|>\tau_\wp\sqrt{b_0}n^{-\varpi}$. Use
\eqref{eq:large_continuous_return} with $p=0$ and $q$ replaced by $q+1$,
then take a union bound over at most $2k_n\le2n$ component returns to
prove \eqref{eq:false_deletion}.
\end{proof}

\begin{lemma}
\label{lem:innovations}
Let $\mathsf M=(\mathsf M^X,\mathsf M^Y)^\top$ be a continuous local
martingale whose covariance density $\mathsf c$ satisfies
$\sup_t\EE|\mathsf c_t|^4\le C$ on a deterministic interval. For
deterministic $u<v$ in this interval with $v-u\le1$, write
$\mathsf M_{u,t}=\mathsf M_t-\mathsf M_u$ and define
\[
\begin{aligned}
Q_{u,v}^\wp
&=(\mathsf M_{u,v}^\wp)^2-\int_u^v\mathsf c_t^\wp dt,
&&\wp=X,Y,\\
Q_{u,v}^{XY}
&=\mathsf M_{u,v}^X\mathsf M_{u,v}^Y
-\int_u^v\mathsf c_t^{XY}dt.
\end{aligned}
\]
Then, for $\wp\in\{X,Y,XY\}$,
\begin{equation}
\EE(Q_{u,v}^\wp\mid\mathcal F_u)=0,
\qquad
\EE|Q_{u,v}^\wp|^p\le C_p(v-u)^p,
\quad p=2,4.
\label{eq:innovation_moments}
\end{equation}
This applies to the original price martingale and, uniformly in $s$, to
$M^{*,s}$ on $[s,s+1]$. For the auxiliary price increments on that
interval, it also gives
\begin{equation}
\begin{aligned}
&\EE|M_v^{*,s,\wp}-M_u^{*,s,\wp}|^r
+\EE|Z_v^{*,s,\wp}-Z_u^{*,s,\wp}|^r
\\
&\qquad\le C_r(v-u)^{r/2},\qquad r=2,4,8,
\end{aligned}
\label{eq:auxiliary_return_moments}
\end{equation}
for $\wp\in\{X,Y\}$.
\end{lemma}
\begin{proof}
For $p=1,2,4$, BDG and Jensen give
\[
\EE\sup_{u\le t\le v}|\mathsf M_{u,t}^\wp|^{2p}
\le C_p\EE\left(\int_u^v\mathsf c_t^\wp dt\right)^p
\le C_p(v-u)^p.
\]
The same bound holds for the $p$th moment of the integrated covariance.
The definitions of $Q$ and Cauchy--Schwarz for the cross product thus
give their stated moments. It\^o's formula gives
\[
Q_{u,v}^\wp=2\int_u^v\mathsf M_{u,t}^\wp d\mathsf M_t^\wp,
\qquad
Q_{u,v}^{XY}
=\int_u^v\mathsf M_{u,t}^Xd\mathsf M_t^Y
+\int_u^v\mathsf M_{u,t}^Yd\mathsf M_t^X.
\]
The preceding bounds and Cauchy--Schwarz yield
\[
\EE\int_u^v|\mathsf M_{u,t}^\alpha|^2\mathsf c_t^\beta\,dt
\le C(v-u)^2,\qquad \alpha,\beta\in\{X,Y\}.
\]
Thus all these stochastic integrals are square integrable and have
conditional mean zero given $\mathcal F_u$. The original covariance has
the required fourth moments by Assumption~\ref{ass:C4_moments}; for
$c^{*,s}$, use \eqref{eq:auxiliary_state_moments}. Finally, the original
drift retained in $Z^{*,s}$ has $r$th increment moment of order
$(v-u)^r$. Combining this with the martingale bound proves
\eqref{eq:auxiliary_return_moments}.
\end{proof}

\subsection{Finite-dimensional correlation expansion}

\begin{lemma}
\label{lem:ratio_expansion}
Let $h(x)=x_3/\sqrt{x_1x_2}$ on
$\mathcal D_a=\{x_1\ge a,x_2\ge a,x_3^2\le x_1x_2\}$, $a>0$.
Its gradient and Hessian are uniformly bounded on this set by a constant
depending only on $a$. For covariance-coordinate vectors $x,\widehat x$,
suppose that $x\in\mathcal D_a$ and the matrix represented by
$\widehat x$ is positive semidefinite. Define $h_0(\widehat x)$ to be the
correlation ratio when both diagonal entries are positive and zero
otherwise, and put $d=\widehat x-x$. Then
\begin{equation}
\begin{aligned}
h_0(\widehat x)-h(x)
&=\nabla h(x)^\top d+\operatorname{rem}(x,\widehat x),\\
|\operatorname{rem}(x,\widehat x)|&\le C_a|d|^2,
\end{aligned}
\label{eq:ratio_taylor}
\end{equation}
and $|h_0(\widehat x)-h(x)|^2\le C_a|d|^2$.
\end{lemma}
\begin{proof}
The gradient is
\[
\nabla h(x)=
\left(-\frac{x_3}{2x_1^{3/2}x_2^{1/2}},
-\frac{x_3}{2x_1^{1/2}x_2^{3/2}},
(x_1x_2)^{-1/2}\right)^\top.
\]
Differentiating once more and using $x_1,x_2\ge a$ and
$|x_3|\le\sqrt{x_1x_2}$ gives
\[
\sup_{x\in\mathcal D_a}
\bigl(|\nabla h(x)|+|\nabla^2h(x)|\bigr)\le C_a.
\]
If $\widehat x_1,\widehat x_2\ge a/2$, the segment joining
$x$ and $\widehat x$ lies in the convex set $\mathcal D_{a/2}$.
Taylor's theorem gives \eqref{eq:ratio_taylor}, and the mean-value
theorem gives $|h_0(\widehat x)-h(x)|\le C_a|d|$.

If either diagonal entry of $\widehat x$ is less than $a/2$, then
$|d|\ge a/2$. Both correlations have absolute value at most one, so
\[
|\operatorname{rem}(x,\widehat x)|
\le 2+C_a|d|\le C'_a|d|^2,
\qquad
|h_0(\widehat x)-h(x)|\le2\le C'_a|d|.
\]
These bounds include the zero-denominator case and complete the proof.
\end{proof}

\section{Consistency and functional convergence}
\label{app:consistency}

The high-frequency bounds below use
Assumptions~\ref{ass:J_price_jumps}--\ref{ass:primitive_correlation_regularity},
the variance lower bound, and the basic sampling conditions. The
polynomial sampling rates are used to make the mean-curve approximation
in Subsection~\ref{app:mean_approximation} negligible at the $T^{-1/2}$
scale. Stationarity enters in Subsection~\ref{app:latent_clt}, where we
study the latent daily curves, and in the main-result proofs in
Subsection~\ref{app:main_mean_proofs}.

\subsection{Comparison with the auxiliary continuous process}

On block $(i,j)$, put $s=s_{i,j}$ and use the process $Z^{*,s}$ from
Lemma~\ref{lem:auxiliary_state}. Write
$\Delta_{i,\ell}Z^*=Z^{*,s}_{t_{i,\ell}}-Z^{*,s}_{t_{i,\ell-1}}$
for $\ell\in\mathcal I_j^n$, and form the untruncated matrix
\[
\widehat c_{i,\kappa_j}^*
=\frac n{k_n}\sum_{\ell\in\mathcal I_j^n}
\Delta_{i,\ell}Z^*(\Delta_{i,\ell}Z^*)^\top.
\]
Use the zero-denominator convention in \eqref{eq:direct_local_ratio}
to form $\widehat\rho_{i,\kappa_j}^*$, and extend these ratios blockwise
to obtain $\widehat\rho_i^*$, with $|\widehat\rho_i^*|\le1$. The auxiliary
process is constructed separately for each block on the independent
extension in Lemma~\ref{lem:auxiliary_state}; expectations of auxiliary
quantities are taken on that space.

\begin{lemma}
\label{lem:auxiliary_comparison}
With $r_n$ as in \eqref{eq:pilot_tail}, for every fixed $M>0$,
\begin{equation}
\sup_i\EE\|\widehat\rho_i-\widehat\rho_i^*\|_{\mathcal H}^2
\le C_M(L_n+r_n+n^{-M}).
\label{eq:auxiliary_daily}
\end{equation}
Uniformly in $i,j$, the blockwise difference also satisfies
\begin{equation}
\left\|\EE(\widehat\rho_{i,\kappa_j}
-\widehat\rho_{i,\kappa_j}^*\mid\mathcal F_{s_{i,j}})
\right\|_{L^2(\Omega)}
\le C_M(L_n+\sqrt{r_n}+n^{-M/2}).
\label{eq:auxiliary_comparison_bias}
\end{equation}
\end{lemma}
\begin{proof}
Fix a block starting at $s$. Let $J_{i,j}$ be the event that the block
contains a price jump or a mark of one of the three state-jump measures.
Its indicator is bounded by the increment of the sum of the five
primitive counting measures, whose intensity is
\[
\chi_u^{\rm all}
=\sum_{\wp=X,Y}\{F^\wp(\RR)\chi_u^\wp
+\widetilde F^\wp(\RR)\widetilde\chi_u^\wp\}
+F^\rho(E)\chi_u^\rho.
\]
This domination also allows simultaneous marks. Since
$\sup_u\EE(\chi_u^{\rm all})^2<C$, compensation and conditional Jensen
give
\[
\PP(J_{i,j}\mid\mathcal F_s)\le L_n\Lambda_{i,j},
\qquad \sup_{i,j}\EE\Lambda_{i,j}^2<C,
\qquad
\Lambda_{i,j}=\frac1{L_n}\EE\left(
\int_s^{s+L_n}\chi_u^{\rm all}\,du\,\middle|\,\mathcal F_s\right).
\]
On $J_{i,j}^c$, the observed, original continuous, and auxiliary returns
coincide. If both continuous-price bipower estimates are at least $b_0$
and no false deletion occurs, the correlation estimates also coincide. Thus, for
$\delta_{i,j}=\widehat\rho_{i,\kappa_j}-\widehat\rho_{i,\kappa_j}^*$,
the bound $|\delta_{i,j}|\le2$ and Lemma~\ref{lem:pilot} imply
\[
\EE|\delta_{i,j}\1_{J_{i,j}^c}|^2\le C_M(r_n+n^{-M}),
\qquad
\EE|\delta_{i,j}|^2\le C_M(L_n+r_n+n^{-M}).
\]
Integration over the daily blocks proves \eqref{eq:auxiliary_daily}.
For the conditional mean, retain the conditional probability bound on
$J_{i,j}$ and use contraction on its complement:
\[
\begin{aligned}
\|\EE(\delta_{i,j}\mid\mathcal F_s)\|_{L^2(\Omega)}
&\le2\|\PP(J_{i,j}\mid\mathcal F_s)\|_{L^2(\Omega)}
+\{\EE|\delta_{i,j}\1_{J_{i,j}^c}|^2\}^{1/2}\\
&\le C_M(L_n+\sqrt{r_n}+n^{-M/2}).
\end{aligned}
\]
This proves \eqref{eq:auxiliary_comparison_bias}.
\end{proof}

\subsection{Auxiliary covariance errors and daily recovery}
\label{app:daily_moments}

Use the coordinate order $(X,Y,XY)$ to write the preceding matrix as
a vector $\widehat C_i^*(\kappa_j)$. For each coordinate, define
\[
\bar c_i^{\wp,*,n}(\kappa_j)
=\frac1{L_n}\int_s^{s+L_n}c_u^{*,s,\wp}\,du,
\qquad s=s_{i,j},
\]
and let $\bar C_i^{*,n}$ collect these three coordinates. Extend the
vectors blockwise and put
\begin{equation}
V_i^*=\bar C_i^{*,n}-C_i^{[n]},\qquad
D_i^*=\widehat C_i^*-\bar C_i^{*,n},\qquad
R_i^*=\widehat C_i^*-C_i^{[n]}=D_i^*+V_i^*.
\label{eq:auxiliary_covariance_errors}
\end{equation}
Thus $V_i^*$ measures variation of the auxiliary covariance within a
block, while $D_i^*$ is its realized-covariance estimation error.

\begin{lemma}
\label{lem:auxiliary_covariance_moments}
For the quantities in \eqref{eq:auxiliary_covariance_errors}, uniformly in $i,j$,
\begin{align}
\EE|D_i^*(\kappa_j)|^2&\le C(k_n^{-1}+n^{-1}),&
\EE|D_i^*(\kappa_j)|^4&\le C(k_n^{-2}+n^{-2}),
\label{eq:auxiliary_D_moments}\\
\EE|R_i^*(\kappa_j)|^2&\le C(k_n^{-1}+L_n+n^{-1}),&
\EE|R_i^*(\kappa_j)|^4&\le C(k_n^{-2}+L_n^2+n^{-2}).
\label{eq:auxiliary_R_moments}
\end{align}
Integrating each pointwise moment over $[0,1]$ preserves the same rate.
\end{lemma}
\begin{proof}
On a block starting at $s$, write $M_{i,\ell}^{\wp,*}$ for the increment
of $M^{*,s,\wp}$ and $Q_{i,\ell}^{\wp,*}$ for its centered covariance
innovation in Lemma~\ref{lem:innovations}. The drift increments
$A_{i,\ell}^\wp$ are unchanged. For a marginal coordinate,
\[
D_i^{\wp,*}(\kappa_j)=\frac n{k_n}\sum_{\ell\in\mathcal I_j^n}
\left\{(A_{i,\ell}^\wp)^2
+2A_{i,\ell}^\wp M_{i,\ell}^{\wp,*}
+Q_{i,\ell}^{\wp,*}\right\}.
\]
The cross coordinate has the corresponding terms
$A^XA^Y$, $A^XM^{Y,*}+A^YM^{X,*}$, and $Q^{XY,*}$.
Denote the three sums by $D_i^{\wp,*,1},D_i^{\wp,*,2},D_i^{\wp,*,3}$.
For $p=2,4$, Jensen, Cauchy--Schwarz, and the increment bounds in
Lemmas~\ref{lem:price_moments} and \ref{lem:innovations} give
\[
\EE|D_i^{\wp,*,1}(\kappa_j)|^p\le C_pn^{-p},
\qquad
\EE|D_i^{\wp,*,2}(\kappa_j)|^p\le C_pn^{-p/2}.
\]
For the mixed products, only auxiliary increment moments up to order
$2p\le8$ are used.

The innovations in the third sum are martingale differences with
respect to the observation-time filtration on the extended space.
Their orthogonality and \eqref{eq:innovation_moments} give
\[
\EE|D_i^{\wp,*,3}(\kappa_j)|^2\le Ck_n^{-1}.
\]
For the fourth moment, discrete martingale BDG and Cauchy--Schwarz give
\[
\begin{aligned}
\EE|D_i^{\wp,*,3}(\kappa_j)|^4
&\le C\left(\frac n{k_n}\right)^4
\EE\left\{\sum_{\ell\in\mathcal I_j^n}
(Q_{i,\ell}^{\wp,*})^2\right\}^2\\
&\le C\left(\frac n{k_n}\right)^4
k_n\sum_{\ell\in\mathcal I_j^n}\EE|Q_{i,\ell}^{\wp,*}|^4
\le Ck_n^{-2}.
\end{aligned}
\]
These estimates prove \eqref{eq:auxiliary_D_moments}. Jensen and
\eqref{eq:auxiliary_state_moments} give
\[
\EE|V_i^*(\kappa_j)|^2\le CL_n,\qquad
\EE|V_i^*(\kappa_j)|^4\le CL_n^2.
\]
Together with $R_i^*=D_i^*+V_i^*$, this proves
\eqref{eq:auxiliary_R_moments}. The integrated bounds follow because
all estimates are uniform in the block index.
\end{proof}

\begin{proof}[Proof of Proposition~\ref{prop:daily_recovery}]
Decompose the daily error as
\[
\widehat\rho_i-\rho_i
=(\widehat\rho_i-\widehat\rho_i^*)
+(\widehat\rho_i^*-\rho_i^{[n]})+(\rho_i^{[n]}-\rho_i).
\]
The first term is controlled by \eqref{eq:auxiliary_daily}.
For the second, Lemma~\ref{lem:ratio_expansion}, applied to
$C_i^{[n]}$ and $\widehat C_i^*$, and
\eqref{eq:auxiliary_R_moments} give an expected squared norm of order
$k_n^{-1}+L_n+n^{-1}$. For the third,
\eqref{eq:rho_increment_bounds} gives
\[
\EE\|\rho_i^{[n]}-\rho_i\|_{\mathcal H}^2
\le C\sum_j\int_{\kappa_j}^{\kappa_{j+1}}(\kappa-\kappa_j)\,d\kappa
\le CL_n.
\]
For $M\ge1$, the inequalities $e^{-c_0k_n}\le C/k_n$,
$n^{-1}\le k_n^{-1}$, and $p_0>2$ imply
$r_n+n^{-M}=O(k_n^{-1})$. Combining the three bounds proves
\eqref{eq:daily_recovery}. Averaging and applying Markov's inequality
gives the stated $O_P(k_n^{-1}+L_n)$ rate. Under the polynomial sampling
conditions, $\sqrt T(k_n^{-1}+L_n)\to0$, which proves the final assertion.
\end{proof}

\subsection{Conditional local bias}

To control the sum of estimation errors across days, we also bound the
conditional mean of the local estimation error given the information
at the start of each block.

\begin{lemma}
\label{lem:conditional_bias}
For a fixed $M>0$, let
\[
\beta_n=k_n^{-1}+L_n+n^{-1/2}+\sqrt{r_n}+n^{-M/2},
\]
where $r_n$ is defined in \eqref{eq:pilot_tail}. Uniformly in $i,j$ and
$\kappa\in[\kappa_j,\kappa_{j+1})$,
\begin{equation}
\left\|\EE\{\widehat\rho_{i,\kappa_j}-\rho_i(\kappa)
\mid\mathcal F_{s_{i,j}}\}\right\|_{L^2(\Omega)}
\le C_M\beta_n.
\label{eq:conditional_bias}
\end{equation}
\end{lemma}
\begin{proof}
Fix a block and write $s=s_{i,j}$. In the decomposition from
Lemma~\ref{lem:auxiliary_covariance_moments}, each covariance innovation
has conditional mean zero at the preceding observation time. The tower
property therefore gives
$\EE(D_i^{\wp,*,3}(\kappa_j)\mid\mathcal F_s)=0$.
The second-moment bounds for the other two terms imply
\[
\|\EE(D_i^*(\kappa_j)\mid\mathcal F_s)\|_{L^2(\Omega)}
\le C(n^{-1}+n^{-1/2}).
\]
For the within-block covariance variation, conditional Fubini and
\eqref{eq:auxiliary_state_moments} give
\[
\|\EE(V_i^*(\kappa_j)\mid\mathcal F_s)\|_{L^2(\Omega)}
\le \frac1{L_n}\int_s^{s+L_n}C(u-s)\,du\le CL_n.
\]
Apply Lemma~\ref{lem:ratio_expansion} with
$x=C_i^{[n]}(\kappa_j)$ and $\widehat x=\widehat C_i^*(\kappa_j)$.
Since $\nabla h(x)$ is bounded and $\mathcal F_s$-measurable, the
preceding estimates control the conditional mean of the linear term.
For the quadratic remainder, the fourth-moment bound in
\eqref{eq:auxiliary_R_moments} yields
\begin{equation}
\left\|\EE\{|R_i^*(\kappa_j)|^2\mid\mathcal F_s\}\right\|_{L^2(\Omega)}
\le \{\EE|R_i^*(\kappa_j)|^4\}^{1/2}
\le C(k_n^{-1}+L_n+n^{-1}).
\label{eq:auxiliary_quadratic_bias}
\end{equation}
Consequently,
\[
\begin{aligned}
&\|\EE(\widehat\rho_{i,\kappa_j}^*-\rho_s
\mid\mathcal F_s)\|_{L^2(\Omega)}\\
&\quad\le C\|\EE(R_i^*(\kappa_j)\mid\mathcal F_s)\|_{L^2(\Omega)}
+C\left\|\EE\{|R_i^*(\kappa_j)|^2\mid\mathcal F_s\}\right\|_{L^2(\Omega)}\\
&\quad\le C(k_n^{-1}+L_n+n^{-1/2}).
\end{aligned}
\]
Decompose the estimation error as
\[
\widehat\rho_{i,\kappa_j}-\rho_i(\kappa)
=(\widehat\rho_{i,\kappa_j}-\widehat\rho_{i,\kappa_j}^*)
+(\widehat\rho_{i,\kappa_j}^*-\rho_s)
+(\rho_s-\rho_i(\kappa)).
\]
The first conditional mean is bounded by
\eqref{eq:auxiliary_comparison_bias}, the second by the preceding
display, and the third by $CL_n$ in $L^2(\Omega)$, using
\eqref{eq:rho_increment_bounds}. Their sum is bounded by $C_M\beta_n$.
\end{proof}

\subsection{Approximation of the mean curve}
\label{app:mean_approximation}

\begin{lemma}
\label{lem:mean_approximation}
Under the polynomial sampling conditions,
\[
\left\|\frac1{\sqrt T}\sum_i(\widehat\rho_i-\rho_i)\right\|_{\mathcal H}
=o_P(1).
\]
\end{lemma}
\begin{proof}
Let $e_i=\widehat\rho_i-\rho_i$ and
$b_i=\EE_{\mathcal H}(e_i\mid\mathcal F_{\underline t_i})$. Write
\[
\frac1{\sqrt T}\sum_i e_i
=\frac1{\sqrt T}\sum_i(e_i-b_i)+\frac1{\sqrt T}\sum_i b_i.
\]
The first sum is controlled by daily recovery. Indeed, $e_i$ is
$\mathcal F_{\underline t_i+1}$-measurable and
$\underline t_i+1\le\underline t_{i+1}$. Thus, for $i<j$,
$e_i-b_i$ is $\mathcal F_{\underline t_j}$-measurable, and
\[
\begin{aligned}
\EE\langle e_i-b_i,e_j-b_j\rangle_{\mathcal H}
&=\EE\left\langle e_i-b_i,
\EE_{\mathcal H}(e_j-b_j\mid\mathcal F_{\underline t_j})
\right\rangle_{\mathcal H}\\
&=0.
\end{aligned}
\]
Proposition~\ref{prop:daily_recovery} therefore gives
\[
\EE\left\|\frac1{\sqrt T}\sum_i(e_i-b_i)\right\|_{\mathcal H}^2
=\frac1T\sum_i\EE\|e_i-b_i\|_{\mathcal H}^2
\le\frac1T\sum_i\EE\|e_i\|_{\mathcal H}^2
\le C(k_n^{-1}+L_n).
\]
For the second sum, use Lemma~\ref{lem:conditional_bias}. Since
$|e_i(\kappa)|\le2$, conditional Fubini identifies a representative of
$b_i$ with $\EE(e_i(\kappa)\mid\mathcal F_{\underline t_i})$
almost everywhere. On each block,
$\mathcal F_{\underline t_i}\subseteq\mathcal F_{s_{i,j}}$, so the tower
property and contraction transfer \eqref{eq:conditional_bias} to the
session start. Hence, uniformly in $i$,
\[
\EE\|b_i\|_{\mathcal H}^2
=\int_0^1\EE\left|\EE(e_i(\kappa)
\mid\mathcal F_{\underline t_i})\right|^2\,d\kappa
\le C_M\beta_n^2.
\]
The triangle inequality now gives
\[
\left\|\frac1{\sqrt T}\sum_i e_i\right\|_{L^2(\Omega;\mathcal H)}
\le C_M\{(k_n^{-1}+L_n)^{1/2}+\sqrt T\,\beta_n\}.
\]
Expanding $\beta_n$ and applying Markov's inequality gives the order
\begin{equation}
O_P\left(
 k_n^{-1/2}+\sqrt{L_n}+\sqrt{T/n}
 +\sqrt T/k_n+\sqrt T L_n+\sqrt{Tr_n}
 +\sqrt T n^{-M/2}\right).
\label{eq:HF_rate}
\end{equation}
Under $T\asymp n^c$ and $k_n\asymp n^d$, the restrictions $c<1$ and
$c/2<d<1-c/2$ make the polynomial terms vanish. The definition of
$r_n$, together with $c<1$ and $p_0>2$, gives $Tr_n\to0$; choosing
$M>c$ controls the final term. Hence \eqref{eq:HF_rate} is $o_P(1)$.
\end{proof}

\subsection{Latent daily curves and their limit}
\label{app:latent_clt}

\begin{lemma}
\label{lem:latent_covariance}
Under Assumptions~\ref{ass:primitive_correlation_regularity} and
\ref{ass:C5_daily_mixing}, the daily curves form a strictly stationary,
ergodic sequence of $\mathcal H$-valued random elements. Its strong mixing
coefficients satisfy
\[
\alpha_\rho(h)\le
\alpha_{\mathsf Z}\bigl(h(1+g_{\mathrm{nt}})-1\bigr),
\qquad h\ge2,
\]
and $\|\rho_i\|_{\mathcal H}\le1$ almost surely.
If $q\ge3$, then
\begin{equation}
\sum_{h\in\mathbb Z}\|\Gamma_{\rho,h}\|_1<\infty.
\label{eq:trace_lag_sum}
\end{equation}
Consequently, $\mathcal C_\rho$ is a nonnegative, self-adjoint,
trace-class operator with kernel \eqref{eq:rho_long_run_kernel}.
\end{lemma}
\begin{proof}
Each bounded c\`adl\`ag daily curve is a measurable $\mathcal H$-valued
function of its state segment. Strict stationarity follows from shifts
by integer multiples of $1+g_{\mathrm{nt}}$ in the stationary state.
Past daily curves through day $i$
are measurable with respect to $\mathcal G_{\underline t_i+1}$, whereas
curves from day $i+h$ onward depend only on state times at or after
$\underline t_{i+h}$. The time gap between these state segments is
\[
\underline t_{i+h}-(\underline t_i+1)
=h(1+g_{\mathrm{nt}})-1.
\]
This gives the stated mixing bound, with rate $O(h^{-q-\iota})$;
mixing also gives ergodicity.

To prove trace-norm summability, let $P_m$ be the orthogonal projection
onto functions constant on $m$ equal intervals. For each such interval
$I$, Jensen gives
\[
\int_I\left|f(u)-\frac1{|I|}\int_If(v)\,dv\right|^2du
\le\frac1{|I|}\int_I\!\int_I|f(u)-f(v)|^2du\,dv.
\]
Apply this to $f=A_i^\rho$. By centering and
\eqref{eq:rho_increment_bounds},
$\EE|A_i^\rho(u)-A_i^\rho(v)|^2
\le\EE|\rho_i(u)-\rho_i(v)|^2\le C|u-v|$.
Summing over the intervals yields
\begin{equation}
\EE\|(I_{\mathcal H}-P_m)A_i^\rho\|_{\mathcal H}^2\le C/m.
\label{eq:latent_projection}
\end{equation}
Here $I_{\mathcal H}$ is the identity on $\mathcal H$.
Write $\|\cdot\|_{\HS}$ for the Hilbert--Schmidt norm.
For $h\ge2$, the kernel of $\Gamma_{\rho,h}$ is
$\Cov\{A_0^\rho(u),A_h^\rho(v)\}$. Boundedness of the centered curves
and the scalar mixing covariance inequality give
\[
\|\Gamma_{\rho,h}\|_{\HS}\le C\alpha_\rho(h).
\]
Moreover,
\[
\Gamma_{\rho,h}-P_m\Gamma_{\rho,h}P_m
=\EE\{(I_{\mathcal H}-P_m)A_0^\rho\otimes A_h^\rho\}
+\EE\{P_mA_0^\rho\otimes(I_{\mathcal H}-P_m)A_h^\rho\}.
\]
The rank-one trace-norm identity, Cauchy--Schwarz, and
\eqref{eq:latent_projection} imply
\[
\|\Gamma_{\rho,h}-P_m\Gamma_{\rho,h}P_m\|_1\le Cm^{-1/2}.
\]
Since $P_m\Gamma_{\rho,h}P_m$ has rank at most $m$,
\[
\|P_m\Gamma_{\rho,h}P_m\|_1
\le\sqrt m\,\|P_m\Gamma_{\rho,h}P_m\|_{\HS}
\le\sqrt m\,\|\Gamma_{\rho,h}\|_{\HS}.
\]
Combining the last three bounds gives
\[
\|\Gamma_{\rho,h}\|_1
\le C\{\sqrt m\,\alpha_\rho(h)+m^{-1/2}\}.
\]
For $\alpha_\rho(h)>0$, choose
$m=\lceil\alpha_\rho(h)^{-1}\rceil$; when $\alpha_\rho(h)=0$, let
$m\to\infty$. In both cases,
$\|\Gamma_{\rho,h}\|_1\le C\alpha_\rho(h)^{1/2}$.
This is summable for $q\ge3$. For $h=0,1$, the trace norms are finite
because $\|A_i^\rho\|_{\mathcal H}\le2$. The bounds for negative lags follow from
$\Gamma_{\rho,-h}=\Gamma_{\rho,h}^*$.

The covariance operator of $T^{-1/2}\sum_iA_i^\rho$ is
\[
\sum_{|h|<T}(1-|h|/T)\Gamma_{\rho,h}.
\]
It converges in trace norm to $\mathcal C_\rho$ by absolute summability,
so $\mathcal C_\rho$ is nonnegative and self-adjoint. The summable
pointwise bound $C\alpha_\rho(h)$ also permits interchange of the lag
sum and the kernel integrals, giving \eqref{eq:rho_long_run_kernel}.
\end{proof}

\begin{lemma}
\label{lem:bounded_mixing_hilbert_clt}
Let $(\upsilon_i)_{i\in\mathbb Z}$ be a centered, strictly stationary sequence
of $\mathcal H$-valued random elements, and define
\[
\alpha_\upsilon(h)
=\alpha\bigl(\sigma(\upsilon_i:i\le0),
\sigma(\upsilon_i:i\ge h)\bigr).
\]
Suppose that $\|\upsilon_0\|_{\mathcal H}\le K$ almost surely for some
$K<\infty$ and
\[
\sum_{h\ge1}\alpha_\upsilon(h)<\infty.
\]
Then
\[
\frac1{\sqrt T}\sum_{i=1}^T \upsilon_i
\Rightarrow G_\upsilon
\quad\text{in }\mathcal H,
\]
where $G_\upsilon$ is a centered Gaussian element. If, in addition,
\[
\sum_{h\in\mathbb Z}\|\Gamma_h\|_1<\infty,
\qquad
\Gamma_h=\EE(\upsilon_0\otimes \upsilon_h),
\]
then the covariance operator of $G_\upsilon$ is
\[
\mathcal C=\sum_{h\in\mathbb Z}\Gamma_h.
\]
\end{lemma}
\begin{proof}
Apply Corollary~2 of \citet{DedeckerMerlevede2003HilbertCCLT}. Set
\[
\mathcal M_0=\sigma(\upsilon_i:i\le0),
\qquad
\alpha_{\upsilon,1}(h)
=\alpha(\mathcal M_0,\sigma(\upsilon_h)).
\]
Then $\alpha_{\upsilon,1}(h)\le\alpha_\upsilon(h)$. The upper-tail
quantile function $q_\upsilon$ of $\|\upsilon_0\|_{\mathcal H}$ satisfies
$q_\upsilon\le K$, so
\[
\sum_{h\ge1}\int_0^{\alpha_{\upsilon,1}(h)}q_\upsilon(u)^2\,du
\le K^2\sum_{h\ge1}\alpha_\upsilon(h)<\infty.
\]
This verifies condition $(\alpha)$ of that corollary. Its functional
conditional central limit theorem, evaluated at time $1$, and Remark~1
of the same reference give the stated Gaussian limit; the covariance is
deterministic because strong mixing implies ergodicity.

Under the additional trace-norm summability condition, the covariance
operators of the scaled partial sums satisfy
\[
\sum_{|h|<T}(1-|h|/T)\Gamma_h
\longrightarrow\sum_{h\in\mathbb Z}\Gamma_h
\quad\text{in trace norm}.
\]
The coordinate-covariance formula in the cited result then gives this
sum as the covariance operator of $G_\upsilon$.
\end{proof}

\subsection{Proofs of consistency and the functional central limit theorem}
\label{app:main_mean_proofs}

\begin{proof}[Proof of Theorem~\ref{thm:finite_jump_consistency}]
By Lemma~\ref{lem:latent_covariance}, $(\rho_i)$ is stationary, ergodic,
and bounded in $\mathcal H$. The Hilbert-valued ergodic theorem gives
$\bar\rho_T\to r_u$ almost surely in $\mathcal H$; see
\citet{PopStojanovic1972}. By Jensen and
Proposition~\ref{prop:daily_recovery},
\[
\|\widetilde\rho_u-\bar\rho_T\|_{\mathcal H}^2
\le\frac1T\sum_i\|\widehat\rho_i-\rho_i\|_{\mathcal H}^2=o_P(1),
\]
which proves consistency of the unnormalized estimator. If $\bar r_u\ne0$, continuity of $\ell$
gives $|\ell(\widetilde\rho_u)|\ge|\bar r_u|/2$ with probability tending
to one. Since $\varepsilon_T\downarrow0$, the estimator in
\eqref{eq:direct_normalized_estimator} equals $\Phi(\widetilde\rho_u)$
on this event for all sufficiently large $T$. Continuity of $\Phi$ at
$r_u$, established in Lemma~\ref{lem:normalization}, proves normalized
consistency.
\end{proof}

\begin{proof}[Proof of Theorem~\ref{thm:finite_jump_fclt_unnormalized}]
Write
\[
\sqrt T(\widetilde\rho_u-r_u)
=\frac1{\sqrt T}\sum_{i=1}^T A_i^\rho
+\sqrt T(\widetilde\rho_u-\bar\rho_T).
\]
The second term is $o_P(1)$ in $\mathcal H$ by
Lemma~\ref{lem:mean_approximation} under
Assumption~\ref{ass:rates_finite_jump}(ii), proving
\eqref{eq:mean_expansion}.
On a two-sided stationary extension,
Lemma~\ref{lem:latent_covariance} gives strict stationarity,
$\|A_i^\rho\|_{\mathcal H}\le2$, summable mixing coefficients for
$q\ge3$, and \eqref{eq:trace_lag_sum}.
Lemma~\ref{lem:bounded_mixing_hilbert_clt} therefore yields
$T^{-1/2}\sum_iA_i^\rho\Rightarrow G_\rho$, with covariance
$\mathcal C_\rho$. Slutsky's theorem proves \eqref{eq:raw_fclt}; the
operator and kernel assertions follow from
Lemma~\ref{lem:latent_covariance}.
\end{proof}

\section{Normalization and finite aggregation}
\label{app:normalization}

\begin{lemma}
\label{lem:normalization}
The map $\Phi$ in \eqref{eq:normalization_map} is Fr\'echet
differentiable on its open domain $\mathcal D_\Phi$, with derivative
\eqref{eq:normalization_derivative_section2}. For every
$f\in\mathcal D_\Phi$ and $g\in\mathcal H$ satisfying
$\|g\|_{\mathcal H}\le|\ell(f)|/2$,
\begin{equation}
\|\Phi(f+g)-\Phi(f)-D\Phi_f[g]\|_{\mathcal H}
\le\frac{2\{|\ell(f)|+\|f\|_{\mathcal H}\}}{|\ell(f)|^3}
\|g\|_{\mathcal H}^2.
\label{eq:normalization_remainder}
\end{equation}
Moreover, if $\bar r_u\ne0$, there is a constant $C=C(r_u)<\infty$
such that, whenever
$\|f-r_u\|_{\mathcal H}\le\min\{1,|\bar r_u|/2\}$,
\begin{equation}
\|D\Phi_f-D\Phi_{r_u}\|_{\op}
\le C\|f-r_u\|_{\mathcal H}.
\label{eq:derivative_lipschitz}
\end{equation}
\end{lemma}
\begin{proof}
Since $\ell(g)=\langle\1,g\rangle_{\mathcal H}$, the functional $\ell$
has norm one and $\mathcal D_\Phi$ is open. The derivative in
\eqref{eq:normalization_derivative_section2} is linear and satisfies
\[
\|D\Phi_f[g]\|_{\mathcal H}
\le\left(|\ell(f)|^{-1}
+\|f\|_{\mathcal H}|\ell(f)|^{-2}\right)\|g\|_{\mathcal H}.
\]
For $\|g\|_{\mathcal H}\le|\ell(f)|/2$, we have
$|\ell(f+g)|\ge|\ell(f)|/2$, and direct algebra gives
\[
\Phi(f+g)-\Phi(f)-D\Phi_f[g]
=-\frac{\ell(g)}{\ell(f)^2\{\ell(f)+\ell(g)\}}
\{\ell(f)g-f\ell(g)\}.
\]
Taking norms and using $|\ell(g)|\le\|g\|_{\mathcal H}$ proves
\eqref{eq:normalization_remainder} and hence Fr\'echet differentiability.

If $\|f-r_u\|_{\mathcal H}\le\min\{1,|\bar r_u|/2\}$, then
$|\ell(f)|\ge|\bar r_u|/2$ and
$\|f\|_{\mathcal H}\le\|r_u\|_{\mathcal H}+1$.
The derivative formula and
$|\ell(f)-\bar r_u|\le\|f-r_u\|_{\mathcal H}$ therefore give
\[
\begin{aligned}
\|D\Phi_f-D\Phi_{r_u}\|_{\op}
&\le\left|\frac1{\ell(f)}-\frac1{\bar r_u}\right|
+\left\|\frac{f}{\ell(f)^2}
-\frac{r_u}{\bar r_u^2}\right\|_{\mathcal H}\\
&\le C\|f-r_u\|_{\mathcal H},
\end{aligned}
\]
which proves \eqref{eq:derivative_lipschitz}.
\end{proof}

\begin{lemma}
\label{lem:normalization_covariance}
Under the assumptions of Lemma~\ref{lem:latent_covariance}, with $q\ge3$
and $\bar r_u\ne0$, the lag operators of
$\xi_i=D\Phi_{r_u}[A_i^\rho]$ are absolutely summable in trace norm.
Their sum satisfies
$\mathcal C_k=D\Phi_{r_u}\mathcal C_\rho D\Phi_{r_u}^*$ and has kernel
\eqref{eq:Ck_finite_jump_theory} almost everywhere.
\end{lemma}
\begin{proof}
Write
\[
\begin{aligned}
\mathsf D_u=D\Phi_{r_u}
&=\bar r_u^{-1}I_{\mathcal H}-\bar r_u^{-2}(r_u\otimes\1),\\
\mathsf D_u^*
&=\bar r_u^{-1}I_{\mathcal H}-\bar r_u^{-2}(\1\otimes r_u).
\end{aligned}
\]
Since $\xi_i=\mathsf D_uA_i^\rho$,
\[
\EE(\xi_0\otimes\xi_h)
=\mathsf D_u\Gamma_{\rho,h}\mathsf D_u^*,
\qquad h\in\mathbb Z.
\]
Lemma~\ref{lem:latent_covariance} and the trace-norm inequality give
\[
\sum_h\|\EE(\xi_0\otimes\xi_h)\|_1
\le\|\mathsf D_u\|_{\op}^2
\sum_h\|\Gamma_{\rho,h}\|_1<\infty.
\]
Summing over $h$ shows that the covariance operator in
\eqref{eq:long_run_covariance_normalized} equals
$\mathsf D_u\mathcal C_\rho\mathsf D_u^*$.
Expanding this operator yields
\begin{align*}
\mathsf D_u\mathcal C_\rho\mathsf D_u^*
={}&\bar r_u^{-2}\mathcal C_\rho
-\bar r_u^{-3}\{(r_u\otimes\1)\mathcal C_\rho
+\mathcal C_\rho(\1\otimes r_u)\}\\
&+\bar r_u^{-4}(r_u\otimes\1)\mathcal C_\rho(\1\otimes r_u).
\end{align*}
The rank-one convention gives the four kernel terms in
\eqref{eq:Ck_finite_jump_theory}. The summable pointwise covariance
bound in Lemma~\ref{lem:latent_covariance} justifies the integrations
and their interchange with the lag sum, so the kernel identity holds
almost everywhere.
\end{proof}

\begin{proof}[Proof of Theorem~\ref{thm:finite_jump_fclt_normalized}]
Put $h_T=\widetilde\rho_u-r_u$. The functional CLT for the unnormalized mean curve gives
$\sqrt T h_T\Rightarrow G_\rho$ and
$\|h_T\|_{\mathcal H}=O_P(T^{-1/2})$.
The event $\mathcal E_T=\{\|h_T\|_{\mathcal H}\le|\bar r_u|/2\}$
has probability tending to one. Since $\varepsilon_T\downarrow0$,
on $\mathcal E_T$ we have
$|\ell(\widetilde\rho_u)|\ge|\bar r_u|/2>\varepsilon_T$ for all
sufficiently large $T$, and hence $\widehat k_u^\circ=\Phi(r_u+h_T)$.
On this event, Lemma~\ref{lem:normalization} bounds the norm of the
scaled remainder by
$C\sqrt T\|h_T\|_{\mathcal H}^2=O_P(T^{-1/2})$.
Since $\PP(\mathcal E_T)\to1$, it follows that
\[
\sqrt T(\widehat k_u^\circ-k_u^\circ)
=D\Phi_{r_u}[\sqrt T h_T]+o_P(1)
\Rightarrow D\Phi_{r_u}[G_\rho]
\quad\text{in }\mathcal H,
\]
where the convergence uses boundedness of $D\Phi_{r_u}$ and Slutsky's
theorem. The limit $G_k=D\Phi_{r_u}[G_\rho]$ is centered Gaussian.
Lemma~\ref{lem:normalization_covariance} gives its covariance
\eqref{eq:long_run_covariance_normalized} and kernel
\eqref{eq:Ck_finite_jump_theory}.
\end{proof}

\begin{proof}[Proof of Corollary~\ref{cor:fixed_aggregation}]
The componentwise recovery bounds and convexity give
\[
\sup_i\EE\|\widehat\rho_i^{(w)}-\rho_i^{(w)}\|_{\mathcal H}^2
\le\sum_aw_a\sup_i\EE
\|\widehat\rho_i^{(a)}-\rho_i^{(a)}\|_{\mathcal H}^2
\le C(k_n^{-1}+L_n).
\]
The consistency assertion for the unnormalized mean curve in
Theorem~\ref{thm:finite_jump_consistency}, applied componentwise, gives
consistency of the aggregate mean.

In the separable product space $\mathcal H^P$, the centered vector
$(A_i^{(1)},\ldots,A_i^{(P)})$ is bounded and inherits the common mixing
condition. The argument of Lemma~\ref{lem:bounded_mixing_hilbert_clt}
in this space gives a joint Gaussian limit. Since $P$ is fixed, the
componentwise high-frequency approximations also hold in the product
norm. Applying the bounded linear map
$(f_1,\ldots,f_P)\mapsto\sum_aw_af_a$ gives the aggregate approximation
and functional limit.

Under the common mixing condition, the projection argument in
Lemma~\ref{lem:latent_covariance} applies to each pair $a,b$ and gives
\[
\sum_{h\in\mathbb Z}
\|\EE(A_0^{(a)}\otimes A_h^{(b)})\|_1<\infty.
\]
Expanding the outer product of the weighted sums and summing over $h$
gives \eqref{eq:aggregate_covariance}, including all cross-covariance
terms. If $\ell(r_u^{(w)})\ne0$, Lemma~\ref{lem:normalization} and the
preceding normalized-limit argument give the normalized conclusions
for the aggregate mean.

For feasible inference, the recovery bound and the sampling rates imply
\[
\frac1T\sum_{i=1}^T
\|\widehat\rho_i^{(w)}-\rho_i^{(w)}\|_{\mathcal H}^2
=O_P(k_n^{-1}+L_n)=o_P(T^{-1/2}).
\]
Together with the aggregate functional CLT, boundedness and mixing of
the centered latent aggregate, and trace-norm summability of its lag
covariances, this verifies the conditions used in Supplementary
Appendix~\ref{app:hac}. The covariance and calibration arguments in
Supplementary Appendices~\ref{app:hac} and \ref{app:calibration} then
give the feasible procedures under their respective nondegeneracy
conditions.
\end{proof}

\section{Long-run covariance estimation and inference}
\label{app:inference}

\subsection{Consistency of the covariance estimators}
\label{app:hac}

Under the assumptions of Theorem~\ref{thm:hac}, use the notation
\[
(U_i,\mathcal C,\ell_T)
=(A_i^\rho,\mathcal C_\rho,\ell_T^\rho)
\quad\text{or}\quad
(\xi_i,\mathcal C_k,\ell_T^k),
\qquad \xi_i=D\Phi_{r_u}[A_i^\rho].
\]
The condition $\bar r_u\ne0$ is imposed only in the normalized case.
Write $\Gamma_h=\EE(U_0\otimes U_h)$ and
$\mathcal C=\sum_{h\in\mathbb Z}\Gamma_h$. In either case, $(U_i)$ is
centered and strictly stationary, inherits the mixing rate of $(\rho_i)$,
and has representatives satisfying $\|U_i\|_\infty\le C$ almost surely.
Moreover,
\begin{equation}
\sum_h\|\Gamma_h\|_1<\infty,
\qquad
\sum_h|\EE\langle U_0,U_h\rangle_{\mathcal H}|<\infty.
\label{eq:generic_lag_sums}
\end{equation}
For the unnormalized residuals these properties follow from
Lemma~\ref{lem:latent_covariance}. In the normalized case, boundedness
and mixing are preserved by the fixed map $D\Phi_{r_u}$, and the lag
sums follow from Lemma~\ref{lem:normalization_covariance}.

\begin{lemma}
\label{lem:latent_hac}
If $\ell_T\to\infty$ and $\ell_T^2/T\to0$, then
\begin{equation}
\widetilde{\mathcal C}_{U,T}
=\frac1T\sum_{i,j=1}^T\omega((i-j)/\ell_T)U_i\otimes U_j
\quad\text{satisfies}\quad
\|\widetilde{\mathcal C}_{U,T}-\mathcal C\|_1\xrightarrow{P}0.
\label{eq:latent_hac}
\end{equation}
\end{lemma}
\begin{proof}
Write $\mathcal S_T=\widetilde{\mathcal C}_{U,T}$ and
$\mathcal B_T=\EE\mathcal S_T$. Stationarity gives
\[
\mathcal B_T
=\sum_{|h|<T}(1-|h|/T)\omega(h/\ell_T)\Gamma_h.
\]
Since $\omega$ is bounded and continuous at zero with $\omega(0)=1$,
\eqref{eq:generic_lag_sums} and dominated convergence give
\begin{equation}
\|\mathcal B_T-\mathcal C\|_1\longrightarrow0.
\label{eq:hac_bias}
\end{equation}

For a fixed $f\in\mathcal H$, put
$X_i=\langle U_i,f\rangle_{\mathcal H}$,
$\gamma_h=\EE(X_0X_h)$, and $Y_{i,h}=X_iX_{i+h}$ for $h\ge0$.
The process $(X_i)$ is centered and uniformly bounded. We establish
\begin{equation}
\sup_{h\ge0}\sum_{s\in\mathbb Z}
|\Cov(Y_{0,h},Y_{s,h})|\le C_f.
\label{eq:hac_fixed_lag_sum}
\end{equation}
Set $\beta(0)=\beta(1)=1$ and $\beta(s)=\alpha_\rho(s)$ for $s\ge2$.
The bounded-variable mixing inequality gives
$|\gamma_h|\le C_f\beta(h)$. If $s\ge h$, the two lag products are
separated by $s-h$, so
$|\Cov(Y_{0,h},Y_{s,h})|\le C_f\beta(s-h)$.
For $0\le s\le h/2$ and $h>0$, use the identity
\[
\Cov(Y_{0,h},Y_{s,h})
=\Cov(X_0X_s,X_hX_{h+s})+\gamma_s^2-\gamma_h^2.
\]
The resulting bound is
$C_f\{\beta(h-s)+\beta(s)^2+\beta(h)^2\}$.
For $h/2<s<h$, centering $X_0$ gives
\[
|\EE(X_0X_sX_hX_{h+s})|
=|\Cov(X_0,X_sX_hX_{h+s})|\le C_f\beta(s),
\]
and hence the bound $C_f\{\beta(s)+\beta(h)^2\}$.
Summing over the three ranges yields, for $h\ge1$,
\[
\sum_{s\ge0}|\Cov(Y_{0,h},Y_{s,h})|
\le C_f\left\{\sum_{s\ge0}\bigl[\beta(s)+\beta(s)^2\bigr]
+h\beta(h)^2\right\}\le C_f,
\]
where the last bound is uniform in $h$ by the maintained polynomial
mixing rate. The case $h=0$ uses the first bound alone; stationarity
and covariance symmetry handle negative $s$. This proves
\eqref{eq:hac_fixed_lag_sum}.

It follows that
\begin{equation}
\Var\left\{\frac1T\sum_{i=1}^{T-h}Y_{i,h}\right\}
\le C_f/T,\qquad 0\le h<T.
\label{eq:hac_direction_variance}
\end{equation}
There are $O(\ell_T)$ nonzero lag weights. Minkowski's inequality gives
\[
\bigl\{\EE|\langle(\mathcal S_T-\mathcal B_T)f,f\rangle_{\mathcal H}|^2
\bigr\}^{1/2}\le C_f\ell_T/\sqrt T\longrightarrow0.
\]
For any fixed finite-rank orthogonal projection $P$, polarization gives
convergence for every matrix entry on its range. Equivalence of
finite-dimensional norms and \eqref{eq:hac_bias} then imply
\[
\EE\|P(\mathcal S_T-\mathcal C)P\|_1\longrightarrow0.
\]

To pass to trace norm on $\mathcal H$, we use the nonnegativity of
$\mathcal S_T$ supplied by the kernel condition. For a nonnegative
trace-class operator $\mathcal S$, let $P$ be a finite-rank orthogonal
projection and $Q=I_{\mathcal H}-P$. The off-diagonal block satisfies
\[
\begin{aligned}
\|P\mathcal S Q\|_1
&=\|(P\mathcal S^{1/2})(\mathcal S^{1/2}Q)\|_1\\
&\le\{\tr(P\mathcal S P)\tr(Q\mathcal S Q)\}^{1/2}.
\end{aligned}
\]
Using $\mathcal S-P\mathcal S P
=Q\mathcal S Q+P\mathcal S Q+Q\mathcal S P$ gives
\begin{equation}
\|\mathcal S-P\mathcal S P\|_1
\le\tr(Q\mathcal S Q)
+2\{\tr(\mathcal S)\tr(Q\mathcal S Q)\}^{1/2}.
\label{eq:positive_tail}
\end{equation}
The limit $\mathcal C$ is nonnegative by
Lemmas~\ref{lem:latent_covariance} and \ref{lem:normalization_covariance}.
Suppose first that $\mathcal C\ne0$. Let $P_m$ project onto the span of
$m$ orthonormal eigenfunctions corresponding to its largest strictly
positive eigenvalues, and put
$Q_m=I_{\mathcal H}-P_m$ and $\tau_m=\tr(Q_m\mathcal C Q_m)$.
If $\mathcal C$ has finite rank, take all the corresponding orthonormal
eigenfunctions once $m$ reaches that rank. For each fixed $m$, \eqref{eq:hac_bias} gives
\[
\EE\tr(Q_m\mathcal S_TQ_m)\longrightarrow\tau_m,
\qquad \EE\tr(\mathcal S_T)\longrightarrow\tr(\mathcal C).
\]
Apply \eqref{eq:positive_tail} and Cauchy--Schwarz. Combining these
limits with the fixed-projection convergence and
$\|\mathcal C-P_m\mathcal C P_m\|_1=\tau_m$ gives
\[
\limsup_{T\to\infty}\EE\|\mathcal S_T-\mathcal C\|_1
\le2\tau_m+2\{\tr(\mathcal C)\tau_m\}^{1/2}.
\]
Letting $m\to\infty$ makes the right side vanish. If $\mathcal C=0$,
nonnegativity instead gives
$\EE\|\mathcal S_T\|_1=\tr(\mathcal B_T)\to0$ directly.
Markov's inequality proves \eqref{eq:latent_hac}.
\end{proof}

\begin{lemma}
\label{lem:residual_recovery}
For $U_i=A_i^\rho$ and $\widehat U_i=\widehat A_i^\rho$,
\begin{equation}
a_{T,n}:=\frac1T\sum_i\|\widehat U_i-U_i\|_{\mathcal H}^2
=o_P(T^{-1/2}).
\label{eq:residual_rate}
\end{equation}
\end{lemma}
\begin{proof}
Put $e_i^\rho=\widehat\rho_i-\rho_i$ and
$h_T=\widetilde\rho_u-r_u=O_P(T^{-1/2})$ in $\mathcal H$.
Then $\widehat A_i^\rho-A_i^\rho=e_i^\rho-h_T$, whence
\[
\frac1T\sum_i\|\widehat A_i^\rho-A_i^\rho\|_{\mathcal H}^2
\le\frac2T\sum_i\|e_i^\rho\|_{\mathcal H}^2
+2\|h_T\|_{\mathcal H}^2
=o_P(T^{-1/2})
\]
by Proposition~\ref{prop:daily_recovery} and the functional CLT.
\end{proof}

\begin{lemma}
\label{lem:hac_replacement}
If $a_{T,n}$ satisfies \eqref{eq:residual_rate} and
$\ell_T=o(T^{1/2})$, then
\begin{equation}
\left\|\frac1T\sum_{i,j}\omega((i-j)/\ell_T)
(\widehat U_i\otimes\widehat U_j-U_i\otimes U_j)\right\|_1=o_P(1).
\label{eq:replacement}
\end{equation}
\end{lemma}
\begin{proof}
Put $d_i=\widehat U_i-U_i$ and
$\mathbf W=[\omega((i-j)/\ell_T)]_{i,j=1}^T$.
Under the kernel assumptions, $\mathbf W$ is positive semidefinite.
Its absolute row sums give $\|\mathbf W\|_{\op}\le C\ell_T$.
Factor the lag-window form through $\mathbf W^{1/2}$ by defining
$\mathsf L_U,\mathsf L_d:\RR^T\to\mathcal H$ as
\[
\mathsf L_Uz=T^{-1/2}\sum_i(\mathbf W^{1/2}z)_iU_i,
\qquad
\mathsf L_dz=T^{-1/2}\sum_i(\mathbf W^{1/2}z)_id_i.
\]
For $\mathsf L_d$, we have
\[
\begin{aligned}
\|\mathsf L_d\|_{\HS}^2
&=\frac1T\sum_{i,j}W_{ij}\langle d_i,d_j\rangle_{\mathcal H}\\
&\le\|\mathbf W\|_{\op}a_{T,n}
\le C\ell_Ta_{T,n}=o_P(1).
\end{aligned}
\]
We also have
$\mathsf L_U\mathsf L_U^*=\widetilde{\mathcal C}_{U,T}$.
Stationarity and \eqref{eq:generic_lag_sums} give
\[
\EE\tr(\widetilde{\mathcal C}_{U,T})
\le C\sum_{h\in\mathbb Z}|\tr(\Gamma_h)|<\infty
\]
uniformly in $T$. Nonnegativity and Markov's inequality imply
$\|\mathsf L_U\|_{\HS}^2
=\tr(\widetilde{\mathcal C}_{U,T})=O_P(1)$.
The operator in \eqref{eq:replacement} is
\[
(\mathsf L_U+\mathsf L_d)(\mathsf L_U+\mathsf L_d)^*
-\mathsf L_U\mathsf L_U^*
=\mathsf L_U\mathsf L_d^*+\mathsf L_d\mathsf L_U^*
+\mathsf L_d\mathsf L_d^*.
\]
Thus the Hilbert--Schmidt product inequality yields
\begin{equation}
\begin{aligned}
&\|\mathsf L_U\mathsf L_d^*+\mathsf L_d\mathsf L_U^*
+\mathsf L_d\mathsf L_d^*\|_1\\
&\qquad\le 2\|\mathsf L_U\|_{\HS}\|\mathsf L_d\|_{\HS}
+\|\mathsf L_d\|_{\HS}^2=o_P(1).
\end{aligned}
\qedhere
\label{eq:hac_replacement_bound}
\end{equation}
\end{proof}

\begin{proof}[Proof of Theorem~\ref{thm:hac}]
Apply Lemmas~\ref{lem:latent_hac}, \ref{lem:residual_recovery},
and \ref{lem:hac_replacement} to $U_i=A_i^\rho$ with bandwidth
$\ell_T^\rho$. The triangle inequality proves
\eqref{eq:raw_hac_consistency}. The same argument with bandwidth
$\ell_T^k$ gives
$\|\widehat{\mathcal C}_\rho[\ell_T^k]-\mathcal C_\rho\|_1\to_P0$.

For the normalized estimator, consistency of $\widetilde\rho_u$,
$\bar r_u\ne0$, and $\varepsilon_T\downarrow0$ give
$\PP\{|\widehat{\bar r}_u|>\varepsilon_T\}\to1$.
Lemma~\ref{lem:normalization} gives
$\|\widehat{\mathsf D}_u-\mathsf D_u\|_{\op}=o_P(1)$ and
$\|\widehat{\mathsf D}_u\|_{\op}=O_P(1)$, where
$\mathsf D_u=D\Phi_{r_u}$. The exact relation
\eqref{eq:normalized_hac_conjugation} and
$\mathcal C_k=\mathsf D_u\mathcal C_\rho\mathsf D_u^*$ yield
\begin{align*}
\|\widehat{\mathcal C}_k-\mathcal C_k\|_1
\le{}&\|\widehat{\mathsf D}_u\|_{\op}^2
\|\widehat{\mathcal C}_\rho[\ell_T^k]-\mathcal C_\rho\|_1\\
&+\|\mathcal C_\rho\|_1
(\|\widehat{\mathsf D}_u\|_{\op}+\|\mathsf D_u\|_{\op})
\|\widehat{\mathsf D}_u-\mathsf D_u\|_{\op}
=o_P(1),
\end{align*}
which proves \eqref{eq:normalized_hac_consistency}.
\end{proof}

\begin{proof}[Proof of Corollary~\ref{cor:linear_inference}]
The functional CLT under $f\mapsto\langle f,g\rangle_{\mathcal H}$ gives
the scalar normal limit with variance $v_g$. Moreover,
\[
|\widehat v_g-v_g|
\le\|\widehat{\mathcal C}_\rho-\mathcal C_\rho\|_1
\|g\|_{\mathcal H}^2\xrightarrow{P}0
\]
by Theorem~\ref{thm:hac}. If $v_g>0$, Slutsky's theorem gives the
studentized limit and coverage. For the normalized functional, use
$\mathcal C_k$, $\widehat{\mathcal C}_k$, and $v_{k,g}>0$ in the same
argument.
\end{proof}

\subsection{Eigenvalues, conditional laws, and critical values}
\label{app:calibration}

\begin{lemma}
\label{lem:calibration}
Suppose $\widehat{\mathcal C}$ and $\mathcal C$ are nonnegative
self-adjoint trace-class operators and
$\|\widehat{\mathcal C}-\mathcal C\|_1\to_P0$.
Let $(\widehat\lambda_j)$ and $(\lambda_j)$ be their ordered
nonnegative eigenvalues, with multiplicity and completed by zeros. Then
\begin{equation}
\sum_{j\ge1}|\widehat\lambda_j-\lambda_j|\xrightarrow{P}0.
\label{eq:spectral_l1}
\end{equation}
Let $(\eta_j)$ be an i.i.d. sequence of standard normal random variables,
independent of the data. Write $\PP_\eta$ and $\EE_\eta$ for probability
and expectation with respect to these variables, and set
\[
\widehat{\mathcal Q}=\sum_j\widehat\lambda_j\eta_j^2,
\qquad
\mathcal Q=\sum_j\lambda_j\eta_j^2.
\]
The conditional law of $\widehat{\mathcal Q}$ given the data converges
weakly in probability to the law of $\mathcal Q$.
If $\tr(\mathcal C)>0$, the distribution function $F$ of
$\mathcal Q$ is
continuous on $\RR$ and strictly increasing on $(0,\infty)$.
Its $(1-\alpha)$ quantile is positive and unique for every
$\alpha\in(0,1)$, and the corresponding conditional quantile of
$\widehat{\mathcal Q}$ converges to it in probability.
\end{lemma}
\begin{proof}
Trace-norm convergence implies operator-norm and trace convergence.
The variational characterization therefore gives
$\widehat\lambda_j\to_P\lambda_j$ for every fixed $j$.
For each fixed $m$, nonnegativity of the eigenvalues yields
\[
\begin{aligned}
\sum_{j\ge1}|\widehat\lambda_j-\lambda_j|
\le{}&2\sum_{j\le m}|\widehat\lambda_j-\lambda_j|
+|\tr(\widehat{\mathcal C})-\tr(\mathcal C)|\\
&+2\sum_{j>m}\lambda_j.
\end{aligned}
\]
The first two terms tend to zero in probability. Taking $m$ sufficiently
large makes the last term arbitrarily small, proving
\eqref{eq:spectral_l1}.

Use the same sequence $(\eta_j)$ in both sums. Both are
almost surely finite by nonnegativity and finite trace, and
\[
\EE_\eta(|\widehat{\mathcal Q}-\mathcal Q|\mid\text{data})
\le\sum_j|\widehat\lambda_j-\lambda_j|\xrightarrow{P}0.
\]
The bounded Lipschitz criterion now gives conditional weak convergence
in probability.

Suppose $\tr(\mathcal C)>0$. Then $\lambda_1>0$, and
\[
\mathcal Q=\lambda_1\eta_1^2+R_\lambda,
\qquad R_\lambda=\sum_{j\ge2}\lambda_j\eta_j^2,
\]
where $R_\lambda$ is independent of $\eta_1$. Since
$\lambda_1\eta_1^2$ has no atoms, $F$ is continuous.
For every $\epsilon>0$, choose $m\ge2$ so that
$\sum_{j>m}\lambda_j<\epsilon/4$. Markov's inequality gives
\[
\PP_\eta\left(\sum_{j>m}\lambda_j\eta_j^2<\epsilon/2\right)>1/2,
\]
whereas the finite sum $\sum_{j=2}^m\lambda_j\eta_j^2$ is less than
$\epsilon/2$ with positive probability. Independence of these sums
therefore gives $\PP_\eta(R_\lambda<\epsilon)>0$.
For $0<a<b$, conditional on any $R_\lambda<a$, the positive density of
$\lambda_1\eta_1^2$ on $(0,\infty)$ gives positive probability to
$(a-R_\lambda,b-R_\lambda)$. Thus
$\PP_\eta(a<\mathcal Q<b)>0$, so $F$ is strictly increasing on
$(0,\infty)$. Applying the same finite-sum and tail argument to
$\mathcal Q$ gives $F(x)>0$ for $x>0$; the unbounded support of
$\lambda_1\eta_1^2$ gives $F(x)<1$ for every finite $x$.

Let $c>0$ be the unique population $(1-\alpha)$ quantile, and set
$\widehat F_T(x)=\PP_\eta(\widehat{\mathcal Q}\le x\mid\text{data})$.
Conditional weak convergence gives $\widehat F_T(x)\to_P F(x)$ at
every fixed $x$. For every $\delta>0$,
\[
F(c-\delta)<1-\alpha<F(c+\delta),
\]
with the first inequality also valid when $c-\delta\le0$.
Consequently, with probability tending to one,
\[
\widehat F_T(c-\delta)<1-\alpha<\widehat F_T(c+\delta).
\]
These inequalities place the conditional quantile between $c-\delta$
and $c+\delta$, proving its convergence.
\end{proof}

\begin{proof}[Proof of Corollary~\ref{cor:confidence_sets}]
The continuous mapping theorem and Gaussian spectral expansion give
\[
T\|\widetilde\rho_u-r_u\|_{\mathcal H}^2
\Rightarrow\sum_j\lambda_{\rho,j}\eta_j^2.
\]
Theorem~\ref{thm:hac} and Lemma~\ref{lem:calibration} give
$\widehat c_{\rho,1-\alpha}\to_P c_{\rho,1-\alpha}$ when
$\tr(\mathcal C_\rho)>0$. Continuity of the limiting distribution at
this quantile gives coverage $1-\alpha$ asymptotically.
The estimated operator has range in the $m_n$-dimensional step space,
so \eqref{eq:gaussian_calibration} includes all its nonzero eigenvalues.
The normalized case follows in the same way when $\tr(\mathcal C_k)>0$.
\end{proof}

\begin{proof}[Proof of Corollary~\ref{cor:flatness} and
Remark~\ref{rem:projection}]
For the normalized test, consistency, $\bar r_u\ne0$, and
$\varepsilon_T\downarrow0$ give
$\PP\{|\widehat{\bar r}_u|>\varepsilon_T\}\to1$ under the null and
every fixed alternative under consideration. Not rejecting when
$|\widehat{\bar r}_u|\le\varepsilon_T$ therefore has no asymptotic effect.
Under the null, $k_u^\circ=\1$,
so the normalized functional CLT and continuity of the squared norm
give the stated limiting law. Lemma~\ref{lem:calibration} gives exact
asymptotic size when $\tr(\mathcal C_k)>0$.

Under a fixed alternative, consistency gives
$Q_T^{\mathrm{flat}}/T\to_P\|k_u^\circ-\1\|_{\mathcal H}^2>0$.
By conditional Markov's inequality,
\[
\widehat c_{k,1-\alpha}
\le\alpha^{-1}\EE_\eta(\widehat{\mathcal Q}_k\mid\text{data})
=\alpha^{-1}\tr(\widehat{\mathcal C}_k)=O_P(1).
\]
The statistic diverges and the critical value remains bounded in
probability, proving consistency. This argument does not require
$\tr(\mathcal C_k)>0$ under the alternative.

For Remark~\ref{rem:projection}, apply the bounded self-adjoint
projection $P_\perp$ to the functional CLT for the unnormalized mean curve. The covariance estimator
is trace-norm consistent because
\[
\|P_\perp(\widehat{\mathcal C}_\rho-\mathcal C_\rho)P_\perp\|_1
\le\|\widehat{\mathcal C}_\rho-\mathcal C_\rho\|_1\xrightarrow{P}0.
\]
Lemma~\ref{lem:calibration} then gives exact asymptotic size when
$\tr(P_\perp\mathcal C_\rho P_\perp)>0$.
At $r_u=\bar r_u\1\ne0$, \eqref{eq:normalization_derivative_section2}
gives $D\Phi_{r_u}=\bar r_u^{-1}P_\perp$.
\end{proof}

\section{Population identities and examples}
\label{app:population}

This appendix compares $r_u(\kappa)=\EE\rho_i(\kappa)$ with the
population correlation ratio formed from the mean covariance coordinates.
Expectations in the population ratios are taken at a fixed time of day,
and the mean marginal variances are assumed positive.

\subsection{A population identity}

Let $\mathcal V_i(\kappa)=\sqrt{c_i^X(\kappa)c_i^Y(\kappa)}$.
Under the maintained moment conditions, the expectations below are finite.
Define
\[
r_{\CF}(\kappa)
=\frac{\EE c_i^{XY}(\kappa)}
 {\sqrt{\EE c_i^X(\kappa)\,\EE c_i^Y(\kappa)}}.
\]
Since $c_i^{XY}=\rho_i\mathcal V_i$, the covariance decomposition gives
\begin{align}
r_{\CF}(\kappa)
={}&\frac{\EE \mathcal V_i(\kappa)}
 {\sqrt{\EE c_i^X(\kappa)\,\EE c_i^Y(\kappa)}}\,r_u(\kappa)\notag\\
&+\frac{\Cov\{\rho_i(\kappa),\mathcal V_i(\kappa)\}}
 {\sqrt{\EE c_i^X(\kappa)\,\EE c_i^Y(\kappa)}}.
\label{eq:population_identity}
\end{align}
By Cauchy--Schwarz, the coefficient of $r_u(\kappa)$ is at most one,
but it may vary with $\kappa$. Thus
$\Cov\{\rho_i(\kappa),\mathcal V_i(\kappa)\}=0$ at every $\kappa$ does not
ensure equal integral-normalized shapes, even when both normalizations
are defined.

\subsection{A common multiplicative specification}

\begin{proposition}
\label{prop:population_proportionality}
Suppose
\[
\sigma_i^X(\kappa)=s_X(\kappa)\Lambda_i^X(\kappa),\qquad
\sigma_i^Y(\kappa)=s_Y(\kappa)\Lambda_i^Y(\kappa),\qquad
\rho_i(\kappa)=k(\kappa)R_i(\kappa),
\]
where $s_X,s_Y$ are positive deterministic functions, $k$ is deterministic,
and the joint distribution of
$(\Lambda_i^X(\kappa),\Lambda_i^Y(\kappa),R_i(\kappa))$ does not depend on
$i$ or $\kappa$. Suppose the random variables in the expectations defining
$a$ and $b$ below are integrable and the two denominator moments are finite
and strictly positive. Then
\[
r_u(\kappa)=a\,k(\kappa),\qquad
r_{\CF}(\kappa)=b\,k(\kappa),
\]
where
\[
a=\EE R_i(\kappa),\qquad
b=\frac{\EE[\Lambda_i^X(\kappa)\Lambda_i^Y(\kappa)R_i(\kappa)]}
 {\sqrt{\EE[\Lambda_i^X(\kappa)^2]\,\EE[\Lambda_i^Y(\kappa)^2]}}.
\]
If $a,b$, and $\int_0^1k(s)\,ds$ are nonzero, the two curves have the
same integral-normalized shape, namely
$k(\kappa)/\int_0^1k(s)\,ds$. The unnormalized curves coincide when $a=b$;
if $k\not\equiv0$, the converse also holds.
\end{proposition}
\begin{proof}
Taking expectations of the correlation representation gives the first
formula. Substitution of the two volatility representations in
$r_{\CF}$ cancels $s_Xs_Y$ from its numerator and denominator, giving
the second. Since the joint distribution does not depend on $i$ or
$\kappa$, neither do $a$ and $b$.
Division by the respective signed integrals gives the shape assertion.
\end{proof}

No independence among $\Lambda_i^X$, $\Lambda_i^Y$, and $R_i$ is required.
If $a=0$ or $b=0$, the corresponding curve is identically zero and cannot
be integral-normalized.

\subsection{Different conclusions about intraday constancy}

Take $g_{\mathrm{nt}}=0$, so that $\underline t_i=i-1$. The two examples
below use the same volatility process.

Let $V_t$ be a stationary Ornstein--Uhlenbeck process satisfying
$dV_t=-\lambda V_t\,dt+\sqrt{2\lambda}\,dB_t$, $\lambda>0$, with
$V_0\sim N(0,1)$. Set $U_t=\tanh(V_t)$,
$m_2=\EE(U_t^2)\in(0,1)$,
$\varphi(\kappa)=\sin(2\pi\kappa)$, and
$v(\kappa)=1+0.4\cos(2\pi\kappa)$.
For $|\delta|<1$, define, at $t=i-1+\kappa$,
\begin{equation}
c_i^X(\kappa)=c_i^Y(\kappa)
=v(\kappa)\{1+\delta\varphi(\kappa)U_t\}.
\label{eq:example_variances}
\end{equation}
The two variances are bounded below by $0.6(1-|\delta|)>0$ and
$\EE U_t=0$ by symmetry. Choose nonzero constants $r_0,\gamma,\delta$
with $|r_0|+|\gamma|(1+|\delta|)<1$.

\begin{example}[A constant mean correlation curve]
\label{ex:constant_ED}
For $\rho_i(\kappa)=r_0+\gamma U_t$, the mean of the daily correlation
curves is constant, while the covariance-first curve is not:
\[
r_u(\kappa)=r_0,
\qquad
r_{\CF}(\kappa)=r_0+\delta\gamma m_2\varphi(\kappa).
\]
\end{example}
\begin{proof}
The equality $\EE U_t=0$ gives $r_u(\kappa)=r_0$.
The common variances imply $\mathcal V_i(\kappa)=c_i^X(\kappa)=c_i^Y(\kappa)$,
and
\[
\begin{aligned}
\EE\mathcal V_i(\kappa)&=v(\kappa),\\
\Cov\{\rho_i(\kappa),\mathcal V_i(\kappa)\}
&=v(\kappa)\delta\gamma m_2\varphi(\kappa).
\end{aligned}
\]
The coefficient of $r_u(\kappa)$ in \eqref{eq:population_identity} is
therefore one. Substitution gives the stated $r_{\CF}$, whose sine
coefficient $\delta\gamma m_2$ is nonzero.
\end{proof}

\begin{example}[A constant covariance-first curve]
\label{ex:constant_CF}
With the same variances, let
\[
\rho_i(\kappa)=r_0+\gamma U_t-\delta\gamma m_2\varphi(\kappa).
\]
Then
\[
r_u(\kappa)=r_0-\delta\gamma m_2\varphi(\kappa),
\qquad r_{\CF}(\kappa)=r_0.
\]
\end{example}
\begin{proof}
The deterministic adjustment to Example~\ref{ex:constant_ED} subtracts
$\delta\gamma m_2\varphi(\kappa)$ from $r_u(\kappa)$ and leaves the
covariance with $\mathcal V_i(\kappa)$ unchanged. Since the coefficient
of $r_u(\kappa)$ in \eqref{eq:population_identity} is one,
\[
r_{\CF}(\kappa)=r_u(\kappa)+\delta\gamma m_2\varphi(\kappa)=r_0.
\]
The mean curve $r_u$ is nonconstant because $\delta\gamma m_2\ne0$.
\end{proof}

In both examples, each target integrates to $r_0\ne0$ over $[0,1]$.
Integral normalization therefore preserves the different constancy
conclusions. The parameter restriction keeps both correlation processes
strictly inside $(-1,1)$.

The deterministic intraday functions are smooth and one-periodic, so
they match across day boundaries. For $U_t=\tanh(V_t)$, It\^o's formula
gives the diffusion and drift coefficients
\[
\sqrt{2\lambda}\operatorname{sech}^2(V_t),\qquad
-\lambda V_t\operatorname{sech}^2(V_t)
-2\lambda\tanh(V_t)\operatorname{sech}^2(V_t),
\]
respectively. Both are bounded. The square roots of the variances in
\eqref{eq:example_variances} are smooth functions of $(\kappa,U_t)$,
with bounded derivatives on $[0,1]\times[-1,1]$. Thus the volatility and
correlation processes have bounded drift and diffusion coefficients and
no state jumps.

For the stationarity and mixing condition, take $\mathsf Z_t=V_t$.
The mappings to the variance and correlation coordinates depend on
$\kappa$ but not on $i$.
The OU transition law is $N(e^{-\lambda h}x,1-e^{-2\lambda h})$;
averaged over $x\sim N(0,1)$, its total-variation distance from $N(0,1)$
is $O(e^{-\lambda h})$. The Markov property bounds dependence between
separated past and future sigma-fields by this distance. Hence the
required polynomial mixing rate holds for any fixed $q$.

Take two independent Brownian motions $W^X,W^Y$, independent of the
state process $V$, and define the continuous price parts by
\[
\begin{aligned}
dX_t^c&=\sigma_t^X\,dW_t^X,\\
dY_t^c&=\sigma_t^Y\{\rho_t\,dW_t^X
+\sqrt{1-\rho_t^2}\,dW_t^Y\}.
\end{aligned}
\]
These prices have zero drift and the prescribed spot covariance.
Independent compound Poisson price jumps with constant finite intensities
may be added. The primitive moment and nondegeneracy conditions then hold.

\section{Additional Monte Carlo evidence}
\label{suppsec:simulations}

This appendix gives implementation details and supplementary results for
Section~\ref{sec:simulations}, including recovery errors, confidence-set
coverage, flatness diagnostics, and the paired covariance-first
comparison.

\subsection{Implementation}
\label{suppsec:simulation_implementation}

For the design in Section~\ref{subsec:simulation_design}, each OU state is
advanced by its exact transition on a one-second grid with $23{,}400$
cells per day. The final state of one day is the initial state of the
next. The covariance coordinates are evaluated at the left endpoint of
each cell. Conditional on these values, continuous returns are generated
as follows, with cell indices suppressed and $\Delta t=1/23{,}400$:
\begin{align*}
 \Delta X&=\sqrt{c^X\Delta t}\,\epsilon_X,\\
 \Delta Y&=\sqrt{c^Y\Delta t}
 \{\rho\epsilon_X+(1-\rho^2)^{1/2}\epsilon_Y\},
\end{align*}
where $\epsilon_X$ and $\epsilon_Y$ are independent standard normals,
independent of the state innovations. The variances lie in $[0.45,1.75]$.
Price jumps are independent across assets and of the continuous drivers;
multiple jumps within one cell are summed. Coarser returns are obtained
by summing the one-second returns. Paired comparisons use the same latent
and Brownian paths.

The estimator and threshold parameters are those in
Section~\ref{subsec:simulation_design}. Integrals of estimated step curves
use weight $1/m_n=L_n$ on each block. The population target is known
analytically; it and the latent sample mean are evaluated on the
one-second left grid. For daily recovery, we separate estimation error
from the error of the orthogonal block-average projection, as defined in
\suppref{suppsec:ring1_full}. This projection is used for evaluation and
is distinct from the left-endpoint approximation in the proofs.

For confidence sets and flatness tests, the unnormalized and normalized HAC
estimates use their respective residual curves and denominator $T$.
The primary Parzen bandwidth rule in Section~\ref{subsec:simulation_design}
gives $4,5,6$ in the joint-growth design and $2,4,6$ in the
fixed-frequency design. Operator matrices include the integration
weights. Numerically negligible eigenvalues are set to zero according
to the prescribed floating-point tolerance; materially negative
eigenvalues are not removed by this rule.

Each replication has its own HAC estimate and conditional critical
value. Gaussian squared norms are simulated using $4{,}999$ draws from
a separate calibration random-number stream. Critical values use the
higher empirical quantile, and rejection is strict. For a realized
squared-norm statistic $Q_{\mathrm{obs}}$ and its $b$th simulated
counterpart $Q_b^*$, the reported Monte Carlo $p$-value is
\[
 \widehat p=\frac{1+\#\{Q_b^*\geq Q_{\mathrm{obs}}\}}{4{,}999+1}.
\]
Within each replication, paired calculations on the same intraday grid
reuse the underlying Gaussian draws. Each procedure
uses its own covariance estimate to obtain the corresponding critical
value.

Bandwidths, block sizes, parameter combinations, and reporting thresholds
were fixed before the formal coverage and rejection runs. The replication
counts for the main experiments are given in
Section~\ref{subsec:simulation_design}. The zero-integral coverage
experiment uses $2{,}000$ replications for each configuration. The covariance-first
comparison uses $5{,}000$ under the null and $2{,}000$ under fixed
alternatives. An entry of $1.0000$ with zero plug-in Monte Carlo standard
error (MCSE) means that every simulated sample rejected, not that
population power is known to equal one.

The reported calculations had no local denominator failures, estimator
boundary corrections, or materially negative covariance eigenvalues.
The normalization cutoff was never triggered in calculations of the
normalized estimator.

\subsection{Complete recovery results}
\label{suppsec:ring1_full}

Let $\Pi_n$ denote the orthogonal projection onto blockwise constant
curves, obtained by averaging the latent curve within each block. Since
$\widehat\rho_i$ lies in this space, the daily squared error decomposes as
\[
 \|\widehat\rho_i-\rho_i\|_{\mathcal H}^2
 =\|\widehat\rho_i-\Pi_n\rho_i\|_{\mathcal H}^2
  +\|\Pi_n\rho_i-\rho_i\|_{\mathcal H}^2.
\]
Table~\ref{tab:supp_ring1_daily} reports the square roots of the averaged
squared errors in this decomposition. Table~\ref{tab:supp_ring1_mean}
reports the mean-curve errors defined in
Section~\ref{subsec:simulation_recovery}. Those three norms, unlike the
daily squared projection errors, are not additive in levels or squares.

\begin{table}[p]
\centering
\caption{Daily-curve recovery with and without jumps}
\label{tab:supp_ring1_daily}
\small
\begin{tabular}{rrrrrrrr}
\toprule
$T$ & $n$ & $k_n$ & $r_0$ & Jumps & Estimation & Projection & Full RMSE \\
\midrule
66  & 390  & 30  & 0.2 & No  & 0.1737 & 0.0262 & 0.1757 \\
66  & 390  & 30  & 0.2 & Yes & 0.1739 & 0.0262 & 0.1759 \\
66  & 390  & 30  & 0.6 & No  & 0.1209 & 0.0268 & 0.1238 \\
66  & 390  & 30  & 0.6 & Yes & 0.1226 & 0.0268 & 0.1255 \\
132 & 1560 & 60 & 0.2 & No  & 0.1224 & 0.0187 & 0.1238 \\
132 & 1560 & 60 & 0.2 & Yes & 0.1224 & 0.0187 & 0.1239 \\
132 & 1560 & 60 & 0.6 & No  & 0.0842 & 0.0189 & 0.0863 \\
132 & 1560 & 60 & 0.6 & Yes & 0.0844 & 0.0189 & 0.0865 \\
228 & 4680 & 104 & 0.2 & No  & 0.0929 & 0.0143 & 0.0939 \\
228 & 4680 & 104 & 0.2 & Yes & 0.0929 & 0.0143 & 0.0939 \\
228 & 4680 & 104 & 0.6 & No  & 0.0635 & 0.0144 & 0.0651 \\
228 & 4680 & 104 & 0.6 & Yes & 0.0636 & 0.0144 & 0.0652 \\
\bottomrule
\end{tabular}
\begin{minipage}{0.94\linewidth}
\footnotesize
\emph{Notes:} $N=2{,}000$ in every row and $\delta=0.10$.  The largest
MCSE among the displayed RMSE components is $1.11\times10^{-4}$. The
Estimation, Projection, and Full columns are obtained by averaging the
corresponding integrated squared errors over replications and trading
days and then taking square roots. At full numerical
precision, the squared Estimation and Projection components sum to the
squared Full RMSE; displayed entries are rounded independently.
\end{minipage}
\end{table}

\begin{table}[p]
\centering
\caption{Recovery of the mean curve with and without jumps}
\label{tab:supp_ring1_mean}
\small
\begin{tabular}{rrrrrrrr}
\toprule
$T$ & $n$ & $r_0$ & Jumps & Sampling & HF estimation & Total & $\sqrt T$ HF error \\
\midrule
66  & 390  & 0.2 & No  & 0.01681 & 0.02162 & 0.02729 & 0.1757 \\
66  & 390  & 0.2 & Yes & 0.01681 & 0.02175 & 0.02735 & 0.1767 \\
66  & 390  & 0.6 & No  & 0.01681 & 0.01726 & 0.02377 & 0.1402 \\
66  & 390  & 0.6 & Yes & 0.01681 & 0.01814 & 0.02423 & 0.1474 \\
132 & 1560 & 0.2 & No  & 0.01186 & 0.01080 & 0.01616 & 0.1241 \\
132 & 1560 & 0.2 & Yes & 0.01186 & 0.01081 & 0.01616 & 0.1242 \\
132 & 1560 & 0.6 & No  & 0.01186 & 0.00855 & 0.01460 & 0.0982 \\
132 & 1560 & 0.6 & Yes & 0.01186 & 0.00866 & 0.01464 & 0.0995 \\
228 & 4680 & 0.2 & No  & 0.00898 & 0.00625 & 0.01110 & 0.0944 \\
228 & 4680 & 0.2 & Yes & 0.00898 & 0.00625 & 0.01110 & 0.0944 \\
228 & 4680 & 0.6 & No  & 0.00898 & 0.00492 & 0.01030 & 0.0743 \\
228 & 4680 & 0.6 & Yes & 0.00898 & 0.00494 & 0.01030 & 0.0746 \\
\bottomrule
\end{tabular}
\begin{minipage}{0.94\linewidth}
\footnotesize
\emph{Notes:} $N=2{,}000$ and $\delta=0.10$. Sampling,
high-frequency estimation error, and total are the three norms defined in
Section~\ref{subsec:simulation_recovery}.
Their largest MCSE is $2.02\times10^{-4}$; the largest MCSE in the scaled
last column is $7.42\times10^{-4}$.
\end{minipage}
\end{table}

For the specified jump intensity and mark distribution, the largest
daily RMSE difference between observations with and without jumps is
$0.0017$, at $(T,n,r_0)=(66,390,0.6)$. The differences are smaller at the
finest configuration.

\subsection{Additional confidence-set results}
\label{suppsec:ring2_additional}

Table~\ref{tab:supp_ring2_scalar} reports coverage and mean lengths for
the three prespecified scalar intervals. Coverage of the unnormalized
full-day interval is below the nominal level in all reported configurations.
Close-minus-early coverage is near nominal in the coarsest low-correlation
configuration and is conservative in the finer joint-growth configurations.
When $\left|\int_0^1\widetilde\rho_u(s)\,ds\right|>\varepsilon_T$,
the normalized estimator integrates to one, so no interval is reported
for its full-day integral.

\begin{table}[p]
\centering
\caption{Coverage and mean length of nominal $95\%$ scalar intervals}
\label{tab:supp_ring2_scalar}
\small
\begin{tabular}{rrrrrrr}
\toprule
$T$ & $n$ & $r_0$ & $N$ & \shortstack[r]{Unnormalized\\full day} & \shortstack[r]{Unnormalized\\close $-$ early} & \shortstack[r]{Normalized\\close $-$ early} \\
\midrule
66  & 390  & 0.2 & 5000 & 0.8946 / 0.06176 & 0.9508 / 0.09452 & 0.9534 / 0.48880 \\
66  & 390  & 0.6 & 5000 & 0.8680 / 0.06055 & 0.9650 / 0.07539 & 0.9672 / 0.12806 \\
132 & 1560 & 0.2 & 2000 & 0.9085 / 0.04386 & 0.9755 / 0.04289 & 0.9750 / 0.21790 \\
132 & 1560 & 0.6 & 2000 & 0.9015 / 0.04381 & 0.9835 / 0.03757 & 0.9855 / 0.06335 \\
228 & 4680 & 0.2 & 2000 & 0.9175 / 0.03401 & 0.9810 / 0.02583 & 0.9795 / 0.13099 \\
228 & 4680 & 0.6 & 2000 & 0.9150 / 0.03407 & 0.9870 / 0.02420 & 0.9850 / 0.04079 \\
\bottomrule
\end{tabular}
\begin{minipage}{0.94\linewidth}
\footnotesize
\emph{Notes:} Each entry is coverage / mean length.  The close-minus-early contrast is the mean over the last 60 minutes minus
the mean over the first 60 minutes, fixed before simulation.  Jumps are
included and $\delta=0.10$.
\end{minipage}
\end{table}

Figure~\ref{fig:supp_coverage_calibration} reports probability--probability
diagnostics for the unnormalized and normalized confidence sets in the main design.
The unnormalized calibration curves lie mostly below the 45-degree line, consistent
with undercoverage. Normalization brings the curves closer to that line,
but deviations remain. The shading measures Monte Carlo uncertainty in
calibration, not uncertainty about the intraday mean curve.

\begin{figure}[!htbp]
\centering
\includegraphics[width=\textwidth]{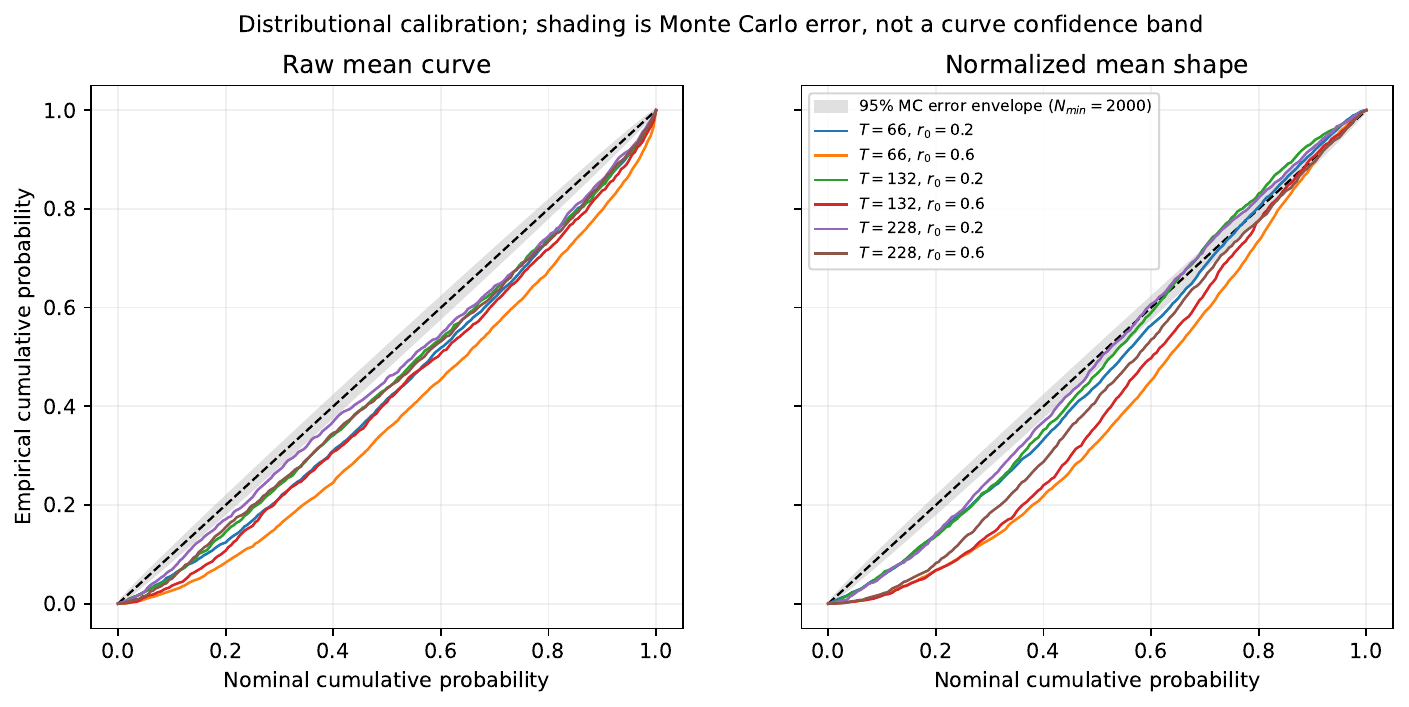}
\caption{Probability--probability diagnostics for unnormalized and normalized
whole-curve confidence sets in the main design. At nominal probability $p$,
the vertical coordinate is the fraction of replications whose conditional
CDF evaluated at the observed statistic is at most $p$. Gray shading is the
pointwise 95\% normal-approximation envelope under uniform calibration for
$N_{\min}=2{,}000$; it is not a confidence band for the mean curve.}
\label{fig:supp_coverage_calibration}
\end{figure}

For the zero-integral design, the mean curve is
$r_u(\kappa)=\beta_{\mathrm{ZI}}s(\kappa)$, where
$s(\kappa)=-\sin(2\pi\kappa)$ and
$\beta_{\mathrm{ZI}}\in\{0,0.06\}$. Normalization is undefined and is not
computed. Table~\ref{tab:supp_ring2_zero} reports similar unnormalized confidence-ball
coverage and radii at the two values of $\beta_{\mathrm{ZI}}$. These are
coverage results; the corresponding unnormalized flatness test is
reported in \suppref{suppsec:ring3_diagnostics}.

\begin{table}[p]
\centering
\caption{Unnormalized confidence sets and scalar intervals in the zero-integral design}
\label{tab:supp_ring2_zero}
\small
\begin{tabular}{rrrrrrr}
\toprule
$T$ & $\beta_{\mathrm{ZI}}$ & $N$ & Ball coverage & Radius & \shortstack[r]{Full day\\cov./len.} & \shortstack[r]{Close $-$ early\\cov./len.} \\
\midrule
66  & 0.00 & 2000 & 0.9170 & 0.04018 & 0.8965 / 0.06203 & 0.9545 / 0.09697 \\
66  & 0.06 & 2000 & 0.9130 & 0.04015 & 0.8960 / 0.06201 & 0.9525 / 0.09692 \\
132 & 0.00 & 2000 & 0.9105 & 0.02544 & 0.9050 / 0.04389 & 0.9720 / 0.04366 \\
132 & 0.06 & 2000 & 0.9055 & 0.02543 & 0.9060 / 0.04388 & 0.9725 / 0.04364 \\
228 & 0.00 & 2000 & 0.9255 & 0.01875 & 0.9245 / 0.03401 & 0.9815 / 0.02608 \\
228 & 0.06 & 2000 & 0.9250 & 0.01875 & 0.9245 / 0.03401 & 0.9815 / 0.02607 \\
\bottomrule
\end{tabular}
\end{table}

\clearpage
\subsection{Flatness diagnostics and sensitivity analysis}
\label{suppsec:ring3_diagnostics}

We compare feasible and latent-curve calculations, alternative HAC
bandwidths, and paired observations with and without price jumps.
The latent calculation replaces the estimated daily curves by block
averages of the latent correlation, retaining the primary bandwidth and
calibration scheme. Its null rejection frequencies are
$0.0274$--$0.0386$, so using latent curves does not remove the conservative
size pattern. Across the alternative bandwidths, null rejection rates are
$0.0198$--$0.0412$. These comparisons do not separately identify the
contributions of Gaussian approximation, covariance estimation, and
discretization.

Table~\ref{tab:supp_ring3_latent} shows a substantial power loss associated
with high-frequency estimation in the coarse, low-correlation
configuration. At $r_0=0.2$, $\delta=0.10$, and $T=66$, feasible and latent
rejection rates are $0.1800$ and $0.7285$. The gap narrows along the
joint-growth design, reaching $0.9980$ versus $1.0000$ at $T=228$.
Because both rates are then close to one, this last comparison describes
agreement in rejection frequencies rather than the magnitude of the
estimation error.

\begin{table}[p]
\centering
\caption{Latent and feasible flatness rejection frequencies}
\label{tab:supp_ring3_latent}
\small
\begin{tabular}{rrrrrr}
\toprule
$T$ & $r_0$ & $\delta$ & $N$ & Feasible & Latent \\
\midrule
66  & 0.2 & 0.00 & 5000 & 0.0348 & 0.0274 \\
66  & 0.2 & 0.10 & 2000 & 0.1800 & 0.7285 \\
66  & 0.6 & 0.00 & 5000 & 0.0304 & 0.0288 \\
66  & 0.6 & 0.10 & 2000 & 1.0000 & 1.0000 \\
132 & 0.2 & 0.00 & 5000 & 0.0282 & 0.0314 \\
132 & 0.2 & 0.10 & 2000 & 0.7935 & 0.9725 \\
132 & 0.6 & 0.00 & 5000 & 0.0288 & 0.0320 \\
132 & 0.6 & 0.10 & 2000 & 1.0000 & 1.0000 \\
228 & 0.2 & 0.00 & 5000 & 0.0328 & 0.0372 \\
228 & 0.2 & 0.10 & 2000 & 0.9980 & 1.0000 \\
228 & 0.6 & 0.00 & 5000 & 0.0364 & 0.0386 \\
228 & 0.6 & 0.10 & 2000 & 1.0000 & 1.0000 \\
\bottomrule
\end{tabular}
\end{table}

The bandwidth sensitivity uses
$\ell_{T,\mathrm{short}}^k=\lfloor T^{1/4}\rfloor$,
$\ell_T^k=\lfloor T^{1/3}\rfloor$, and
$\ell_{T,\mathrm{long}}^k=\lfloor T^{2/5}\rfloor$.
Table~\ref{tab:supp_ring3_bandwidth} shows similar power patterns across
these bandwidths, while null rejection frequencies vary more visibly in
several configurations.

\begin{table}[p]
\centering
\caption{Flatness rejection frequencies across HAC bandwidths}
\label{tab:supp_ring3_bandwidth}
\small
\begin{tabular}{rrrrrrr}
\toprule
$T$ & $r_0$ & $\delta$ & $(\ell_{T,\mathrm{short}}^k,\ell_T^k,\ell_{T,\mathrm{long}}^k)$ & Short & Primary & Long \\
\midrule
66  & 0.2 & 0.00 & $(2,4,5)$ & 0.0320 & 0.0348 & 0.0340 \\
66  & 0.2 & 0.10 & $(2,4,5)$ & 0.1705 & 0.1800 & 0.1820 \\
66  & 0.6 & 0.00 & $(2,4,5)$ & 0.0244 & 0.0304 & 0.0306 \\
66  & 0.6 & 0.10 & $(2,4,5)$ & 1.0000 & 1.0000 & 1.0000 \\
132 & 0.2 & 0.00 & $(3,5,7)$ & 0.0238 & 0.0282 & 0.0284 \\
132 & 0.2 & 0.10 & $(3,5,7)$ & 0.7855 & 0.7935 & 0.7890 \\
132 & 0.6 & 0.00 & $(3,5,7)$ & 0.0200 & 0.0288 & 0.0320 \\
132 & 0.6 & 0.10 & $(3,5,7)$ & 1.0000 & 1.0000 & 1.0000 \\
228 & 0.2 & 0.00 & $(3,6,8)$ & 0.0198 & 0.0328 & 0.0346 \\
228 & 0.2 & 0.10 & $(3,6,8)$ & 0.9975 & 0.9980 & 0.9985 \\
228 & 0.6 & 0.00 & $(3,6,8)$ & 0.0200 & 0.0364 & 0.0412 \\
228 & 0.6 & 0.10 & $(3,6,8)$ & 1.0000 & 1.0000 & 1.0000 \\
\bottomrule
\end{tabular}
\begin{minipage}{0.92\linewidth}
\footnotesize
\emph{Notes:} Null rows use $N=5{,}000$ and alternative rows use
$N=2{,}000$.  The three bandwidth calculations share paths and Gaussian
draws within a replication.
\end{minipage}
\end{table}

The jump comparison in Table~\ref{tab:supp_ring3_jumps} uses $2{,}000$
paired replications. Its jump-on null frequencies use the first $2{,}000$
draws and therefore need not equal the main size frequencies based on
$5{,}000$ replications. With the specified jump intensity and mark
distribution, adding jumps changes the rejection frequency by at most
$0.0035$ in absolute value across the reported configurations.

\begin{table}[p]
\centering
\caption{Paired jump and no-jump flatness rejection frequencies}
\label{tab:supp_ring3_jumps}
\small
\begin{tabular}{rrrrrr}
\toprule
$T$ & $r_0$ & $\delta$ & Jump on & No jump & On minus off \\
\midrule
66  & 0.2 & 0.00 & 0.0340 & 0.0360 & -0.0020 \\
66  & 0.2 & 0.10 & 0.1800 & 0.1765 &  0.0035 \\
66  & 0.6 & 0.00 & 0.0320 & 0.0325 & -0.0005 \\
66  & 0.6 & 0.10 & 1.0000 & 1.0000 &  0.0000 \\
132 & 0.2 & 0.00 & 0.0290 & 0.0305 & -0.0015 \\
132 & 0.2 & 0.10 & 0.7935 & 0.7915 &  0.0020 \\
132 & 0.6 & 0.00 & 0.0280 & 0.0305 & -0.0025 \\
132 & 0.6 & 0.10 & 1.0000 & 1.0000 &  0.0000 \\
228 & 0.2 & 0.00 & 0.0330 & 0.0325 &  0.0005 \\
228 & 0.2 & 0.10 & 0.9980 & 0.9980 &  0.0000 \\
228 & 0.6 & 0.00 & 0.0355 & 0.0360 & -0.0005 \\
228 & 0.6 & 0.10 & 1.0000 & 1.0000 &  0.0000 \\
\bottomrule
\end{tabular}
\begin{minipage}{0.92\linewidth}
\footnotesize
\emph{Notes:} $N=2{,}000$ in every row.  The two versions share all latent
and continuous-price draws and differ only in the inclusion of price jumps.
\end{minipage}
\end{table}

The unnormalized projection test in Remark~\ref{rem:projection} applies to
$r_u(\kappa)=\beta_{\mathrm{ZI}}s(\kappa)$ and tests $P_\perp r_u=0$ without
normalization. Table~\ref{tab:supp_ring3_raw_projected} reports conservative
null rejection frequencies at $\beta_{\mathrm{ZI}}=0$ and high power at
$\beta_{\mathrm{ZI}}=0.06$. Its statistic and critical values are on a
different scale from those of the normalized flatness test.

\begin{table}[p]
\centering
\caption{Unnormalized projection test in the zero-integral design}
\label{tab:supp_ring3_raw_projected}
\small
\begin{tabular}{rrrrr}
\toprule
$T$ & Null $\beta_{\mathrm{ZI}}=0$ & MCSE & Alternative $\beta_{\mathrm{ZI}}=0.06$ & MCSE \\
\midrule
66  & 0.0394 & 0.00275 & 0.9890 & 0.00233 \\
132 & 0.0290 & 0.00237 & 1.0000 & 0.00000 \\
228 & 0.0336 & 0.00255 & 1.0000 & 0.00000 \\
\bottomrule
\end{tabular}
\begin{minipage}{0.92\linewidth}
\footnotesize
\emph{Notes:} The null results use $N=5{,}000$ replications and the
alternative results use $N=2{,}000$.
\end{minipage}
\end{table}

\clearpage
\subsection{Paired covariance-first comparison}
\label{suppsec:chl_simulation}

The reference procedure follows
\citet[Section~5.1]{ChristensenHounyoLiu2026}.  Let $\widehat c_i^\wp$
denote the blockwise step-function extension of
$\widehat c_{i,\kappa_j}^\wp$ for $\wp\in\{X,Y,XY\}$.  The procedure first
forms the covariance-first curve
\[
 \widehat r_{\CF}(\kappa)
 =\frac{T^{-1}\sum_{i=1}^T\widehat c_i^{XY}(\kappa)}
 {\{T^{-1}\sum_{i=1}^T\widehat c_i^X(\kappa)\}^{1/2}
  \{T^{-1}\sum_{i=1}^T\widehat c_i^Y(\kappa)\}^{1/2}}
\]
and uses
\[
 Q_T^{\CF}
 =T\left\|\widehat r_{\CF}
 -\left\{\int_0^1\widehat r_{\CF}(v)\,dv\right\}\1
 \right\|_{\mathcal H}^2.
\]
The HAC estimate uses covariance-coordinate residuals, the
correlation-ratio linearization, and projection off the constant direction.
It gives the covariance and critical value corresponding to the
covariance-first estimator.

The comparison under common numerical settings uses the same paths,
high-frequency blocks, thresholds, Parzen bandwidths, $4{,}999$ Gaussian
draws, and quantile convention as the proposed test. Under this design,
the two procedures test the same flatness restriction. Rejection rates
are not adjusted to a common empirical size. The paired-count columns
record both tests rejecting, the proposed test only, CHL only, and neither.

\begin{figure}[!htbp]
\centering
\includegraphics[width=\textwidth]{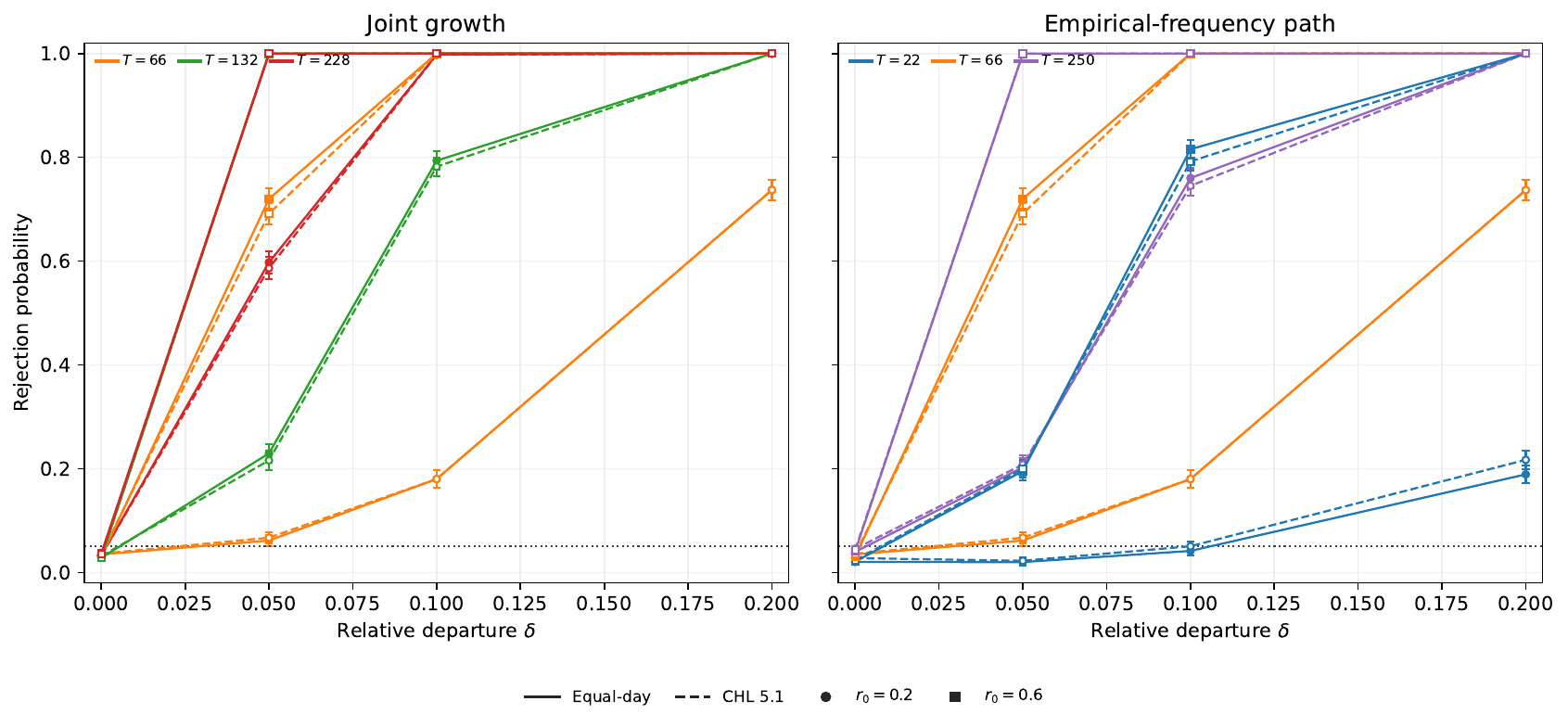}
\caption{Paired rejection frequencies for the proposed and CHL Section 5.1
procedures under common numerical settings. Solid and dashed curves use
the same simulated paths and common Gaussian calibration draws.}
\label{fig:supp_chl_comparison}
\end{figure}

\begin{table}[p]
\centering
\caption{Paired comparison under common numerical settings: null rejection frequencies}
\label{tab:supp_chl_size}
\scriptsize
\resizebox{\textwidth}{!}{%
\begin{tabular}{llrrrrrrrr}
\toprule
Design & $r_0$ & $T$ & Proposed & CHL 5.1 & CHL--proposed (paired MCSE) & Both & Proposed only & CHL only & Neither \\
\midrule
Joint & 0.2 & 66  & 0.0348 & 0.0362 &  0.0014 (0.002) & 124 & 50 & 57 & 4769 \\
Joint & 0.2 & 132 & 0.0282 & 0.0314 &  0.0032 (0.002) & 116 & 25 & 41 & 4818 \\
Joint & 0.2 & 228 & 0.0328 & 0.0326 & -0.0002 (0.002) & 131 & 33 & 32 & 4804 \\
Joint & 0.6 & 66  & 0.0304 & 0.0350 &  0.0046 (0.002) & 108 & 44 & 67 & 4781 \\
Joint & 0.6 & 132 & 0.0288 & 0.0292 &  0.0004 (0.002) & 114 & 30 & 32 & 4824 \\
Joint & 0.6 & 228 & 0.0364 & 0.0356 & -0.0008 (0.002) & 149 & 33 & 29 & 4789 \\
\midrule
Fixed minute & 0.2 & 22  & 0.0202 & 0.0276 & 0.0074 (0.002) & 85 & 16 & 53 & 4846 \\
Fixed minute & 0.2 & 66  & 0.0348 & 0.0362 & 0.0014 (0.002) & 124 & 50 & 57 & 4769 \\
Fixed minute & 0.2 & 250 & 0.0396 & 0.0456 & 0.0060 (0.002) & 161 & 37 & 67 & 4735 \\
Fixed minute & 0.6 & 22  & 0.0208 & 0.0234 & 0.0026 (0.002) & 67 & 37 & 50 & 4846 \\
Fixed minute & 0.6 & 66  & 0.0304 & 0.0350 & 0.0046 (0.002) & 108 & 44 & 67 & 4781 \\
Fixed minute & 0.6 & 250 & 0.0408 & 0.0434 & 0.0026 (0.002) & 147 & 57 & 70 & 4726 \\
\bottomrule
\end{tabular}%
}
\begin{minipage}{0.94\linewidth}
\footnotesize
\emph{Notes:} $N=5{,}000$ in every row.  The $(T,n)=(66,390)$ rows are the
same simulated samples in the two sampling designs.
\end{minipage}
\end{table}

\begin{table}[p]
\centering
\caption{Paired comparison under common numerical settings: joint-growth alternatives}
\label{tab:supp_chl_power_joint}
\scriptsize
\resizebox{\textwidth}{!}{%
\begin{tabular}{rrrrrrrrrr}
\toprule
$r_0$ & $T$ & $\delta$ & Proposed & CHL 5.1 & CHL--proposed (paired MCSE) & Both & Proposed only & CHL only & Neither \\
\midrule
0.2 & 66  & 0.05 & 0.0615 & 0.0670 &  0.0055 (0.004) & 100  & 23  & 34  & 1843 \\
0.2 & 66  & 0.10 & 0.1800 & 0.1800 &  0.0000 (0.005) & 302  & 58  & 58  & 1582 \\
0.2 & 66  & 0.20 & 0.7370 & 0.7370 &  0.0000 (0.006) & 1411 & 63  & 63  & 463 \\
0.2 & 132 & 0.05 & 0.2295 & 0.2160 & -0.0135 (0.005) & 401  & 58  & 31  & 1510 \\
0.2 & 132 & 0.10 & 0.7935 & 0.7820 & -0.0115 (0.005) & 1527 & 60  & 37  & 376 \\
0.2 & 132 & 0.20 & 1.0000 & 1.0000 &  0.0000 (0.000) & 2000 & 0   & 0   & 0 \\
0.2 & 228 & 0.05 & 0.5975 & 0.5865 & -0.0110 (0.005) & 1126 & 69  & 47  & 758 \\
0.2 & 228 & 0.10 & 0.9980 & 0.9975 & -0.0005 (0.001) & 1994 & 2   & 1   & 3 \\
0.2 & 228 & 0.20 & 1.0000 & 1.0000 &  0.0000 (0.000) & 2000 & 0   & 0   & 0 \\
0.6 & 66  & 0.05 & 0.7200 & 0.6915 & -0.0285 (0.007) & 1313 & 127 & 70  & 490 \\
0.6 & 66  & 0.10 & 1.0000 & 0.9995 & -0.0005 (0.001) & 1999 & 1   & 0   & 0 \\
0.6 & 66  & 0.20 & 1.0000 & 1.0000 &  0.0000 (0.000) & 2000 & 0   & 0   & 0 \\
0.6 & 132 & 0.05 & 1.0000 & 0.9995 & -0.0005 (0.001) & 1999 & 1   & 0   & 0 \\
0.6 & 132 & 0.10 & 1.0000 & 1.0000 &  0.0000 (0.000) & 2000 & 0   & 0   & 0 \\
0.6 & 132 & 0.20 & 1.0000 & 1.0000 &  0.0000 (0.000) & 2000 & 0   & 0   & 0 \\
0.6 & 228 & 0.05 & 1.0000 & 1.0000 &  0.0000 (0.000) & 2000 & 0   & 0   & 0 \\
0.6 & 228 & 0.10 & 1.0000 & 1.0000 &  0.0000 (0.000) & 2000 & 0   & 0   & 0 \\
0.6 & 228 & 0.20 & 1.0000 & 1.0000 &  0.0000 (0.000) & 2000 & 0   & 0   & 0 \\
\bottomrule
\end{tabular}%
}
\begin{minipage}{0.94\linewidth}
\footnotesize
\emph{Notes:} $N=2{,}000$ in every row.  Rejection rates are not adjusted
to a common empirical size.
\end{minipage}
\end{table}

\begin{table}[p]
\centering
\caption{Paired comparison under common numerical settings: one-minute alternatives}
\label{tab:supp_chl_power_fixed}
\scriptsize
\resizebox{\textwidth}{!}{%
\begin{tabular}{rrrrrrrrrr}
\toprule
$r_0$ & $T$ & $\delta$ & Proposed & CHL 5.1 & CHL--proposed (paired MCSE) & Both & Proposed only & CHL only & Neither \\
\midrule
0.2 & 22  & 0.05 & 0.0200 & 0.0225 &  0.0025 (0.003) & 30   & 10  & 15  & 1945 \\
0.2 & 22  & 0.10 & 0.0415 & 0.0505 &  0.0090 (0.003) & 69   & 14  & 32  & 1885 \\
0.2 & 22  & 0.20 & 0.1890 & 0.2170 &  0.0280 (0.006) & 330  & 48  & 104 & 1518 \\
0.2 & 66  & 0.05 & 0.0615 & 0.0670 &  0.0055 (0.004) & 100  & 23  & 34  & 1843 \\
0.2 & 66  & 0.10 & 0.1800 & 0.1800 &  0.0000 (0.005) & 302  & 58  & 58  & 1582 \\
0.2 & 66  & 0.20 & 0.7370 & 0.7370 &  0.0000 (0.006) & 1411 & 63  & 63  & 463 \\
0.2 & 250 & 0.05 & 0.2025 & 0.2080 &  0.0055 (0.005) & 350  & 55  & 66  & 1529 \\
0.2 & 250 & 0.10 & 0.7600 & 0.7450 & -0.0150 (0.006) & 1434 & 86  & 56  & 424 \\
0.2 & 250 & 0.20 & 1.0000 & 1.0000 &  0.0000 (0.000) & 2000 & 0   & 0   & 0 \\
0.6 & 22  & 0.05 & 0.1955 & 0.2000 &  0.0045 (0.006) & 316  & 75  & 84  & 1525 \\
0.6 & 22  & 0.10 & 0.8155 & 0.7920 & -0.0235 (0.006) & 1530 & 101 & 54  & 315 \\
0.6 & 22  & 0.20 & 1.0000 & 1.0000 &  0.0000 (0.000) & 2000 & 0   & 0   & 0 \\
0.6 & 66  & 0.05 & 0.7200 & 0.6915 & -0.0285 (0.007) & 1313 & 127 & 70  & 490 \\
0.6 & 66  & 0.10 & 1.0000 & 0.9995 & -0.0005 (0.001) & 1999 & 1   & 0   & 0 \\
0.6 & 66  & 0.20 & 1.0000 & 1.0000 &  0.0000 (0.000) & 2000 & 0   & 0   & 0 \\
0.6 & 250 & 0.05 & 1.0000 & 0.9995 & -0.0005 (0.001) & 1999 & 1   & 0   & 0 \\
0.6 & 250 & 0.10 & 1.0000 & 1.0000 &  0.0000 (0.000) & 2000 & 0   & 0   & 0 \\
0.6 & 250 & 0.20 & 1.0000 & 1.0000 &  0.0000 (0.000) & 2000 & 0   & 0   & 0 \\
\bottomrule
\end{tabular}%
}
\begin{minipage}{0.94\linewidth}
\footnotesize
\emph{Notes:} $N=2{,}000$ in every row.  The $T=66$ rows repeat the same
simulated samples as in Table~\ref{tab:supp_chl_power_joint}.
\end{minipage}
\end{table}

Equal rejection frequencies need not mean identical decisions on the same
samples. At $(T,n,r_0,\delta)=(66,390,0.2,0.10)$, there are 302 common
rejections and 58 rejections unique to each method. Both marginal rates
are $0.1800$, but $5.8\%$ of paired decisions differ. Across the design,
rejection frequencies are generally close, with differences in either
direction.

We next vary the numerical settings while keeping the data-generating
process and the CHL Section~5.1 statistic fixed. We use threshold
constant $q=5$ ($\tau_X=\tau_Y=5$ in our notation), $9{,}999$ Gaussian
draws, and the block lengths and Parzen bandwidths listed below. The
quantile and $p$-value conventions are unchanged. For the three
joint-growth configurations, the settings are
\[
(T,n,k_n,m_n,\ell_T)
=(66,390,130,3,4),\quad
(132,1560,390,4,5),\quad
(228,4680,936,5,6).
\]
This sensitivity analysis is restricted to the joint-growth
configurations with $\delta\in\{0,0.05\}$.

\begin{table}[p]
\centering
\caption{Sensitivity of CHL rejection frequencies to the numerical settings}
\label{tab:supp_chl_source}
\scriptsize
\resizebox{\textwidth}{!}{%
\begin{tabular}{lrrrrrrrrr}
\toprule
Part & $r_0$ & $T$ & Alternative settings & Common settings & Alternative--common (paired MCSE) & Both & Alternative only & Common only & Neither \\
\midrule
Size & 0.2 & 66  & 0.0506 & 0.0362 & 0.0144 (0.004) & 55  & 198 & 126 & 4621 \\
Size & 0.2 & 132 & 0.0356 & 0.0314 & 0.0042 (0.003) & 71  & 107 & 86  & 4736 \\
Size & 0.2 & 228 & 0.0362 & 0.0326 & 0.0036 (0.002) & 116 & 65  & 47  & 4772 \\
Size & 0.6 & 66  & 0.0424 & 0.0350 & 0.0074 (0.003) & 55  & 157 & 120 & 4668 \\
Size & 0.6 & 132 & 0.0320 & 0.0292 & 0.0028 (0.002) & 89  & 71  & 57  & 4783 \\
Size & 0.6 & 228 & 0.0362 & 0.0356 & 0.0006 (0.002) & 139 & 42  & 39  & 4780 \\
\midrule
Power & 0.2 & 66  & 0.0820 & 0.0670 &  0.0150 (0.007) & 51   & 113 & 83  & 1753 \\
Power & 0.2 & 132 & 0.3005 & 0.2160 &  0.0845 (0.009) & 358  & 243 & 74  & 1325 \\
Power & 0.2 & 228 & 0.6455 & 0.5865 &  0.0590 (0.007) & 1123 & 168 & 50  & 659 \\
Power & 0.6 & 66  & 0.6855 & 0.6915 & -0.0060 (0.010) & 1187 & 184 & 196 & 433 \\
Power & 0.6 & 132 & 1.0000 & 0.9995 &  0.0005 (0.001) & 1999 & 1   & 0   & 0 \\
Power & 0.6 & 228 & 1.0000 & 1.0000 &  0.0000 (0.000) & 2000 & 0   & 0   & 0 \\
\bottomrule
\end{tabular}%
}
\begin{minipage}{0.94\linewidth}
\footnotesize
\emph{Notes:} Size rows have $N=5{,}000$ and power rows, all at
$\delta=0.05$, have $N=2{,}000$. Alternative and Common refer to the two
prespecified numerical settings. Alternative--common is the paired
rejection-rate difference. The four paired-count columns report both,
Alternative only, Common only, and neither.
\end{minipage}
\end{table}

The alternative settings materially affect some low-correlation
configurations. At $(T,r_0,\delta)=(132,0.2,0.05)$, rejection rises from
$0.2160$ to $0.3005$. At the null configuration $(T,r_0)=(66,0.2)$, it
rises from $0.0362$ to $0.0506$. Thus the settings affect both size and
power in this comparison. Because several numerical choices change
together, these results do not isolate the effect of any one choice.
\clearpage

\section{Empirical implementation and additional results}
\label{suppsec:empirical}

\subsection{Data and aggregation}
\label{supsubsec:empirical_implementation}

The input is the balanced panel of precomputed one-minute log returns
described in Section~\ref{subsec:empirical_design}. The same 30 stocks are
used throughout:

\begin{quote}\small
AAPL, MSFT, JPM, BAC, XOM, CVX, JNJ, PFE, WMT, HD, KO, MCD, IBM,
CSCO, INTC, VZ, MRK, DIS, PG, UNH, MMM, GE, BA, CAT, TRV, AXP,
NKE, WBA, T, and ORCL.
\end{quote}

Quality checks require positive within-day variance, absolute one-minute
returns below 0.5, and prespecified limits on stock-specific and
market-wide zero-return runs. The largest observed absolute one-minute
return is 0.144431. The fractions of returns truncated are 0.00011735 for
stocks and 0.00013059 for leave-one-out proxies.

Let $R_{i,t}^{a}$ be the one-minute return of stock $a$ and define its
leave-one-out proxy by
\[
 R_{i,t}^{(-a)}=\frac1{29}\sum_{b\ne a}R_{i,t}^{b}.
\]
Proxies are constructed from the original returns before computing local
scales and truncating returns. The proposed stock--proxy aggregate equally
averages the 30 daily correlation curves for $(R^a,R^{(-a)})$; the
average-pairwise aggregate equally averages the 435 unordered stock-pair
curves. These weights are unchanged across days and quarterly windows.

For the CHL reference, covariance coordinates are averaged across days
within a window for each component pair. The correlation ratio is then
formed for each pair, followed by the fixed-weight average across pairs.
For either collection, let $a$ index a component curve and let
$\widehat\psi_i^{(a)}$ denote its daily linearized curve. The aggregate
HAC estimator uses
\[
 \widehat\psi_i^{\mathrm{agg}}
   =\sum_a w_a\widehat\psi_i^{(a)},
\]
rather than an average of the component HAC matrices. Products of the
aggregate curves include the cross terms
$w_aw_b\widehat\psi_i^{(a)}\otimes\widehat\psi_{i-h}^{(b)}$, including those
with $a\ne b$. The proposed procedure likewise forms the aggregate daily
correlation curve first and uses its corresponding residual curves to
estimate the long-run covariance.

\subsection{Correlation summaries}
\label{supsubsec:empirical_descriptive}

The pooled summaries in Table~\ref{tab:empirical_levels} use all 2,370
trading days. For each quarter, the proposed procedure also gives $95\%$
$L^2$ confidence-ball radii and marginal Gaussian intervals for the
full-day mean and the unnormalized and normalized close-minus-early contrasts. The
scalar intervals are unadjusted; the Holm adjustment below applies to the
global flatness tests, not to these intervals.

Figure~\ref{fig:supp_quarterly_contrasts} shows the quarterly normalized
close-minus-early estimates for both methods, together with the proposed
marginal intervals.

\begin{figure}[!htbp]
  \centering
  \includegraphics[width=\textwidth]{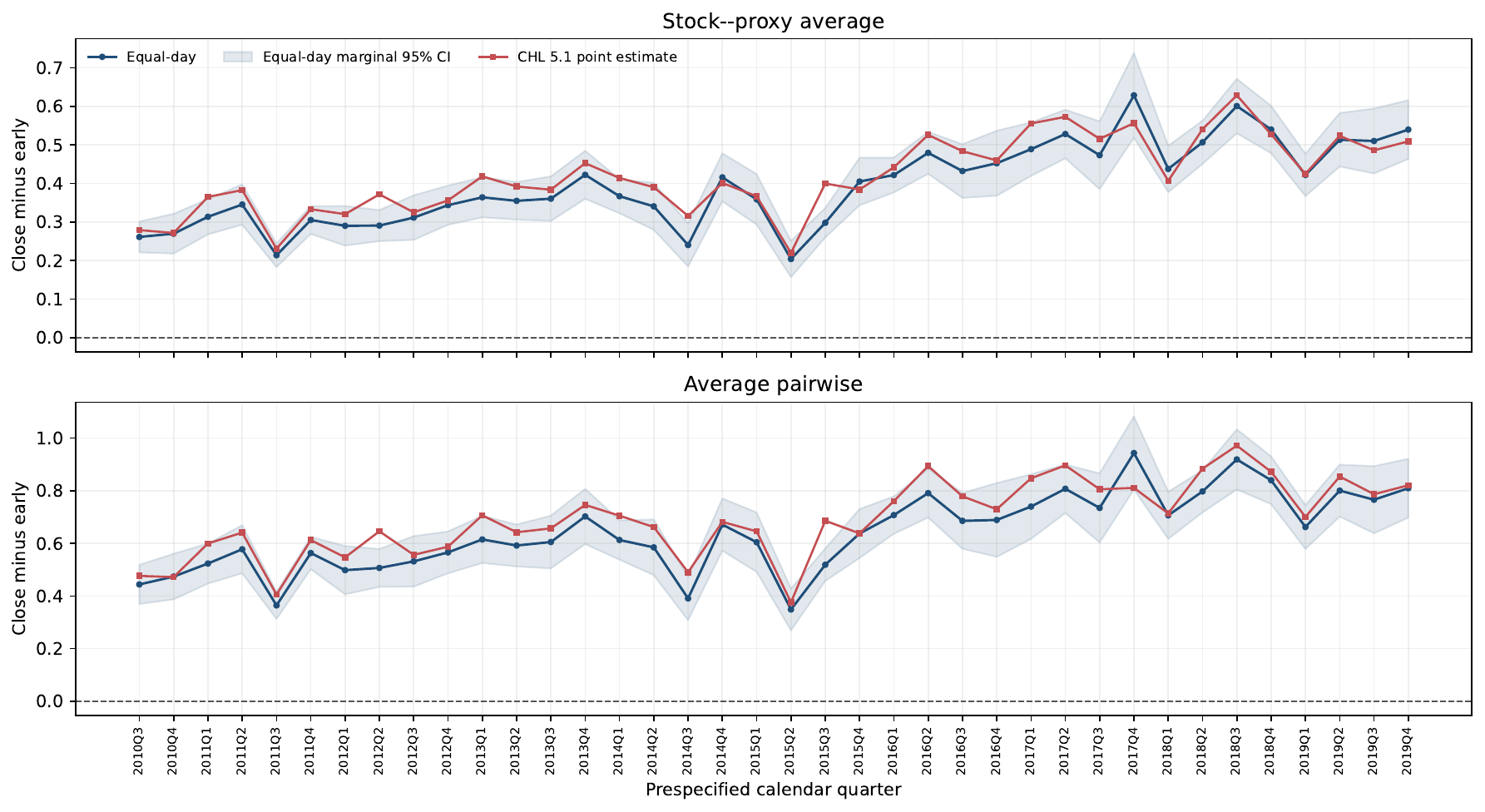}
  \caption{Quarter-specific normalized close-minus-early correlation contrasts.
  Normalization uses each curve's own full-day mean within the quarter.
  Shading gives unadjusted marginal 95\% intervals for the proposed
  estimator. The CHL line reports point estimates.}
  \label{fig:supp_quarterly_contrasts}
\end{figure}

\subsection{Quarterly flatness tests}
\label{supsubsec:empirical_quarterly}

All 38 complete quarters are retained regardless of their estimates or
$p$-values. The bandwidth rule in Section~\ref{subsec:empirical_design}
gives 4 for the seven quarters with 64 trading days and 3 for the other
31 quarters, which contain 60--63 days.

For each aggregate in each quarter, both methods use the same
$B=200{,}000$ vectors of squared standard normal draws, weighted by the eigenvalues
of their respective covariance estimates. For an observed statistic $Q$
and its simulated counterparts $Q_b^*$, the Monte Carlo $p$-value is
\[
 \widehat p=\frac{1+\#\{b:Q_b^*\ge Q\}}{B+1}.
\]
The smallest possible value is $1/200{,}001\simeq0.000005$, attained when
no simulated statistic is at least as large as $Q$. This reporting floor
is not a zero tail probability. The unadjusted $5\%$ test rejects when
$Q$ strictly exceeds the ``higher'' empirical 0.95 quantile. Holm-adjusted
$p$-values are computed separately within each of the four method--object
families, each containing 38 quarterly tests.

\begin{table}[!htbp]
\centering
\caption{Agreement of quarterly flatness decisions}
\label{tab:supp_quarter_agreement}
\small
\begin{tabular}{lrrrr}
\toprule
Object & Both reject & Proposed only & CHL only & Neither\\
\midrule
Stock--proxy average & 38 & 0 & 0 & 0\\
Average pairwise & 37 & 1 & 0 & 0\\
\bottomrule
\end{tabular}
\begin{minipage}{0.94\linewidth}\footnotesize
\emph{Notes:} Counts are identical for unadjusted and Holm-adjusted
$5\%$ decisions. Both methods use the same data and underlying Gaussian
draws, with their own covariance estimates and critical values.
\end{minipage}
\end{table}

Figure~\ref{fig:supp_quarterly_pvalues} compares the unadjusted Monte Carlo
$p$-values. All proposed values reach the reporting floor, while a small
number of CHL values lie above it. Values at the floor cannot be ranked
further at this simulation resolution.

\begin{figure}[!htbp]
  \centering
  \includegraphics[width=0.94\textwidth]{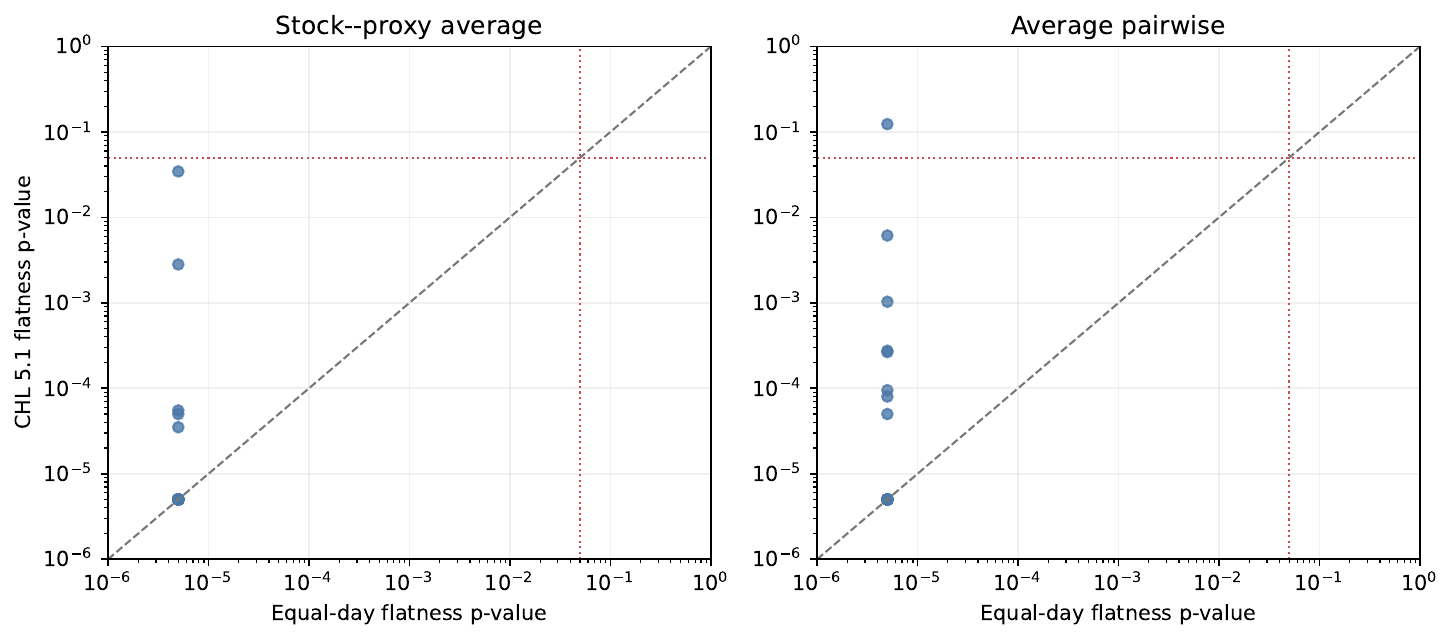}
  \caption{Unadjusted quarter-level flatness $p$-values under the two
  procedures. Dotted lines mark 0.05; the dashed diagonal marks equality
  of the reported $p$-values. Points at approximately $5\times10^{-6}$
  are at the Monte Carlo reporting floor for 200,000 draws.}
  \label{fig:supp_quarterly_pvalues}
\end{figure}

The only decision disagreement, for average-pairwise correlation in
2017Q4 ($T_q=62$, HAC bandwidth 3), is reported in
Section~\ref{subsec:empirical_quarterly}. The proposed statistic, 5.878419,
exceeds its 95\% critical value, 0.517456. Its unadjusted $p$-value is at
the reporting floor and its Holm-adjusted value is 0.000190. The CHL~5.1
statistic, 0.165343, is below its critical value, 0.264980; its unadjusted
and Holm-adjusted $p$-values are both 0.123469. These statistics are
compared with their own critical values, not directly with one another.

The $L^2$ distances between the two estimated curves in this quarter are
0.047253 on the original scale and 0.352334 after normalization. The latter
is the largest normalized distance among the quarterly comparisons.
These are descriptive differences between estimates, not a test of
equality of their population targets.

The following table reports all 76 quarter--object comparisons, with both
methods shown in each row.

\begin{landscape}
\begingroup
\footnotesize
\setlength{\tabcolsep}{4pt}
\begin{longtable}{llrrrrrr}
\caption{Quarter-level flatness probabilities under the equal-day and CHL 5.1 procedures}\label{tab:supp_quarter_details}\\
\toprule
Quarter & Object & $T_q$ & HAC lag & ED $p$ & CHL $p$ & ED Holm $p$ & CHL Holm $p$ \\
\midrule
\endfirsthead
\toprule
Quarter & Object & $T_q$ & HAC lag & ED $p$ & CHL $p$ & ED Holm $p$ & CHL Holm $p$ \\
\midrule
\endhead
2010Q3 & Stock--proxy & 64 & 4 & 0.000005 & 0.000005 & 0.000190 & 0.000190 \\
2010Q3 & Pairwise & 64 & 4 & 0.000005 & 0.000005 & 0.000190 & 0.000190 \\
2010Q4 & Stock--proxy & 63 & 3 & 0.000005 & 0.000035 & 0.000190 & 0.000190 \\
2010Q4 & Pairwise & 63 & 3 & 0.000005 & 0.000275 & 0.000190 & 0.001325 \\
2011Q1 & Stock--proxy & 62 & 3 & 0.000005 & 0.000005 & 0.000190 & 0.000190 \\
2011Q1 & Pairwise & 62 & 3 & 0.000005 & 0.000095 & 0.000190 & 0.000570 \\
2011Q2 & Stock--proxy & 63 & 3 & 0.000005 & 0.000005 & 0.000190 & 0.000190 \\
2011Q2 & Pairwise & 63 & 3 & 0.000005 & 0.000005 & 0.000190 & 0.000190 \\
2011Q3 & Stock--proxy & 64 & 4 & 0.000005 & 0.000005 & 0.000190 & 0.000190 \\
2011Q3 & Pairwise & 64 & 4 & 0.000005 & 0.000005 & 0.000190 & 0.000190 \\
2011Q4 & Stock--proxy & 62 & 3 & 0.000005 & 0.000005 & 0.000190 & 0.000190 \\
2011Q4 & Pairwise & 62 & 3 & 0.000005 & 0.000005 & 0.000190 & 0.000190 \\
2012Q1 & Stock--proxy & 62 & 3 & 0.000005 & 0.000005 & 0.000190 & 0.000190 \\
2012Q1 & Pairwise & 62 & 3 & 0.000005 & 0.000005 & 0.000190 & 0.000190 \\
2012Q2 & Stock--proxy & 63 & 3 & 0.000005 & 0.000005 & 0.000190 & 0.000190 \\
2012Q2 & Pairwise & 63 & 3 & 0.000005 & 0.000005 & 0.000190 & 0.000190 \\
2012Q3 & Stock--proxy & 62 & 3 & 0.000005 & 0.002815 & 0.000190 & 0.005630 \\
2012Q3 & Pairwise & 62 & 3 & 0.000005 & 0.006140 & 0.000190 & 0.012280 \\
2012Q4 & Stock--proxy & 60 & 3 & 0.000005 & 0.000005 & 0.000190 & 0.000190 \\
2012Q4 & Pairwise & 60 & 3 & 0.000005 & 0.000005 & 0.000190 & 0.000190 \\
2013Q1 & Stock--proxy & 60 & 3 & 0.000005 & 0.000005 & 0.000190 & 0.000190 \\
2013Q1 & Pairwise & 60 & 3 & 0.000005 & 0.000005 & 0.000190 & 0.000190 \\
2013Q2 & Stock--proxy & 64 & 4 & 0.000005 & 0.000055 & 0.000190 & 0.000200 \\
2013Q2 & Pairwise & 64 & 4 & 0.000005 & 0.001030 & 0.000190 & 0.003090 \\
2013Q3 & Stock--proxy & 63 & 3 & 0.000005 & 0.000005 & 0.000190 & 0.000190 \\
2013Q3 & Pairwise & 63 & 3 & 0.000005 & 0.000005 & 0.000190 & 0.000190 \\
2013Q4 & Stock--proxy & 62 & 3 & 0.000005 & 0.000005 & 0.000190 & 0.000190 \\
2013Q4 & Pairwise & 62 & 3 & 0.000005 & 0.000005 & 0.000190 & 0.000190 \\
2014Q1 & Stock--proxy & 61 & 3 & 0.000005 & 0.000005 & 0.000190 & 0.000190 \\
2014Q1 & Pairwise & 61 & 3 & 0.000005 & 0.000005 & 0.000190 & 0.000190 \\
2014Q2 & Stock--proxy & 63 & 3 & 0.000005 & 0.000005 & 0.000190 & 0.000190 \\
2014Q2 & Pairwise & 63 & 3 & 0.000005 & 0.000005 & 0.000190 & 0.000190 \\
2014Q3 & Stock--proxy & 63 & 3 & 0.000005 & 0.000005 & 0.000190 & 0.000190 \\
2014Q3 & Pairwise & 63 & 3 & 0.000005 & 0.000050 & 0.000190 & 0.000400 \\
2014Q4 & Stock--proxy & 62 & 3 & 0.000005 & 0.000005 & 0.000190 & 0.000190 \\
2014Q4 & Pairwise & 62 & 3 & 0.000005 & 0.000005 & 0.000190 & 0.000190 \\
2015Q1 & Stock--proxy & 61 & 3 & 0.000005 & 0.000005 & 0.000190 & 0.000190 \\
2015Q1 & Pairwise & 61 & 3 & 0.000005 & 0.000005 & 0.000190 & 0.000190 \\
2015Q2 & Stock--proxy & 63 & 3 & 0.000005 & 0.000005 & 0.000190 & 0.000190 \\
2015Q2 & Pairwise & 63 & 3 & 0.000005 & 0.000080 & 0.000190 & 0.000560 \\
2015Q3 & Stock--proxy & 64 & 4 & 0.000005 & 0.000005 & 0.000190 & 0.000190 \\
2015Q3 & Pairwise & 64 & 4 & 0.000005 & 0.000005 & 0.000190 & 0.000190 \\
2015Q4 & Stock--proxy & 62 & 3 & 0.000005 & 0.000005 & 0.000190 & 0.000190 \\
2015Q4 & Pairwise & 62 & 3 & 0.000005 & 0.000005 & 0.000190 & 0.000190 \\
2016Q1 & Stock--proxy & 61 & 3 & 0.000005 & 0.000005 & 0.000190 & 0.000190 \\
2016Q1 & Pairwise & 61 & 3 & 0.000005 & 0.000005 & 0.000190 & 0.000190 \\
2016Q2 & Stock--proxy & 64 & 4 & 0.000005 & 0.000005 & 0.000190 & 0.000190 \\
2016Q2 & Pairwise & 64 & 4 & 0.000005 & 0.000005 & 0.000190 & 0.000190 \\
2016Q3 & Stock--proxy & 64 & 4 & 0.000005 & 0.000005 & 0.000190 & 0.000190 \\
2016Q3 & Pairwise & 64 & 4 & 0.000005 & 0.000005 & 0.000190 & 0.000190 \\
2016Q4 & Stock--proxy & 62 & 3 & 0.000005 & 0.000005 & 0.000190 & 0.000190 \\
2016Q4 & Pairwise & 62 & 3 & 0.000005 & 0.000005 & 0.000190 & 0.000190 \\
2017Q1 & Stock--proxy & 62 & 3 & 0.000005 & 0.000005 & 0.000190 & 0.000190 \\
2017Q1 & Pairwise & 62 & 3 & 0.000005 & 0.000005 & 0.000190 & 0.000190 \\
2017Q2 & Stock--proxy & 63 & 3 & 0.000005 & 0.000005 & 0.000190 & 0.000190 \\
2017Q2 & Pairwise & 63 & 3 & 0.000005 & 0.000005 & 0.000190 & 0.000190 \\
2017Q3 & Stock--proxy & 62 & 3 & 0.000005 & 0.000005 & 0.000190 & 0.000190 \\
2017Q3 & Pairwise & 62 & 3 & 0.000005 & 0.000005 & 0.000190 & 0.000190 \\
2017Q4 & Stock--proxy & 62 & 3 & 0.000005 & 0.034495 & 0.000190 & 0.034495 \\
2017Q4 & Pairwise & 62 & 3 & 0.000005 & 0.123469 & 0.000190 & 0.123469 \\
2018Q1 & Stock--proxy & 61 & 3 & 0.000005 & 0.000005 & 0.000190 & 0.000190 \\
2018Q1 & Pairwise & 61 & 3 & 0.000005 & 0.000005 & 0.000190 & 0.000190 \\
2018Q2 & Stock--proxy & 64 & 4 & 0.000005 & 0.000005 & 0.000190 & 0.000190 \\
2018Q2 & Pairwise & 64 & 4 & 0.000005 & 0.000005 & 0.000190 & 0.000190 \\
2018Q3 & Stock--proxy & 62 & 3 & 0.000005 & 0.000005 & 0.000190 & 0.000190 \\
2018Q3 & Pairwise & 62 & 3 & 0.000005 & 0.000005 & 0.000190 & 0.000190 \\
2018Q4 & Stock--proxy & 61 & 3 & 0.000005 & 0.000005 & 0.000190 & 0.000190 \\
2018Q4 & Pairwise & 61 & 3 & 0.000005 & 0.000005 & 0.000190 & 0.000190 \\
2019Q1 & Stock--proxy & 61 & 3 & 0.000005 & 0.000005 & 0.000190 & 0.000190 \\
2019Q1 & Pairwise & 61 & 3 & 0.000005 & 0.000005 & 0.000190 & 0.000190 \\
2019Q2 & Stock--proxy & 63 & 3 & 0.000005 & 0.000005 & 0.000190 & 0.000190 \\
2019Q2 & Pairwise & 63 & 3 & 0.000005 & 0.000005 & 0.000190 & 0.000190 \\
2019Q3 & Stock--proxy & 63 & 3 & 0.000005 & 0.000005 & 0.000190 & 0.000190 \\
2019Q3 & Pairwise & 63 & 3 & 0.000005 & 0.000005 & 0.000190 & 0.000190 \\
2019Q4 & Stock--proxy & 62 & 3 & 0.000005 & 0.000050 & 0.000190 & 0.000200 \\
2019Q4 & Pairwise & 62 & 3 & 0.000005 & 0.000265 & 0.000190 & 0.001325 \\
\bottomrule
\end{longtable}

\endgroup
\end{landscape}

\clearpage
\begingroup
\small
\singlespacing
\setlength{\bibsep}{4pt}
\bibliographystyle{plainnat}
\bibliography{references}
\endgroup
\end{document}